\documentclass[12pt, reqno]{amsart}

\usepackage{graphicx}
\usepackage{amssymb,amsmath,amsthm,amscd}
\usepackage{mathrsfs}
\usepackage{enumerate}
\usepackage{enumitem}
\usepackage{verbatim}
\usepackage[usenames,dvipsnames]{color}
\usepackage[colorlinks=true, pdfstartview=FitV,
linkcolor=blue,citecolor=blue,urlcolor=blue]{hyperref}
\usepackage[all]{xy}
\usepackage{tikz}
\usepackage{tikz-cd}
\usetikzlibrary{matrix, arrows}
\usepackage{stmaryrd} 
\usepackage{dsfont}
\usepackage{bm} 
\usepackage{scalerel}[2016/12/29] 

\usetikzlibrary{arrows.meta}
\usetikzlibrary{ decorations.markings}

\allowdisplaybreaks[4]

\newcommand{\nc}{\newcommand}
\newcommand{\rnc}{\renewcommand}
\numberwithin{equation}{section}

\usepackage[colorinlistoftodos]{todonotes}
\usepackage{mathtools}
\mathtoolsset{showonlyrefs,showmanualtags}
\usepackage{empheq}

\theoremstyle{plain}
\newtheorem{lemma}{Lemma}[subsection]
\newtheorem{prop}[lemma]{Proposition}
\newtheorem{theorem}[lemma]{Theorem}
\newcommand{\Prop}{\begin{prop}}
	\newcommand{\enprop}{\end{prop}}
\newcommand{\Lemma}{\begin{lemma}}
	\newcommand{\enlemma}{\end{lemma}}
\newcommand{\Th}{\begin{theorem}}
	\newcommand{\enth}{\end{theorem}}
\newtheorem{corollary}[lemma]{Corollary}
\newcommand{\Cor}{\begin{corollary}}
	\newcommand{\encor}{\end{corollary}}
\newtheorem{definition}[lemma]{Definition}

\newcommand{\Def}{\begin{definition}}
	\newcommand{\edf}{\end{definition}}
\newtheorem{sublemma}[lemma]{Sublemma}
\newcommand{\Sublemma}{\begin{sublemma}}
	\newcommand{\ensub}{\end{sublemma}}

\theoremstyle{definition}
\newtheorem{remark}[lemma]{Remark}
\newtheorem{remarks}[lemma]{Remarks}
\newtheorem{example}[lemma]{Example}
\newtheorem{Convention}[lemma]{Convention}
\newcommand{\Conv}{\begin{Convention}}
	\newcommand{\enconv}{\end{Convention}}
\nc{\Rem}{\begin{remark}}
	\nc{\enrem}{\end{remark}}

\nc{\rmkend}{\hfill$\triangledown$}
\nc{\defend}{\hfill$\triangle$}

\newcommand{\on}{\operatorname}
\nc{\be}{\begin{enumerate}}
	\nc{\ee}{\end{enumerate}}
\newcommand{\eq}{\begin{eqnarray}}
	\newcommand{\eneq}{\end{eqnarray}}
\nc{\bc}{\begin{cases}}
	\nc{\ec}{\end{cases}}
\newcommand{\eqn}{\begin{eqnarray*}}
	\newcommand{\eneqn}{\end{eqnarray*}}
\newcommand{\ba}{\begin{array}}
	\newcommand{\ea}{\end{array}}

\newcommand{\arxiv}[1]{\href{http://arxiv.org/abs/#1}{\tt arXiv:\nolinkurl{#1}}}

\DeclarePairedDelimiter\norm{\lvert}{\rvert} 
 
\DeclarePairedDelimiter\Norm{\lVert}{\rVert}

\nc{\Rp}{\Phi^+}
\nc{\mb}{\underline{m}}
\nc{\nacts}{\triangleleft} 
\nc{\acts}{\mathrel{\raisebox{2pt}{$\scaleobj{0.65}{\varolessthan}$}}}
\nc{\lacts}{\mathrel{\hat{\raisebox{2pt}{$\scaleobj{0.65}{\varolessthan}$}}}}
\nc{\lcirc}{\mathrel{\hat{\circ}}}
\nc{\tdelta}{{}^\theta\Delta} 
\nc{\tnabla}{{}^\theta\nabla}
\nc{\tL}{{}^\theta L}
\nc{\wor}{J^\bullet}
\nc{\twor}{{}^\theta J^\bullet}
\nc{\ch}{\operatorname{ch}}
\nc{\gdim}{\operatorname{dim}_q}
\nc{\word}{J^\theta}
\nc{\good}{J^\beta_+}
\nc{\tgood}{{}^\theta J^\beta_+}
\nc{\gooda}{J^\bullet_+}
\nc{\tgooda}{{}^\theta J^\bullet_+}
\nc{\tgoodb}{{}^\theta J^\bullet_{+,0}}
\nc{\tgoodc}{{}^\theta J^\bullet_{+,c}}
\nc{\tw}{{}^\theta w}
\nc{\EK}{{}^\theta \mathcal{B}(\mathfrak{g})}
\nc{\Qq}{\mathcal{K}}
\nc{\Ql}{\mathcal{A}}
\nc{\EKm}{{}^\theta V(\lambda)}
\nc{\EKmb}{{}^\theta \mathbf{V}(\lambda)} 
\nc{\EKmz}{{}^\theta V}
\nc{\EKmi}{\EKm_{\Ql}} 
\nc{\EKmiz}{\EKmz_{\Ql}} 
\nc{\tcat}{{}^\theta \gamma}
\nc{\Um}{U_q^-(\mathfrak{g})}
\nc{\EKmiu}{\EKmi^{\mathrm{up}}}
\nc{\EKmil}{\EKmi^{\mathrm{low}}}
\nc{\EKmiuz}{\EKmiz^{\mathrm{up}}}
\nc{\EKmilz}{\EKmiz^{\mathrm{low}}}
\nc{\EKup}{\ttt\mathbf{V}_{\Ql}^{\mathrm{up}}}
\nc{\EKlow}{\ttt\mathbf{V}_{\Ql}^{\mathrm{low}}}
\nc{\oklrn}[1]{{}^\theta \mathcal{R}_{#1}(\bm \lambda)}
\nc{\Fn}{F_i^{(n)}}
\nc{\En}{E_i^{(n)}}
\nc{\Fns}{(F_i^*)^{(n)}}
\nc{\Ens}{(E_i^*)^{(n)}}
\nc{\ff}{\mathbf{f}}
\nc{\ffd}{\mathbf{f}^*}
\nc{\fint}{\mathbf{f}_\Ql}
\nc{\fintd}{\mathbf{f}^*_\Ql}
\nc{\FF}{\mathcal{F}}
\nc{\tFF}{{}^\theta\FF(\lambda)}
\nc{\tFFlow}{{}^\theta\FF(\lambda)_{\Ql}^{\mathrm{low}}}
\nc{\tFFup}{{}^\theta\FF(\lambda)_{\Ql}^{\mathrm{up}}}
\nc{\tFFz}{{}^\theta\FF}
\nc{\FFd}{\mathcal{F}^*}
\nc{\tFFd}{{}^\theta\FFd}
\nc{\Qplus}{\mathsf{Q}_+}
\nc{\Qmin}{\mathsf{Q}_-}
\nc{\tQplus}{\Qplus^\theta} 
\nc{\ttt}{{}^\theta}
\nc{\tchq}{\ttt\operatorname{ch}_q}
\nc{\chq}{\operatorname{ch}_q}
\nc{\ddelta}{\bar{\Delta}}
\nc{\tddelta}{\ttt\ddelta}
\nc{\nnabla}{\bar{\nabla}}
\nc{\tnnabla}{\ttt\nnabla}
\nc{\lyn}{\mathcal{L}_+}
\nc{\tlyn}{\ttt\mathcal{L}_+}
\nc{\front}{\operatorname{front}}
\nc{\back}{\operatorname{back}} 
\nc{\minlet}{\operatorname{minlet}}
\nc{\tcan}[1]{\ttt\mathbf{b}_{#1}}
\nc{\tdcan}[1]{\ttt\mathbf{b}^*_{#1}}
\nc{\tpbw}[1]{\ttt \mathbf{P}_{#1}}
\nc{\tdpbw}[1]{\ttt \mathbf{P}_{#1}^*}
\nc{\proot}{\Phi^+}
\nc{\zodd}{\Z_{\text{odd}}} 
\nc{\zeven}{\Z_{\text{even}}} 
\nc{\nodd}{\N_{\text{odd}}} 
\nc{\Seg}{\operatorname{Seg}}

\nc{\Cc}{\mathcal{C}}
\nc{\Mm}{\mathcal{M}}
\nc{\aacts}{\triangleleft}
\nc{\Ccl}{\Cc^{\textrm{loc}}}
\nc{\Mml}{\Mm^{\textrm{loc}}}
\nc{\Ss}{\mathcal{S}}
\nc{\Aa}{\mathcal{A}}
\nc{\Bb}{\mathcal{B}}
\nc{\Tt}{\mathcal{T}}
\nc{\Sseg}{\underline{\Seg}}
\nc{\plyn}[1]{\nu^{\langle #1 \rangle}}

\nc{\cor}{\Bbbk}
\nc{\corr}{\mathbf{k}}

\newcommand{\C}{{\mathbb C}}

\newcommand{\Z}{{\mathbb Z}}
\newcommand{\N}{{\mathbb N}} 

\newcommand{\g}{{\mathfrak{g}}}
\nc{\Ad}{\operatorname{Ad}}

\nc{\sym}{\mathfrak{S}} 
\nc{\weyl}{\mathfrak{W}}
\nc{\ttau}{{}^\theta\tau}
\nc{\tcoset}[2]{\ttt\mathtt{D}_{#1,#2}}
\nc{\coset}[2]{\mathtt{D}_{#1,#2}}

\newcommand{\Hom}{\operatorname{HOM}}
\newcommand{\hhom}{\operatorname{Hom}}
\newcommand{\End}{\operatorname{End}}
\nc{\Aut}{\operatorname{Aut}}

\nc{\coker}{\operatorname{coker}}

\nc{\Img}{\on{Im}}

\nc{\modv}[1]{{#1}\operatorname{-mod}}
\nc{\gmodv}[1]{{#1}\operatorname{-Mod}}
\nc{\Modv}[1]{{#1}\operatorname{-Mod}}
\nc{\gModv}[1]{{#1}\operatorname{-gMod}}
\nc{\pmodv}[1]{{#1}\operatorname{-pMod}}
\nc{\fmodv}[1]{{#1}\operatorname{-fMod}}
\nc{\ffmodv}[1]{{#1}\operatorname{-fmod}}

\nc{\Res}{\on{Res}}
\nc{\res}{\on{res}}

\nc{\triv}{\mathds{1}}
\nc{\ttriv}{{}^\theta\mathds{1}}
\nc{\tind}[1]{\ttriv \varolessthan {#1}}

\nc{\Ob}{\on{Ob}} 

\nc{\klr}{\mathcal{R}( \alpha )}
\nc{\oklr}{{}^\theta \mathcal{R}( \beta; {\bm\lambda} )}
\nc{\oaklr}{{}^\theta \mathcal{R}( \beta )}
\nc{\klrv}[1]{\mathcal{R}({#1})}
\nc{\klrvv}[2]{\mathcal{R}({#1},{#2})}
\nc{\oklrv}[1]{{}^\theta \mathcal{R}({#1};{\bm\lambda})}
\nc{\oklrvv}[2]{{}^\theta \mathcal{R}({#1},{#2};{\bm\lambda})} 
\nc{\klrz}[1]{\mathcal{R}_z({#1})}
\nc{\oklrz}[1]{{}^\theta \mathcal{R}_z({#1};{\bm\lambda})}

\nc{\gklr}[1]{\mathcal{R}_{z}({#1})} 
\nc{\gmklr}[1]{\mathcal{R}_{-z}({#1})} 
\nc{\goklr}[1]{{}^\theta \mathcal{R}_{z}({#1};{\bm\lambda})}

\nc{\comp}{{}^\theta J^\beta}

\nc{\HecB}{\mathcal{H}_{\mathsf{B}_n}(p)}
\nc{\HecBf}{\mathcal{H}^{f}_{\mathsf{B}_n}(p)}

\nc{\bP}{{\mathbb{P}}}
\nc{\bPh}{\widehat{\mathbb{P}}}
\nc{\Oh}{\widehat{\mathcal{O}}} 
\nc{\bK}[1][{n}]{\widehat{\mathbb{K}}_{#1}}

\nc{\tOh}{{}^\theta\widehat{\mathcal{O}}}
\nc{\tOhb}{{}^\theta\widehat{\mathcal{O}}_{\beta}}
\nc{\tbP}{{}^\theta{\mathbb{P}}}
\nc{\tbPh}{{}^\theta\widehat{\mathbb{P}}}
\nc{\tbK}[1][{n}]{{}^\theta\widehat{\mathbb{K}}_{#1}}
\nc{\Kcompl}{{}^\theta\widehat{\mathcal{K}}_\beta} 

\nc{\hV}{\widehat{V}}
\nc{\hVK}{\widetilde{V}}
\nc{\bV}[1][{n}]{\widehat{V}^{\otimes{#1}}}
\nc{\bVK}[1][{n}]{\widetilde{V}^{\otimes{#1}}}

\nc{\tbV}[1][{n}]{{}^\theta\widehat{V}^{\otimes{#1}}}
\nc{\tbVK}[1][{n}]{{}^\theta\widetilde{V}^{\otimes{#1}}}
\nc{\thV}{{}^\theta\widehat{V}}

\nc{\skewr}{{}^\theta\widehat{\mathcal{K}}_\beta \rtimes \cor [ \weyl_n ] }
\nc{\skewrr}{\tbK[\beta] \rtimes \cor [ \weyl_n ] }
\nc{\skewrrr}{\cor [ \weyl_n ] \ltimes \tbK[\beta] }

\nc{\Ulg}{U_q(L\mathfrak{g})}
\nc{\Ugl}{U_q(L\mathfrak{gl}_n)}
\nc{\Usl}{U_q(L\mathfrak{sl}_n)}

\nc{\aff}{{\rm{aff}}}

\nc{\Kmv}[1]{\mathbf{K}_{{#1}}}
\nc{\Rmv}[1]{\mathbf{R}_{{#1}}}
\nc{\Km}{\mathbf{K}}
\nc{\Rm}{\mathbf{R}}

\nc{\tfun}{{}^\theta\mathsf{F}} 
\nc{\fun}{\mathsf{F}} 
\nc{\ttfun}{{}^\theta\tilde{\mathsf{F}}} 

\nc{\bl}{\bigl(}
\nc{\br}{\bigr)}
\nc{\lan}{\langle}
\nc{\ran}{\rangle}

\newcommand{\seteq}{\mathbin{:=}}
\nc{\supp}{\on{supp}}
\newcommand{\soplus}{\mathop{\mbox{\normalsize$\bigoplus$}}\limits}
\newcommand{\id}{\on{id}}
\nc{\ord}{\operatorname{ord}}
\nc{\Id}{\operatorname{Id}}
\nc{\gr}{\operatorname{gr}}

\def\vecsign{\mathchar"017E}
\def\dvecsign{\smash{\stackon[-2.335pt]{\vecsign}{\rotatebox{180}{$\vecsign$}}}}
\def\dvec#1{\def\useanchorwidth{T}\stackon[-4.5pt]{#1}{\,\dvecsign}}
\usepackage{stackengine}
\stackMath

\nc{\epito}{\twoheadrightarrow}
\newcommand{\isoto}[1][]{\mathop{\xrightarrow%
		[{\raisebox{.3ex}[0ex][.3ex]{$\scriptstyle{#1}$}}]%
		{{\raisebox{-.6ex}[0ex][-.6ex]{$\mspace{2mu}\sim\mspace{2mu}$}}}}}

\newlength{\mylength}
\nc{\cA}{{\mathcal A}}
\nc{\cB}{{\mathcal B}}
\nc{\cC}{{\mathcal C}}
\nc{\cD}{{\mathcal D}}
\nc{\cE}{{\mathcal E}}
\nc{\cF}{{\mathcal F}}
\nc{\cG}{{\mathcal G}}
\nc{\cH}{{\mathcal H}}
\nc{\cI}{{\mathcal I}}
\nc{\cJ}{{\mathcal J}}
\nc{\cK}{{\mathcal K}}
\nc{\cL}{{\mathcal L}}
\nc{\cM}{\mathcal{M}}
\nc{\cN}{{\mathcal N}}
\nc{\cO}{{\mathcal O}}
\nc{\cP}{{\mathcal P}}
\nc{\calQ}{{\mathcal Q}}
\nc{\cR}{{\mathcal R}}
\nc{\cS}{\mathcal{S}}
\nc{\cT}{{\mathcal T}}
\nc{\cU}{\mathcal{U}}
\nc{\cV}{{\mathcal V}}
\nc{\cX}{{\mathcal X}}
\nc{\cY}{\mathcal{Y}}
\nc{\cW}{\mathcal{W}}
\nc{\cZ}{{\mathcal Z}}

\nc{\bbA}{{\mathbb{A}}}
\nc{\bbB}{{\mathbb{B}}}
\nc{\bbC}{{\mathbb{C}}}
\nc{\bbD}{{\mathbb{D}}}
\nc{\bbE}{{\mathbb{E}}}
\nc{\bbF}{{\mathbb{F}}}
\nc{\bbG}{{\mathbb{G}}}
\nc{\bbH}{{\mathbb{H}}}
\nc{\bbI}{{\mathbb{I}}}
\nc{\bbJ}{{\mathbb{J}}}
\nc{\bbK}{{\mathbb{K}}}
\nc{\bbL}{{\mathbb{L}}}
\nc{\bbM}{{\mathbb{M}}}
\nc{\bbN}{{\mathbb{N}}}
\nc{\bbO}{{\mathbb{O}}}
\nc{\bbP}{{\mathbb{P}}}
\nc{\bbQ}{{\mathbb{Q}}}
\nc{\bbR}{{\mathbb{R}}}
\nc{\bbS}{{\mathbb{S}}}
\nc{\bbT}{{\mathbb{T}}}
\nc{\bbU}{{\mathbb{U}}}
\nc{\bbV}{{\mathbb{V}}}
\nc{\bbX}{{\mathbb{X}}}
\nc{\bbY}{{\mathbb{Y}}}
\nc{\bbW}{{\mathbb{W}}}
\nc{\bbZ}{{\mathbb{Z}}}

\usepackage{mathrsfs}

\nc{\scrA}{{\mathscr A}}
\nc{\scrB}{{\mathscr B}}
\nc{\scrC}{{\mathscr C}}
\nc{\scrD}{{\mathscr D}}
\nc{\scrE}{{\mathscr E}}
\nc{\scrF}{{\mathscr F}}
\nc{\scrG}{{\mathscr G}}
\nc{\scrH}{{\mathscr H}}
\nc{\scrI}{{\mathscr I}}
\nc{\scrJ}{{\mathscr J}}
\nc{\scrK}{{\mathscr K}}
\nc{\scrL}{{\mathscr L}}
\nc{\scrM}{{\mathscr M}}
\nc{\scrN}{{\mathscr N}}
\nc{\scrO}{{\mathscr O}}
\nc{\scrP}{{\mathscr P}}
\nc{\scrQ}{{\mathscr Q}}
\nc{\scrR}{{\mathscr R}}
\nc{\scrS}{{\mathscr S}}
\nc{\scrT}{{\mathscr T}}
\nc{\scrU}{{\mathscr U}}
\nc{\scrV}{{\mathscr V}}
\nc{\scrX}{{\mathscr X}}
\nc{\scrY}{{\mathscr Y}}
\nc{\scrW}{{\mathscr W}}
\nc{\scrZ}{{\mathscr Z}}

\nc{\sfA}{{\mathsf A}}
\nc{\sfB}{{\mathsf B}}
\nc{\sfC}{{\mathsf C}}
\nc{\sfD}{{\mathsf D}}
\nc{\sfE}{{\mathsf E}}
\nc{\sfF}{{\mathsf F}}
\nc{\sfG}{{\mathsf G}}
\nc{\sfH}{{\mathsf H}}
\nc{\sfI}{{\mathsf I}}
\nc{\sfJ}{{\mathsf J}}
\nc{\sfK}{{\mathsf K}}
\nc{\sfL}{{\mathsf L}}
\nc{\sfM}{{\mathsf M}}
\nc{\sfN}{{\mathsf N}}
\nc{\sfO}{{\mathsf O}}
\nc{\sfP}{{\mathsf P}}
\nc{\sfQ}{{\mathsf Q}}
\nc{\sfR}{{\mathsf R}}
\nc{\sfS}{{\mathsf S}}
\nc{\sfT}{{\mathsf T}}
\nc{\sfU}{{\mathsf U}}
\nc{\sfV}{{\mathsf V}}
\nc{\sfX}{{\mathsf X}}
\nc{\sfY}{{\mathsf Y}}
\nc{\sfW}{{\mathsf W}}
\nc{\sfZ}{{\mathsf Z}}

\nc{\sfa}{{\mathsf a}}
\nc{\sfb}{{\mathsf b}}
\nc{\sfc}{{\mathsf c}}
\nc{\sfd}{{\mathsf d}}
\nc{\sfe}{{\mathsf e}}
\nc{\sff}{{\mathsf f}}
\nc{\sfg}{{\mathsf g}}
\nc{\sfh}{{\mathsf h}}
\nc{\sfi}{{\mathsf i}}
\nc{\sfj}{{\mathsf j}}
\nc{\sfk}{{\mathsf k}}
\nc{\sfl}{{\mathsf l}}
\nc{\sfm}{{\mathsf m}}
\nc{\sfn}{{\mathsf n}}
\nc{\sfo}{{\mathsf o}}
\nc{\sfp}{{\mathsf p}}
\nc{\sfq}{{\mathsf q}}
\nc{\sfr}{{\mathsf r}}
\nc{\sfs}{{\mathsf s}}
\nc{\sft}{{\mathsf t}}
\nc{\sfu}{{\mathsf u}}
\nc{\sfv}{{\mathsf v}}
\nc{\sfx}{{\mathsf x}}
\nc{\sfy}{{\mathsf y}}
\nc{\sfw}{{\mathsf w}}
\nc{\sfz}{{\mathsf z}}

\nc {\bfA}{{\mathbf A}}
\nc {\bfB}{{\mathbf B}}
\nc {\bfC}{{\mathbf C}}
\nc {\bfD}{{\mathbf D}}
\nc {\bfE}{{\mathbf E}}
\nc {\bfF}{{\mathbf F}}
\nc {\bfG}{{\mathbf G}}
\nc {\bfH}{{\mathbf H}}
\nc {\bfI}{{\mathbf I}}
\nc {\bfJ}{{\mathbf J}}
\nc {\bfK}{{\mathbf K}}
\nc {\bfL}{{\mathbf L}}
\nc {\bfM}{{\mathbf M}}
\nc {\bfN}{{\mathbf N}}
\nc{\bfO}{{\mathbf O}}
\nc {\bfP}{{\mathbf P}}
\nc {\bfQ}{{\mathbf Q}}
\nc {\bfR}{{\mathbf R}}
\nc {\bfS}{{\mathbf S}}
\nc {\bfT}{{\mathbf T}}
\nc {\bfU}{{\mathbf U}}
\nc {\bfV}{{\mathbf V}}
\nc {\bfX}{{\mathbf X}}
\nc {\bfY}{{\mathbf Y}}
\nc {\bfW}{{\mathbf W}}
\nc {\bfZ}{{\mathbf Z}}

\nc {\fka}{{\mathfrak a}}
\nc {\fkb}{{\mathfrak b}}
\nc {\fkc}{{\mathfrak c}}
\nc {\fkd}{{\mathfrak d}}
\nc {\fke}{{\mathfrak e}}
\nc {\fkf}{{\mathfrak f}}
\nc {\fkg}{{\mathfrak g}}
\nc {\fkh}{{\mathfrak h}}
\nc {\fki}{{\mathfrak i}}
\nc {\fkj}{{\mathfrak j}}
\nc {\fkk}{{\mathfrak k}}
\nc {\fkl}{{\mathfrak l}}
\nc {\fkm}{{\mathfrak m}}
\nc {\fkn}{{\mathfrak n}}
\nc {\fko}{{\mathfrak o}}
\nc {\fkp}{{\mathfrak p}}
\nc {\fkq}{{\mathfrak q}}
\nc {\fkr}{{\mathfrak r}}
\nc {\fks}{{\mathfrak s}}
\nc {\fkt}{{\mathfrak t}}
\nc {\fku}{{\mathfrak u}}
\nc {\fkv}{{\mathfrak v}}
\nc {\fkx}{{\mathfrak x}}
\nc {\fky}{{\mathfrak y}}
\nc {\fkw}{{\mathfrak w}}
\nc {\fkz}{{\mathfrak z}}

\nc {\fkA}{{\mathfrak A}}
\nc {\fkB}{{\mathfrak B}}
\nc {\fkC}{{\mathfrak C}}
\nc {\fkD}{{\mathfrak D}}
\nc {\fkE}{{\mathfrak E}}
\nc {\fkF}{{\mathfrak F}}
\nc {\fkG}{{\mathfrak G}}
\nc {\fkH}{{\mathfrak H}}
\nc {\fkI}{{\mathfrak I}}
\nc {\fkJ}{{\mathfrak J}}
\nc {\fkK}{{\mathfrak K}}
\nc {\fkL}{{\mathfrak L}}
\nc {\fkM}{{\mathfrak M}}
\nc {\fkN}{{\mathfrak N}}
\nc {\fkO}{{\mathfrak O}}
\nc {\fkP}{{\mathfrak P}}
\nc {\fkQ}{{\mathfrak Q}}
\nc {\fkR}{{\mathfrak R}}
\nc {\fkS}{{\mathfrak S}}
\nc {\fkT}{{\mathfrak T}}
\nc {\fkU}{{\mathfrak U}}
\nc {\fkV}{{\mathfrak V}}
\nc {\fkX}{{\mathfrak X}}
\nc {\fkY}{{\mathfrak Y}}
\nc {\fkW}{{\mathfrak W}}
\nc {\fkZ}{{\mathfrak Z}}

\rnc{\a}{\fka}
\rnc{\b}{\fkb}
\rnc{\c}{\fkc}
\rnc{\d}{\fkd}
\nc{\e}{\fke}
\nc{\f}{\fkf}
\nc{\h}{\fkh}
\rnc{\i}{\fki}
\rnc{\j}{\fkj}
\rnc{\k}{\fkk}
\rnc{\l}{\fkl}
\nc{\m}{\fkm}
\nc{\n}{\fkn}
\rnc{\o}{\fko}
\nc{\p}{\fkp}
\nc{\q}{\fkq}
\rnc{\r}{\fkr}
\nc{\s}{\fks}
\rnc{\t}{\fkt}
\rnc{\u}{\fku}
\rnc{\v}{\fkv}
\nc{\x}{\fkx}
\nc{\y}{\fky}
\nc{\w}{\fkw}
\nc{\z}{\fkz}

\nc{\rk}{\mathsf{rk}}

\nc {\ul}{\underline}
\nc {\ol}{\overline}
\nc {\wtil}{\widetilde}
\nc {\wh}{\widehat}
\nc{\scs}{\scriptscriptstyle}
\nc{\scsop}{\scriptscriptstyle\operatorname} 

\nc {\ie}{{\emph{i.e.}}, }
\nc {\eg}{{\emph{e.g.}}, }
\nc {\aand}{\qquad\mbox{and}\qquad}

\nc{\drc}[1]{\delta_{#1}}
\nc{\der}{\partial}
\nc{\sfad}{\mathsf{ad}}
\nc{\ten}{\otimes}

\nc{\andreacomment}[1]{\footnote{\color{red}{#1}}}
\nc{\Omit}[1]{}
\nc{\summary}[1]{}

\nc{\voldemort}{Voldemort } 
\nc{\Ons}{\mathbf{O}}
\nc{\aOns}{\mathbf{A}}
\nc{\vOns}{\mathbf{V}}

\nc{\km}{\mathbf{k}}
\nc{\qkm}{\mathfrak{k}}

\nc{\texp}[1]{\wt{s}_{#1}} 
\nc{\Lus}[1]{T_{#1}} 
\nc{\mLus}[1]{\mathbf{T}_{#1}} 

\nc{\fin}{{0}}

\nc{\Ifin}{{I_\fin}}
\nc{\Afin}{{A_\fin}}
\nc{\gfin}{{\g_\fin}}

\rnc{\aff}[1]{\wh{#1}}

\nc{\de}[1]{\epsilon_{#1}} 

\nc{\fIS}{I} 
\nc{\aIS}{\aff{I}} 

\nc{\rt}[1]{\alpha_{#1}} 
\nc{\cort}[1]{\alpha_{#1}^\vee}
\nc{\fwt}[1]{\Lambda_{#1}}
\nc{\fcwt}[1]{\Lambda_{#1}^\vee}

\nc{\cp}[2]{{#1} (#2)} 
\nc{\iip}[2]{\left( #1 , #2\right)} 

\nc{\drv}[1]{\delta_{#1}} 
\nc{\codrv}[1]{d_{#1}} 

\nc{\bsfld}{\Bbbk} 

\nc{\bsF}{\bsfld}

\nc{\central}[1]{c_{#1}}

\nc{\Qlat}{\mathsf{Q}}
\nc{\Qpm}{\Qlat_\pm}
\nc{\Qp}{\Qlat_+}
\nc{\Qm}{\Qlat_-}
\nc{\Qlatv}{\Qlat^\vee}
\nc{\Qvpm}{\Qlat^\vee_{\pm}}
\nc{\Qvp}{\Qlat^\vee_+}
\nc{\Qvm}{\Qlat^\vee_-}

\nc{\aQ}{\aff{\Qlat}}
\nc{\aQv}{\aff{\Qlat}^{\vee}}
\nc{\aQp}{\aff{\Qlat}_+}
\nc{\aQvp}{\aff{\Qlat}^{\vee}_+}
\nc{\aQvext}{\aff{\Qlat}^{\vee}_{{\scsop{ext}}}}
\nc{\aQvextp}{\aff{\Qlat}^{\vee}_{{\scsop{ext}},+}}

\nc{\Plat}{\mathsf{P}}
\nc{\Platv}{\mathsf{P}^\vee}
\nc{\Ppm}{\mathsf{P}_\pm}
\nc{\Pp}{\mathsf{P}_+}
\nc{\Pm}{\mathsf{P}_-}
\nc{\Pvpm}{\mathsf{P}^\vee_\pm}
\nc{\Pvp}{\mathsf{P}^\vee_+}
\nc{\Pvm}{\mathsf{P}^\vee_-}

\nc{\aP}{\wh{\Plat}}
\nc{\aPext}{\wh{\Plat}_{{\scsop{ext}}}}
\nc{\aPpm}{\aP_{\pm}}
\nc{\aPextpm}{\aP_{{\scsop{ext}},\pm}}

\nc{\aPd}{\wh{\Plat}_{\drv{}}}

\nc{\PZ}{{\Plat_{\Z}}}
\nc{\PvZ}{{\Platv_{\Z}}}

\nc{\cl}{\mathsf{cl}}

\nc{\Pcl}{\Plat_{\cl}}
\nc{\Pvcl}{\Plat^{\vee}_{\cl}}

\nc{\Pclz}{\Plat_{\cl,0}}

\nc{\fr}{_\mathsf{fr}}
\nc{\Qfr}{\Qlat\fr}
\nc{\Pfr}{\Plat\fr}

\nc{\Qvfr}{\Qlatv\fr}
\nc{\Pvfr}{\Platv\fr}

\nc{\Wfin}{{W_{\fin}}}
\nc{\Waff}{W}
\nc{\Wext}{\widetilde{W}}

\nc{\rfl}[1]{{s_{#1}}}

\nc{\Uqg}{U_q\g}
\nc{\Uqgp}{U^\prime_q\g}

\nc{\Uqb}{U_q{\b}}
\nc{\Uqbp}{U_q{\b}^+}
\nc{\Uqbm}{U_q{\b}^-}
\nc{\Uqbpm}{U_q{\b}^{\pm}}

\nc{\Uqag}{U_q\wh{\g}}
\nc{\Uqagp}{U_q\wh{\g}'}

\nc{\CUqag}[1]{(\Uqag^{#1})^{\cO,\scsop{int}}}

\nc{\Uqan}{U_q\wh{\n}}
\nc{\Uqanp}{U_q\wh{\n}^+}
\nc{\Uqanm}{U_q\wh{\n}^-}
\nc{\Uqanpm}{U_q\wh{\n}^{\pm}}

\nc{\Uqah}{U_q\wh{\h}}
\nc{\Uqab}{U_q\wh{\b}}
\nc{\Uqabp}{U_q\wh{\b}^+}
\nc{\Uqabm}{U_q\wh{\b}^-}
\nc{\Uqabpm}{U_q\wh{\b}^{\pm}}

\nc{\Lg}{L\g} 
\nc{\ag}{\wh{\g}} 
\nc{\agp}{\wh{\g}'} 

\nc{\UqLg}{U_q\Lg}
\nc{\UqLsl}[1]{U_qL\mathfrak{sl}_{#1}}

\nc{\CUqLg}[1]{(\UqLg^{#1})^{\scsop{fd}}} 

\nc{\QL}[1]{U_qL{#1}}

\nc{\Kg}[1]{K_{#1}}
\nc{\Kgpm}[1]{K_{#1}^{\pm}}
\nc{\Kgp}[1]{K_{#1}^+}
\nc{\Kgm}[1]{K_{#1}^-}
\nc{\Eg}[1]{E_{#1}}
\nc{\Fg}[1]{F_{#1}}
\nc{\egp}[1]{e^{+}_{#1}}
\nc{\egm}[1]{e^{-}_{#1}}
\nc{\egpm}[1]{e^{\pm}_{#1}}
\nc{\egmp}[1]{e^{\mp}_{#1}}

\nc{\Ce}{\cC}

\nc{\kg}[1]{k_{#1}}
\nc{\kgpm}[1]{k_{#1}^{\pm}}
\nc{\kgp}[1]{k_{#1}^+}
\nc{\kgm}[1]{k_{#1}^-}

\nc{\xpm}[1]{x^{\pm}_{#1}}
\nc{\xmp}[1]{x^{\mp}_{#1}}
\nc{\xp}[1]{x^{+}_{#1}}
\nc{\xm}[1]{x^{-}_{#1}}

\nc{\xz}[1]{\xi_{#1}}

\nc{\phipm}[1]{\phi^{\pm}_{#1}}
\nc{\phip}[1]{\phi^{+}_{#1}}
\nc{\phim}[1]{\phi^{-}_{#1}}

\nc{\axzp}[1]{\Psi^{+}_{#1}}
\nc{\axzm}[1]{\Psi^{-}_{#1}}
\nc{\axzpm}[1]{\Psi^{\pm}_{#1}}
\nc{\axze}[1]{\Psi^{\varepsilon}_{#1}}

\nc{\QLp}[1]{U_qL{#1}^+}
\nc{\QLm}[1]{U_qL{#1}^-}
\nc{\QLpm}[1]{U_qL{#1}^\pm}
\nc{\QLz}[1]{U_q^cL{#1}}

\nc{\chev}{\omega}

\nc{\weightspace}[2]{{{#1}_{#2}}}
\nc{\wsp}[2]{\weightspace{#1}{#2}}
\nc{\wts}[1]{\mathsf{wt}(#1)}
\nc{\hwL}[1]{L(#1)}
\nc{\catO}[3]{\O_{#1}^{#2}{#3}}

\nc{\wgt}[1]{\operatorname{wt}(#1)}

\nc{\frep}[1]{\mathsf{V}_{#1}} 

\nc{\qstr}[2]{\Sigma_{#1,#2}}

\nc{\evrep}[2]{V_{#1}(#2)}
\nc{\qstrep}[2]{V(\qstr{#1}{#2})}
\nc{\qsrep}[1]{V(#1)}

\nc{\shrep}[2]{{#1}(#2)}
\nc{\Lshrep}[2]{\Lfml{#1}{#2}}
\nc{\Pshrep}[2]{\Pfml{#1}{#2}}

\nc{\HL}[1]{\mathcal{C}_{#1}}

\nc{\Oint}{\cO^{\scsop{int}}}
\nc{\Ointp}{\cO^{+,\scsop{int}}}
\nc{\Ointm}{\cO^{-,\scsop{int}}}
\nc{\Ointpm}{\cO^{\pm,\scsop{int}}}

\nc{\Rep}{\operatorname{Rep}}
\nc{\Repfd}{\Rep^{{\scsop{fd}}}}

\nc{\Mod}[1]{{#1}\operatorname{-mod}}
\nc{\Modfd}[1]{{#1}\operatorname{-mod}^{{\scsop{fd}}}}
\nc{\Modfdgr}[1]{{#1}\operatorname{-mod}_{\scsop{gr}}^{\scsop{fd}}}
\nc{\Modgr}[1]{{#1}\operatorname{-mod}_{\scsop{gr}}}

\nc{\shift}[1]{\Sigma_{#1}}
\nc{\shifta}[1]{\chi_{#1}}
\nc{\tshift}[1]{\Sigma^{\tau}_{#1}}
\nc{\shiftm}[2]{{#1}_{#2}}

\nc{\Deltaop}{\Delta^{\scsop{op}}}

\nc{\fml}[2]{{#1}[\negthinspace[#2]\negthinspace]} 
\nc{\Lfml}[2]{{#1}(\negthinspace(#2)\negthinspace)} 
\nc{\Pfml}[2]{{#1}\{\negthinspace\{#2\}\negthinspace\}} 

\nc{\qsl}[1]{U_q{\mathfrak{sl}}_{#1}}
\nc{\qasl}[1]{U_q\wh{\mathfrak{sl}}_{#1}}
\nc{\qlsl}[1]{U_qL\mathfrak{sl}_{#1}}

\nc{\UqL}[1]{U_qL{#1}}

\nc{\brS}[1]{S_{#1}} 
\nc{\brSg}[1]{\widetilde{s}_{#1}} 

\nc{\Br}[1]{\mathscr{B}_{#1}} 

\nc{\qWS}[1]{S_{#1}}
\nc{\LT}[1]{T_{#1}}

\nc{\tcorr}[1]{\xi_{#1}}
\nc{\bt}[1]{\mathcal{S}_{#1}}

\nc{\adt}[1]{\mathcal{T}_{#1}}
\nc{\adbt}[1]{\mathcal{T}_{#1}}

\nc{\Rcorr}[1]{\eta_{#1}}

\nc{\intg}{\mathbf{W}^{\mathsf{int}}}
\nc{\Vect}{\mathsf{Vect}}

\nc{\rank}{\operatorname{rank}}
\nc{\corank}{\operatorname{corank}}

\nc{\gsat}[1]{\mathsf{Sat}(#1)} 
\nc{\auxgsat}[2]{\mathsf{Sat}(#1; #2)} 
\nc{\sat}[1]{{\mathbf{S}}} 
\nc{\tsat}{\vartheta} 
\nc{\zsat}{\zeta} 
\nc{\tsatq}{\tsat_{q}} 
\nc{\zsatq}{\zsat_{q}} 

\nc{\oi}{\mathsf{op}} 

\nc{\Parsetc}{\bm{\Gamma}} 
\nc{\Parsets}{\bm{\Sigma}}

\nc{\Parc}{\bm{\gamma}} 
\nc{\Pars}{\bm{\sigma}}

\nc{\parc}[1]{\bm{\gamma}_{#1}} 
\nc{\pars}[1]{\bm{\sigma}_{#1}}

\nc{\ctheta}{\theta} 

\nc{\qtheta}{\theta_q} 
\nc{\qthetat}{\widetilde{\theta}_q} 

\nc{\Uqk}{U_q\mathfrak{k}}
\nc{\Uqh}{U_q\mathfrak{h}}
\nc{\Uqn}{U_q\mathfrak{n}}
\nc{\Uqnp}{U_q\mathfrak{n}^+}
\nc{\Uqnpm}{U_q\mathfrak{n}^{\pm}}

\nc{\wt}{\widetilde}

\nc{\Ieq}{I_{\rm eq}} 
\nc{\Ins}{I_{\rm ns}} 

\nc{\Idiff}{\bfI_{\rm diff}} 
\nc{\aIeq}{\aIS_{\rm eq}} 
\nc{\aIdiff}{\aIS_{\rm diff}} 
\nc{\aIns}{\aIS_{\rm ns}} 

\nc{\bg}[1]{b_{#1}} 
\nc{\Bg}[1]{B_{#1}} 

\nc{\QK}[1]{\Upsilon_{#1}} 
\nc{\QR}[1]{{\Xi}_{#1}} 
\nc{\RM}[1]{\mathsf{R}_{#1}} 
\nc{\sRM}[2]{{\mathsf{R}}_{#1}(#2)}
\nc{\sRMv}[2]{{\mathsf{R}}^{\vee}_{#1}(#2)}

\nc{\tsRM}[2]{{\mathsf{R}}^{\tau}_{#1}(#2)} 
\nc{\tsRMv}[2]{{\mathsf{R}}^{\tau,\vee}_{#1}(#2)} 

\nc{\rRM}[2]{{\mathbf{R}}_{#1}(#2)} 
\nc{\rRMv}[2]{{\mathbf{R}}^{\vee}_{#1}(#2)} 
\nc{\KM}[1]{\mathsf{K}_{#1}} 
\nc{\sKM}[2]{\mathsf{K}_{#1}(#2)} 
\nc{\rKM}[2]{\mathbf{K}_{#1}(#2)} 

\nc{\auxsat}[1]{\bfT} 

\nc{\hext}[2]{{#1}[\negthinspace[#2]\negthinspace]} 

\nc{\binomb}[2]{\genfrac{[}{]}{0pt}{}{#1}{#2}}

\nc{\rootsys}{\Phi}

\nc{\hrt}{\vartheta} 

\nc{\Gg}{\mathcal{G}}

\nc{\qV}[1]{V(#1)} 
\nc{\qX}[1]{X(#1)} 
\nc{\qJ}{J} 
\nc{\qT}{\theta} 
\nc{\qJp}{{\qJ_+}} 
\nc{\qJm}{{\qJ_-}} 
\nc{\qJT}{{\qJ^{\theta}}} 
\nc{\dR}[1]{d_{#1}}
\nc{\dK}[1]{d_{#1}}
\nc{\qG}{\Gamma} 
\nc{\frv}[1]{{\bm\lambda}{#1}} 
\nc{\zqV}[2]{V(#1)_{#2}} 

\nc{\trKM}[2]{\wt{\mathbf{K}}_{#1}(#2)} 

\AtBeginDocument{%
	\def\MR#1{}
}

\DeclareRobustCommand{\SkipTocEntry}[5]{}

\title[Generalized Schur-Weyl dualities]
{Generalized Schur-Weyl dualities for quantum affine 
	symmetric pairs and orientifold KLR algebras}

\author[A. Appel]{Andrea Appel} 
\address{Dipartimento di Scienze Matematiche, 
	Fisiche e Informatiche, Universit\`a di Parma, 
	and INFN Gruppo Collegato di Parma, 43100 Parma PR, Italy}
\email{\href{mailto:andrea.appel@unipr.it}{andrea.appel@unipr.it}}

\author[T. Prze\'{z}dziecki]{Tomasz Prze\'{z}dziecki}
\address{School of Mathematics, University of Edinburgh, Peter Guthrie Tait Rd, Edinburgh, EH9 3FD, United Kingdom}
\email{\href{mailto:tprzezdz@exseed.ed.ac.uk}{tprzezdz@exseed.ed.ac.uk}}



\thanks{The first author is partially supported by the Programme {FIL 2020} 
	of the University of Parma and co-sponsored by Fondazione Cariparma. 
	The second author is supported by the ERC Starting Grant No. 759967
	{\em Categorified Donaldson-Thomas Theory}, and the EPSRC grant No.\ EP/W022834/1 \emph{Kac-Moody quantum symmetric pairs, KLR algebras and generalized Schur-Weyl duality}.}

\begin{document}
	
	\begin{abstract}
		Let $\g$ be a complex simple Lie algebra and $\UqLg$ the corresponding 
		quantum affine algebra. We construct a functor $\tfun$ between finite-dimensional modules 
		over a quantum symmetric pair subalgebra of affine type $\Uqk\subset\UqLg$ and an {orientifold} 
		KLR algebra arising from a framed quiver with a contravariant involution, 
		providing a {boundary analogue} of the Kang-Kashiwara-Kim-Oh 
		generalized Schur-Weyl duality. With respect to their construction, our combinatorial model is further enriched 
		with the poles of a 
		trigonometric K-matrix intertwining the action of $\Uqk$ on finite-dimensional
		$\UqLg$-modules. By construction, $\tfun$ is naturally compatible with the
		Kang-Kashiwara-Kim-Oh functor in that, while the latter is a functor of monoidal categories, 
		$\tfun$ is a functor of {module} categories. Relying on a suitable isomorphism {\em à la} 
		Brundan-Kleshchev-Rouquier,
		we prove  that $\tfun$ recovers the Schur-Weyl dualities due to
		Fan-Lai-Li-Luo-Wang-Watanabe in quasi-split type $\sf AIII$. 
	\end{abstract}

	\maketitle
	
	\setcounter{tocdepth}{1}
	\tableofcontents
	
	
	\section{Introduction}

	\subsection{}
	In the present paper, we introduce a \emph{boundary analogue} of Kang-Kashiwara-Kim-Oh 
	generalized Schur-Weyl dualities between quantum affine algebras and Khovanov-Lauda-Rouquier 
	(KLR) algebras (also known as quiver Hecke algebras). 
	More precisely, let $\g$ be a complex simple Lie algebra and $\UqLg$ the corresponding 
	quantum affine algebra. Given an affine quantum symmetric pair (QSP) subalgebra $\Uqk\subset\UqLg$, we construct a functor
	\begin{equation}\label{eq:tfun}
		\tfun:\Modfdgr{\oklrn{Q}}\to\Modfd{\Uqk}
	\end{equation}
	where $\oklrn{Q}$ denotes the {\em orientifold} KLR algebra ($o$KLR) associated to a 
	distinguished quiver $Q$ endowed with a contravariant involution $\theta$ and a framing $\bm\lambda$, which depends upon the choice of a suitable family of finite-dimensional
	$\UqLg$-modules.
	By construction, the functor $\tfun$ intertwines the standard module category structures on 
	$\Modfdgr{\oklrn{Q}}$ and $\Modfd{\Uqk}$, and it is expected to yield, under a suitable 
	localization, an equivalence of categories with a boundary analogue of the Hernandez-Leclerc
	category for $\UqLg$. Finally, we prove that 
	$\tfun$ recovers the Schur-Weyl dualities constructed by Fan-Lai-Li-Luo-Wang-Watanabe   between the (quasi-split type $\sf AIII$) QSP algebra $\Uqk$ and the affine Hecke algebra
	of type $\sfC$.\\
	
	While there is a natural parallelism between the results described above and those 
	in \cite{kang-kashiwara-kim-18}, their proof requires several new ideas, which we describe 
	in detail in the rest of this introduction.
	
	\subsection{}
	The classical Schur-Weyl duality is a fundamental symmetry, which allows us to identify the category of finite-dimensional representations of the symmetric group $\mathfrak{S}_{\ell}$ with the subcategory of $\mathfrak{sl}_N$-modules appearing in the decomposition of the ${\ell}$-tensor power of the vector representation $\mathbb{V}\coloneqq\mathbb{C}^N$ of $\mathfrak{sl}_N$ (when $l \leq N$). This symmetry can be realized as a functor
$\mathbb{V}^{\otimes \ell}\otimes_{\mathfrak{S}_{\ell}}\bullet:\Modfd{\mathfrak{S}_{\ell}}\to
	\Modfd{U\mathfrak{sl}_N}$, using the fact that $\mathbb{V}^{\otimes \ell}$ is an $(U\mathfrak{sl}_N,\mathfrak{S}_{\ell})$-bimodule. The quantum analogue of this construction appears as a duality between quantum groups and Hecke algebras. In the affine setting, it is due to Chari-Pressley \cite{Chari-Pressley-96} and amounts to a functor $\Modfd{\hat{H}_{\ell,q^2}}\to\Modfd{U_qL\mathfrak{sl}_N}$ between the affine Hecke algebra $\hat{H}_{\ell,q^2}$ and the 
	quantum loop algebra $U_qL\mathfrak{sl}_N$, arising from their joint action on the $\ell$-tensor product of 
	the affinized vector representation $\mathbb{V}\coloneqq\mathbb{C}(q)^N[z,z^{-1}]$.
	
	\subsection{}
	More recently, in the series of papers \cite{kang-kashiwara-kim-18,kang-kashiwara-kim-15,
		kang-kashiwara-kim-oh-15, kang-kashiwara-kim-oh-16,kang-kashiwara-kim-oh-18}, Kang, Kashiwara, Kim  and Oh
	defined a \emph{generalized} version of Chari and Pressley's Schur-Weyl duality, which goes beyond 
	type $\mathsf{A}$ and is expressed in terms of KLR algebras.
	More precisely, given a complex simple 
	Lie algebra $\g$, the construction depends on a combinatorial datum consisting of a (possibly infinite) set of finite-dimensional $\UqLg$-representations $\qV{i}$, each decorated with a non-zero scalar $\qX{i}\in\mathbb{C}(q)$. By comparing the order of the poles of the trigonometric R-matrices on $\qV{i}\otimes\qV{j}$ at $\qX{j}/\qX{i}$, one obtains a quiver $Q$ for the KLR algebra, which, 
given a dimension vector $\beta$, we denote by $\cR_Q(\beta)$. 

A suitable $(\UqLg, \cR_Q(\beta))$ bimodule $\mathbb{V}^{\beta}$ is constructed as a direct sum of various tensor products of affinized representations $V(i)$. The action of $\cR_Q(\beta)$ on
	$\mathbb{V}^{\beta}$ is given in terms of normalized R-matrices of $\UqLg$. 
The bimodule yields a functor $\Modgr{\cR_Q(\beta)}^{\scsop{fd}}\to \Modfd{\UqLg_Q}$, and taking the sum over all possible dimension vectors, one obtains a monoidal functor
	\[
	\fun \colon \bigoplus_{\beta}\Modgr{\cR_Q(\beta)}^{\scsop{fd}}\to \Modfd{\UqLg}. 
	\]

We emphasize that, \emph{a priori}, $Q$ is not related to the Dynkin diagram of $\g$. However, it was shown in 
\cite{kang-kashiwara-kim-15} that there does exist a combinatorial datum such that the two coincide. 
In that case, $\fun$ induces an equivalence with the Hernandez-Leclerc category $\mathcal{C}_Q\subset\operatorname{Rep_{fd}}(\UqLg_Q)$ \cite{fujita-22, naoi}. 
	
	\subsection{}
	A {\em quantum affine symmetric pair (QSP) subalgebra} is a distinguished coideal subalgebra 
	$\Uqk\subset \UqLg$ (also known in the literature as Letzter-Kolb coideal subalgebras 
	or affine $\iota$quantum groups) \cite{kolb-14}. Building on previous work by Bao-Wang \cite{bao-wang-18b} and Balagovic-Kolb
	\cite{balagovic-kolb-19}, the first author and B. Vlaar proved in \cite{appel-vlaar-20} that QSP subalgebras of arbitrary Kac-Moody type give rise to {\em universal K-matrices}, \ie\, 
	universal solutions of Cherednik's generalized reflection equation. 
	The latter can be thought of as a boundary analogue of the Yang-Baxter equation, since it arises as a consistency condition in the case of particles 
	moving on a half-line \cite{cherednik-84, sklyanin-88} and produces representations 
	of type $\sfB$ braid groups (cylindrical braid groups). 
	As in the case of the universal R-matrix, universal K-matrices formally descend to finite-dimensional $\UqLg$-modules. In the case of irreducible modules, they give rise to \emph{trigonometric} K-matrices, \ie rational solutions of the reflection equation satisfying a unitarity condition, see \cite{appel-vlaar-22}.
	
	\subsection{}
	Our main result is the construction of a boundary analogue of the functor $\fun$ for affine QSP 
	subalgebra. The key idea is to enhance the combinatorial model developed by Kang-Kashiwara-Kim in \cite{kang-kashiwara-kim-18} by further taking into account the poles of the trigonometric K-matrix 
	on the representations $\qV{i}$. Such an enhanced combinatorial datum yields a 
	framed quiver $Q$ equipped with a contravariant involution. 
The framing depends on the choice of two (families of) parameters of the QSP subalgebra $\Uqk$, which appear in the expressions of the poles of the K-matrix. For {\em generic} parameters, we expect that the framing is trivial. 
	
	\subsection{}
	In \cite{VV-HecB}, Varagnolo and Vasserot introduced a modified KLR algebra 
	${}^{\theta}\cR_Q(\beta;{\bm \lambda})$, associated to a framed quiver with an involution.
	In this paper, we consider a mild generalization, which we call an \emph{orientifold} KLR algebra ($o$KLR) \cite{Przez-oklr,PAW-B, PAR}. By considering a completed tensor product of
	the $\UqLg$-modules $\qV{i}$, we construct, in Section~\ref{s: boundary SW}, a $(\Uqk,{}^{\theta}\cR_Q(\beta;{\bm\lambda}))$-
	bimodule ${}^\theta\mathbb{V}^{\beta}$, whose $o$KLR algebra action is given in terms of normalized
	R- and K-matrices. This construction yields a functor $\tfun$, which intertwines
	the respective categorical actions, as illustrated below. 
	\begin{equation}\label{eq:summary}
		\begin{tikzcd}[column sep=1.3cm]
			\displaystyle\bigoplus_{\beta}\Modfdgr{\cR_{Q}(\beta)}
			\arrow[r, "\fun"] \arrow[d]&
			\Modfd{\UqLg} \arrow[d] \\
			\displaystyle\bigoplus_{\beta}\Modfdgr{{}^{\theta}\cR_{Q}(\beta;{\bm\lambda})}
			\arrow[r, "\tfun"'] &
			\Modfd{\Uqk}
		\end{tikzcd}
	\end{equation}
	The vertical arrow  on the KLR side is given by induction, while that on the quantum affine
	side is given by restriction. 

Moreover, we prove in Section~\ref{s:BSW-A} that, in the case of quasi-split QSP subalgebras
	of type $\sf AIII$, the functor $\tfun$ recovers the $\imath/\jmath$Schur-Weyl duality between 
	$\Uqk$ and the $3$-parameter affine Hecke algebra of type $\sfC$ constructed by Fan, Lai, Li, Luo, Wang, and 
	Watanabe in \cite{FLLLWW}. The proof, which eventually reduces to an explicit and direct 
	computation, relies on a modified Brundan-Kleshchev-Rouquier isomorphism constructed in Section~\ref{s: BRK} between 
	(certain completions of) $o$KLR algebras and affine Hecke algebras of type $\sfC$, generalizing
	similar constructions obtained in \cite{VV-HecB, PAW-B}.
	
	\subsection{}
	The functor $\tfun$ does not immediately yield an equivalence, but it is expected
	to give rise to one under a suitable localization procedure, which we briefly outline below.
	
In the case of the functor $\fun$, this procedure involves  
	replacing the category of modules over the KLR algebra with
	a {\em localized quotient} $\cT$ \cite{kang-kashiwara-kim-18}. Namely, one first mods out the kernel of the functor $\fun$ 
	and then localizes the quotient with respect to a family of {\em commuting objects} whose image through $\fun$ is the 
	trivial $\UqLg$-module. The functor $\fun$ factors through $\cT$ and yields an equivalence 
	of monoidal categories $\fun: \cT\to\cC_{\scsop{HL}}$, where $\cC_{\scsop{HL}}$ is a
	Hernandez-Leclerc-type subcategory of $\Modfd{\UqLg}$ \cite{HL-cluster, kang-kashiwara-kim-18, fujita-22,
		naoi}.
	
	A similar strategy can be applied towards the construction of a $\cT$-module category 
	${}^{\theta}\cT$, to which $\tfun$ descends,  
	yielding an analogue of the diagram \eqref{eq:summary}. 
The corresponding localization procedure would involve a family of {\em commuting objects} in a {\em module} category. This requires 
	the construction of suitable K-matrices for $o$KLR algebras, which are  expected to correspond, under the functor $\tfun$, to the trigonometric K-matrices 
	for $\UqLg$-modules constructed in \cite{appel-vlaar-22}. The resulting functor 
	is finally expected to give a rise to an equivalence of module categories between 
	${}^{\theta}\cT$ and a boundary analogue of the Hernandez-Leclerc category, which will be the subject of future work.

		

	\subsection{Future directions}
	Following the seminal work of Bao-Wang \cite{bao-wang-18, bao-wang-18b},  a general picture has begun to emerge that most
	fundamental algebraic, geometric, and categorical constructions in quantum groups can be generalized to quantum symmetric pairs. 
	Fundamental results in this direction have been obtained in \cite{ehrig-stroppel,balagovic-kolb-19}.
	In the same spirit, the finite-dimensional
	representation theory of affine QSP subalgebras is expected to be as rich and interesting as that of quantum 
	affine algebras. Several recent advances, such as the new Drinfeld presentation in the split affine QSP case (see, \eg \cite{lu-wang-drin}), and the study of $\sigma$-quiver varieties \cite{Li}, suggest that the algebraic and geometric methods based, respectively, on Drinfeld polynomials and 
	Nakajima quiver varieties are likely to have their QSP analogues. 
	
	The results of the present paper show that the recent approach based on Kang-Kashiwara-Kim-Oh 
	Schur-Weyl duality also successfully extends to quantum symmetric pairs. While a direct 
	understanding of their representation theory is, in large part, still out of reach, the 
	boundary Schur-Weyl duality offers an alternative method, exploiting the combinatorial nature 
	of KLR representation theory. 

	Moreover, the natural grading on $o$KLR algebras suggests the existence of a \emph{graded} representation theory of  quantum affine symmetric pairs. More specifically, we expect the gradings to manifest themselves through $q$-characters and deformations of Grothendieck modules, compatible with analogous deformations of Grothendieck rings in the case of quantum affine algebras \cite{HL-Hall, Nak3, VVper}. 
On the $o$KLR side, the Grothendieck module of the category of finite-dimensional representations has been described 
in terms of certain irreducible highest weight modules $\ttt V(\bm\lambda)$ over the Enomoto--Kashiwara algebra \cite{Enomoto-Kashiwara-06, Enomoto-Kashiwara-08, VV-HecB}. Following \cite{Kleshchev-Ram-11, Przez-oklr}, we expect that significant new information about the representation theory of quantum affine symmetric pairs 
	may be extracted from the dual canonical basis of $\ttt V(\bm\lambda)$. 
	In rank one, this is expected to be related with the combinatorics of tridiagonal pairs developed by Ito-Terwilliger 
	\cite{Ito-Terwilliger-10}. 
	
	\subsection{Outline}
	In Sections \ref{s:affine} and \ref{s:affine-qsp}, we recall the basic definitions
	of quantum affine algebras and their quantum symmetric pairs (QSP). 
	In Section \ref{s:rational}, we briefly discuss the construction of rational K-matrices for 
	finite-dimensional modules over quantum affine algebras.
	In Section \ref{s:oKLR}, we review the definition
	of orientifold KLR algebras ($o$KLR) associated to a framed quiver with an involution
	and we describe the module structure on their category of finite-dimensional graded modules.
	In Section \ref{s:combinatorics}, we introduce the enhanced combinatorial model and
	discuss few examples.
	In Section \ref{s: boundary SW}, we present the main result of the paper, providing the construction of the
	Schur-Weyl duality functor $\tfun$ (Thm. \ref{thm: bimodule}) and its compatibility with the Kang-Kashiwara-Kim
	functor (Thm. \ref{thm: KKK compatibility}).
	In Section \ref{s: BRK}, we construct an isomorphism {à la} Brundan-Kleshchev-Rouquier (BKR) between 
	suitable completions of $o$KLR algebras and affine Hecke algebras of type $\sfC$ (Thm. \ref{thm: bkr iso C}).
	In Section \ref{s:BSW-A}, we prove that, through the BKR isomorphism, the functor 
	$\tfun$ recovers the $\imath/\jmath$Schur-Weyl dualities between quasi-split QSP subalgebras of type $\sf AIII$ and
	affine Hecke algebras of type $\sfC$ constructed in \cite{FLLLWW} (Thm. \ref{thm:ij-schur}).
	
	\addtocontents{toc}{\SkipTocEntry}
	
	\section*{Acknowledgements} 
	We would like to thank Martina Balagovi\'c, Sachin Gautam, Stefan Kolb, Vidas Regelskis, Catharina Stroppel, Paul Terwilliger, Bart Vlaar, and Curtis Wendlandt for their interest in this work and for useful discussions. We also thank the anonymous reviewer for their insightful suggestions and comments which led to an improved manuscript.

	
	\Omit{
		\subsection{Modules}
		TO BE MODIFIED. If $A$ is a graded algebra, let $\gmodv{A}$ be the category of all graded left $A$-modules, with degree-preserving module homomorphisms as morphisms. If $M, N$ are graded $A$-modules, let $\hhom_A(M,N)_n$ denote the space of all homomorphisms homogeneous of degree $n$, and $\Hom_A(M,N) = \bigoplus_{n \in \Z} \hhom_A(M,N)_n$. Let $M\{n\}$ denote the module obtained from $M$ by shifting the grading by $n$. Let $\pmodv{A}$ denote the full subcategory of finitely generated graded projective modules, and $\fmodv{A}$ the full subcategory of graded finite dimensional modules. Given any of these abelian categories $\mathcal{C}$, we denote its Grothendieck group by $[ \mathcal{C} ]$. 
	}

\section{Quantum affine algebras}\label{s:affine}

\summary{\color{magenta} In this section:
	\begin{itemize}
		\item \ref{ss:affine}: definition of $\Uqag$ and $\UqLg$
		\item \ref{ss:QL-representations}: category $\cO$ and fd representations
		\item \ref{ss:R-matrix}: universal R-matrix
		\item \ref{ss:spectral-R}: spectral R-matrix
		\item \ref{ss:rational-R}: trigonometric R-matrix
	\end{itemize}
}

In this section we recall the definition of quantum affine algebras and basic results on
their irreducible finite-dimensional modules.\\


For any lattice $\mathsf{\Lambda}\subset\bbR^n$, we denote by $\mathsf{\Lambda}_+$
its non--negative component. We regard $q$ as an indeterminate and set $\cor \coloneqq \overline{\C(q)}$.

\subsection{Quantum affine algebras}\label{ss:affine}

Let $\g$ be a complex finite-dimensional simple Lie algebra 
with Cartan subalgebra $\h\subset\g$. Let $\fIS\coloneqq\{1,2,\dots, \rank(\g)\}$ 
be the set of vertices of the corresponding Dynkin diagram, $A=(a_{ij})_{i,j\in\fIS}$ 
the Cartan matrix, $\iip{\cdot}{\cdot}$ the normalized invariant bilinear form on $\g$, 
$\Pi\coloneqq\{\rt{i}\;|\; i\in\fIS\}\subset\h^*$ a basis of simple roots and $\Pi^{\vee}\coloneqq\{\cort{i}\;|\; i\in\fIS\}\subset\h$ a basis of simple 
coroots such that $\cp{\rt{i}}{\cort{i}}=a_{ij}$. Let $\Qlat\coloneqq\bbZ\Pi\subset\h^*$
and $\Qlat^{\vee}\coloneqq\bbZ\Pi^\vee\subset\h$ be the root and coroot lattice, respectively.
Let $\rootsys_+\subset\Qp$ be the set of positive roots. 
Finally, let $\fwt{}\coloneqq\{\fwt{i}\;|\; i\in\fIS\}$ be the set of fundamental weights
and $\Plat\coloneqq\bbZ\fwt{}\subset\h^*$ the weight lattice.\\

Let $\wh{\g}$ be the (untwisted) affine Lie algebra\footnote{The results of this paper apply to the case of twisted affine Lie
algebras. However, in order to simplify the exposition, we consider only the untwisted case.}
 associated to $\g$ 
with affine Cartan subalgebra $\wh{\h}\subset\wh{\g}$ \cite[Ch.~7]{kac-90}.
Let $\aIS\coloneqq\{0\}\cup\fIS$ be the set of vertices of the affine Dynkin diagram 
and $\aff{A}=(a_{ij})_{i,j\in\aIS}$ the extended Cartan matrix \cite[Table Aff. 1]{kac-90}. 
We denote by $\aQ$ and $\aQv$ the affine root and coroot lattices, respectively.
Let $\drv{}\in\aQp$ and $c\in\aQvp$ be the unique elements such that
\begin{align}
	\{\lambda\in\aQ\;|\; \forall i\in\aIS,\, \cp{\lambda}{\cort{i}}=0 \}=\Z\drv{}\aand
	\{h\in\aQv\;|\; \forall i\in\aIS,\, \cp{\rt{i}}{h}=0 \}=\Z \central{}
\end{align}
In particular, $\drv{}-\rt{0}=\sum_{i\in\fIS}a_i\rt{i}\in\rootsys_+$ is the highest root, $c$ is central in $\wh{\g}$ and, under the identification $\nu:\wh{\h}\to\wh{\h}^*$
induced by the bilinear form, one has $\nu(c)=\delta$.
We fix $\codrv{}\in\wh{\h}$ such that $\cp{\rt{i}}{\codrv{}}=\drc{i0}$ for any $i\in\aIS$. 
Note that $\codrv{}$ is uniquely defined up to a summand proportional to $c$ and we obtain
a natural identification $\wh{\h}=\h\oplus\bbC c\oplus\bbC\codrv{}$. Finally, we set
$\aQvext\coloneqq\aQv\oplus\bbZ d\subset\wh{\h}$, $\aP\coloneqq\{\lambda\in{\wh{\h}}^*\;|\;\lambda(\aQvext)\subseteq\bbZ\}$,
and $\aPd\coloneqq\aP/(\aP\cap\bbQ\drv{})$.\\

Fix pairwise coprime non-negative integers $\{\de{i}\;|\; i\in\aIS\}$ such that the matrix $(\de{i}a_{ij})_{i,j\in\aIS}$ is symmetric and set $q_i\coloneqq q^{\de{i}}$.
The quantum Kac--Moody algebra associated to $\wh{\g}$ is the 
algebra $\Uqag$ over $\bsF$ with generators $\Eg{i},\Fg{i}$, $i\in\aIS$, and $\Kg{h}$, $h\in\aQvext$,  subject to the following defining relations. For any $h,h'\in\aQvext$,\footnote{In this paragraph only, we denote the zero element of $\aQvext$ by
$\ul{\mathbf{0}}$ in order to avoid confusion with the element $\Kg{0}=\Kg{\de{0}\cort{0}}$.}
\vspace{0.25cm}
\begin{gather}
	\Kg{h}\Kg{h'}=\Kg{h+h'}\quad\mbox{and}\quad\Kg{\ul{\mathbf{0}}}=1\,,
\end{gather}
and, for any $i,j\in\aIS$,		
\begin{gather}
	\Kg{h}\Eg{i}=q^{\cp{\rt{i}}{h}}\Eg{i}\Kg{h}, \qquad\Kg{h}\Fg{i}=q^{-\cp{\rt{i}}{h}}\Fg{i}\Kg{h}\\
	[\Eg{i},\Fg{j}]=\drc{ij}\frac{\Kg{i}-\Kg{i}^{-1}}{q_i-q_i^{-1}}\\
	\mathsf{Serre}_{ij}(\Eg{i},\Eg{j})=0=\mathsf{Serre}_{ij}(\Fg{i},\Fg{j})\quad(i\neq j)
\end{gather}
where $\Kg{i}^{\pm1}\coloneqq\Kg{\pm \de{i}\cort{i}}$ and $\mathsf{Serre}$ denotes the
usual quantum Serre relations (see, \eg~\cite{lusztig-94}).
On $\Uqag$ we consider the Hopf algebra structure with coproduct
\begin{align}
	\Delta(\Kg{h})=\Kg{h}\ten\Kg{h}, \quad 
	\Delta(\Eg{i})=\Eg{i}\ten1+\Kg{i}\ten\Eg{i}, \quad
	\Delta(\Fg{i})=\Fg{i}\ten\Kg{i}^{-1}+1\ten\Fg{i}
\end{align}
Finally, the Chevalley involution $\chev:\Uqag\to(\Uqag)^{\scsop{op}}$ 
is the isomorphism of Hopf algebras defined by 
\begin{equation}\label{eq:chevalley}
	\chev(\Kg{h})=\Kg{-h},\quad \chev(\Eg{i})=-\Fg{i},\quad \chev(\Fg{i})=-\Eg{i}\,
\end{equation}
for any $i\in\aIS$ and $h\in\aQvext$.\\

We denote by $\Uqanp$ (resp. $\Uqanm$) the subalgebra generated by the elements $\{\Eg{i}\}_{i\in\aIS}$ (resp. $\{\Fg{i}\}_{i\in\aIS}$), 
and we set $\Uqabpm\coloneqq\Uqanpm\Uqah$, where $\Uqah$ is the commutative subalgebra generated by $\Kg{h}$, $h\in\aQvext$. 
Similarly, 	we denote by $\Uqg$ the finite-type subalgebra generated by $\Eg{i}, \Fg{i}$, $i\in\fIS$,
and $\Kg{h}$, $h\in\Qlatv$, and by $\Uqh, \Uqnpm, \Uqbpm\subset\Uqg$ their {finite-type} counterparts.
Finally, we denote by $\Uqagp$ the subalgebra obtained by replacing the extended coroot lattice $\aQvext$ with the standard coroot lattice
$\aQv$.\\

The \emph{quantum loop algebra} $\UqLg$ is the quotient of $\Uqagp$ by the ideal generated 
by $\Kg{c}-1$, where $\Kg{c}=\Kg{0}\prod_{i\in\fIS}\Kg{i}^{a_i}$. Note that the Hopf algebra structure and the 
Chevalley involution descend to $\UqLg$.

\subsection{Finite--dimensional modules}\label{ss:QL-representations}
We assume henceforth that every module is type $\bf{1}$, \ie the action of the generators $\Kg{i}$ ($i\in\aIS$) is semisimple with eigenvalues in $q^{\bbZ}$. Recall that the categories of finite--dimensional modules over $\Uqagp$ and $\UqLg$ are isomorphic, since the action of 
the central element $\Kg{c}$ is necessarily trivial. Note also that finite--dimensional modules admit a weight decomposition over $\aPd$.

It is well--known that irreducible modules are classified by $\rank(\g)$--tuples of monic polynomials in $\bbC[u]$ \cite[Thm.~12.2.6]{chari-pressley}.
The category $\Modfd{\UqLg}$ is monoidal, but it is not semisimple, and it is not braided in the usual sense (see~Section \ref{ss:spectral-R}). We denote by $\CUqLg{}$ the completion of $\UqLg$ with respect to finite--dimensional module, \ie $\CUqLg{}$ is the algebra of endomorphisms of the
forgetful functor from $\Modfd{\UqLg}$ to $\cor$--vector spaces. The monoidal structure induces on $\CUqLg{}$ 
the structure of a cosimplicial algebra (see, \eg \cite[\S 2.10-2.11]{appel-vlaar-20} and \cite[\S 8.9]{appel-toledano-19b}).

\subsection{The spectral R-matrix}\label{ss:spectral-R}
Set $\Uqag[z,z^{-1}]\coloneqq\Uqag\ten\bsF[z,z^{-1}]$ and 
consider the \emph{(homogeneous) grading shift} automorphism
\begin{equation}\label{eq:grading-shift}
	\shift{z}: \Uqag[z,z^{-1}]\to\Uqag[z,z^{-1}]
\end{equation}
given by $\shift{z}(\Kg{h})\coloneqq\Kg{h}$, $\shift{z}(\Eg{i})\coloneqq z^{\drc{i0}}\Eg{i}$,
and $\shift{z}(\Fg{i})\coloneqq z^{-\drc{i0}}\Fg{i}$. Note that, by specializing $z$
in $\bsF^{\times}$, we obtain a one-parameter family of automorphism of $\Uqag$. 
Let $	\Delta_{z}, \Deltaop_{z}:\Uqag[z,z^{-1}]\to(\Uqag\ten\Uqag)[z,z^{-1}]$
be the \emph{shifted coproduct} defined by $\Delta_z(x)\coloneqq\id\ten\shift{z}(\Delta(x))$ and $\Deltaop_z(x)\coloneqq\id\ten\shift{z}(\Deltaop(x))$.
The grading shift naturally descends on $\UqLg$. For any $V\in\Modfd{\UqLg}$
with action $\pi_V:\UqLg\to\End(V)$, we consider the infinite-dimensional modules
$\shrep{V}{z}\coloneqq V\ten\bsF(z)$ and $\Lshrep{V}{z}\coloneqq V\ten\Lfml{\bsF}{z}$
with the natural action given by $\pi_V(\shift{z}(x))$. 

By \cite{drinfeld-quantum-groups-87}, $\UqLg$ has a universal 
\emph{spectral} R-matrix, \ie
a formal series $\sRM{}{z}$, whose coefficients belong to a suitable completions of $\UqLg\ten\UqLg$, 
such that $\shift{a}\ten\shift{b}(\sRM{}{z})=\sRM{}{\frac{b}{a}z}$ ($a,b\in\bsF^{\times}$) and the 
following identities are satisfied:
\begin{align}
	\sRM{}{z}\Delta_{z}(x)=& \ \Deltaop_{z}(x)\sRM{}{z}\, ,\\
	\Delta_z\ten\id(\sRM{}{zw})=& \ \sRM{13}{zw} \sRM{23}{w}\, ,\\
	\id\ten\Delta_w(\sRM{}{z})=& \ \sRM{13}{z} \sRM{12}{zw}\, .
\end{align}
In particular, the Yang--Baxter equation holds:
\begin{equation}\label{eq:z-YBE}
	\sRM{12}{z} \sRM{13}{zw} \sRM{23}{w} = \sRM{23}{w} \sRM{13}{zw} \sRM{12}{z}\, .
\end{equation}
Moreover, for any $V,W\in\Modfd{\UqLg}$, the operator
\begin{equation}\label{eq:sRM-rep}
	\sRM{VW}{z}\coloneqq\pi_{V}\ten\pi_W(\sRM{}{z})
	\in\fml{\End(V\ten W)}{z}
\end{equation}
is well-defined and yields an intertwiner
\begin{equation}\label{eq:sRM-rep}
	\sRMv{VW}{z}\coloneqq(1\,2)\circ\sRM{VW}{z}:V\ten\Lshrep{W}{z}\to \Lshrep{W}{z}\ten V\, .
\end{equation}

\Omit{
\begin{theorem}\label{thm:spectral-R}\hfill
	\begin{enumerate}\itemsep0.25cm
		\item The quantum loop algebra $\UqLg$ has a universal \emph{spectral} R-matrix, \ie
		a formal series $\sRM{}{z}$, whose coefficients belong to a suitable completions of $\UqLg\ten\UqLg$, 
		such that $\shift{a}\ten\shift{b}(\sRM{}{z})=\sRM{}{\frac{b}{a}z}$ ($a,b\in\bsF^{\times}$) and the 
		following identities are satisfied:
		\begin{align}
			\sRM{}{z}\Delta_{z}(x)=& \ \Deltaop_{z}(x)\sRM{}{z}\, ,\\
			\Delta_z\ten\id(\sRM{}{zw})=& \ \sRM{13}{zw} \sRM{23}{w}\, ,\\
			\id\ten\Delta_w(\sRM{}{z})=& \ \sRM{13}{z} \sRM{12}{zw}\, .
		\end{align}
		In particular, the Yang--Baxter equation holds:
		\begin{equation}
			\sRM{12}{z} \sRM{13}{zw} \sRM{23}{w} = \sRM{23}{w} \sRM{13}{zw} \sRM{12}{z}\, .
		\end{equation}
		\item For any $V,W\in\Modfd{\UqLg}$, the operator
		\begin{equation}\label{eq:sRM-rep}
			\sRM{VW}{z}\coloneqq\pi_{V}\ten\pi_W(\sRM{}{z})
			\in\fml{\End(V\ten W)}{z}
		\end{equation}
		is well-defined and yields an intertwiner
		\begin{equation}\label{eq:sRM-rep}
			\sRMv{VW}{z}\coloneqq(1\,2)\circ\sRM{VW}{z}:V\ten\Lshrep{W}{z}\to \Lshrep{W}{z}\ten V\, .
		\end{equation}
	\end{enumerate}
\end{theorem}
}
%


\subsection{The trigonometric R-matrix}\label{ss:rational-R}
In the case of irreducible modules the operator $\sRM{VW}{z}$ is, 
up to a scalar factor, a {\em trigonometric} R-matrix. 

\begin{theorem}[\cite{drinfeld-quantum-groups-87}]\label{thm:rational-R}
	Let $V,W\in\Modfd{\UqLg}$ be two irreducible modules.
	There exists a canonical function $f_{VW}(z)\in\Lfml{\bsF}{z}$ such that
	\begin{equation}
		\rRM{VW}{z}\coloneqq f_{VW}(z)^{-1}\sRM{}{z}\in\Lfml{\End(V\ten W)}{z}
	\end{equation} 
	is a rational non-vanishing operator, satisfying the spectral Yang-Baxter equation \eqref{eq:z-YBE} and the unitarity relation 
	\begin{equation}\label{eq:unitarity-R}
		\rRM{VW}{z}^{-1}=(1\, 2)\circ\rRM{WV}{z^{-1}}\circ(1\, 2)\,.
	\end{equation}
\end{theorem}

The proof of the theorem relies on the \emph{generic irreducibility} of the tensor product
$V\ten W$, \ie on the irreducibility of the module $V\ten\Lfml{W}{z}$ over $\Lfml{\UqLg}{z}$
(cf.~ \cite[\S 4.2]{kazhdan-soibelman-95} or \cite[Thm.~3]{chari-02}). Note also that the function $f_{VW}(z)$ is uniquely determined by the condition $\rRM{}{z}(v\ten w)=v\ten w$, where $v\in V$
and $w\in W$ are highest weight vectors. 

\begin{remarks}\label{rmk:rational-R}
\hfill
\begin{enumerate}\itemsep0.25cm
\item 
As before, this yields a rational intertwiner 
\[\rRMv{VW}{z}\coloneqq(1\,2)\circ\rRM{VW}{z}:V\ten\shrep{W}{z}\to \shrep{W}{z}\ten V\,,\] 
which is sometimes referred to
as the \emph{normalized} R-matrix, see \eg \cite{kang-kashiwara-kim-18}.
\item 
If  $V$ is an irreducible \emph{real} module, \ie $V\ten V$ is irreducible, 
the normalized R--matrix satisfies $\rRMv{VV}{1}=\id$ (see, \eg \cite{frenkel-hernandez-reshetikhin-21}).
\item By considering the shifted module $\shrep{V}{w}$, one obtains an intertwiner
\[\rRMv{VW}{w,z}\coloneqq \rRMv{\shrep{V}{w}W}{z}:\shrep{V}{w}\ten\shrep{W}{z}\to \shrep{W}{z}\ten \shrep{V}{w}\,.\]
Since the spectral R--matrix satisfies the identity $\shift{a}\ten\id(\sRM{}{z})=\sRM{}{\frac{z}{a}}$ ($a\in\bsF^{\times}$), 
it follows that $\rRMv{VW}{w,z}$ depends rationally on $z/w$ only, and it is
denoted $\rRMv{VW}{z/w}$. 
\rmkend
\end{enumerate}
\end{remarks}

\Omit{
Recall that an irreducible $V\in\Modfd{\UqLg}$ is \emph{real} if $V\ten V$ is irreducible. For example, fundamental modules and Kirillov-Reshetikhin modules are real, see \cite{kashiwara-02, chari-hernandez-10, frenkel-hernandez-reshetikhin-21}. 
From the unitarity and non-vanishing conditions, one obtains the following corollary.

[see, \eg \cite[Lem.~10.3]{frenkel-hernandez-reshetikhin-21}]

\begin{corollary}\label{cor:R-regularity-at-1}
	Let $V,W\in\Modfd{\UqLg}$ be irreducible modules. 
	For any $\zeta\in\bsF^\times$ such that $V\ten\shrep{W}{\zeta}$ is irreducible, 
	$\rRM{VW}{\zeta}$ is well-defined. In particular, if $V$ is real, 
	$\rRMv{VV}{1}=\id$.
\end{corollary}

\begin{example}\label{ex:R-matrix-sl2}
	Let $\g=\mathfrak{sl}(2)$, $V_1=\bbC^2$ the fundamental module, and 
	$\evrep{1}{a}$ the corresponding evaluation module of $U_qL\mathfrak{sl}(2)$ at 
	$a\in\bbC^\times$ (see \eg \cite{chari-pressley}). In the case of $\evrep{1}{a}\ten\evrep{1}{bz}$, the rational function $\rRM{a,b}{z}\coloneqq\rRM{\evrep{1}{a}\,\evrep{1}{b}}{z}$ is easily computed (see \eg
	\cite[12.5.7]{chari-pressley}). Set $\lambda=b/a$. Then,
	\begin{equation}
		\rRM{a,b}{z}\coloneqq
		\left[
		\begin{array}{cccc}
			1 & 0 & 0 & 0\\
			0 & \frac{q(1-\lambda z)}{q^2-\lambda z} & \frac{\lambda z(q^2-1)}{q^2-\lambda z} & 0\\
			0 & \frac{q^2-1}{q^2-\lambda z} & \frac{q(1-\lambda z)}{q^2-\lambda z} & 0\\
			0 & 0 & 0 & 1
		\end{array}
		\right]
	\end{equation}
	Note that, at $z=1$, the R-matrix has a pole if $\lambda=q^{2}$, while it is not full rank if $\lambda=q^{-2}$. 
	It is well-known that precisely in these two cases the module $\evrep{1}{a}\ten\evrep{1}{b}$ fails to be
	irreducible. 
	\rmkend
\end{example}
}


\section{Quantum affine symmetric pairs}\label{s:affine-qsp}

\summary{In this section:
	\begin{itemize}
		\item \ref{ss:gsat}: generalized Satake diagrams
		\item \ref{ss:pseudo-involutions}: pseudo involutions
		\item \ref{ss:qtheta}: quantum pseudo-involutions
		\item \ref{ss:qsp}: quantum symmetric pairs
		\item \ref{ss:univ-kmx}: universal $k$--matrices
	\end{itemize}
}

\subsection{Generalized Satake diagrams}\label{ss:gsat}
Classical and quantum Kac-Moody algebras are defined in terms of 
a combinatorial datum encoded by the Dynkin diagram and the Cartan matrix. 
Similarly, classical and quantum symmetric pairs 
(the latter are also known as {\em Letzter-Kolb coideal subalgebras} or {\em $\iota$quantum groups}) 
arise from a refinement of such a datum, see, \eg  \cite{kac-wang-92,letzter-02,kolb-14}.\\

Let $\Aut(\aff{A})$ be the group of \emph{diagram automorphisms} of the affine
Cartan matrix, \ie the group of bijections $\tau:\aIS\to\aIS$ such that
$a_{ij}=a_{\tau(i)\tau(j)}$.
Let $X \subset\aIS$ be a proper subset of indices. Note, in particular, that the 
corresponding Cartan matrix $A_X$ is necessarily of finite type. We denote by 
$\oi_X\in\Aut(A_X)$ the 
{\em opposition involution of $X$}, \ie the involutive diagram automorphism on $X$ 
induced by the action of the longest element $w_X$ of the parabolic Weyl group $W_X$ on $\Qlat_X$, see, \eg
~\cite[\S 3.11]{appel-vlaar-20}. \\

Following the recent approach proposed by Regelskis and Vlaar in \cite{regelskis-vlaar-20,regelskis-vlaar-21}, we say that a pair $(X,\tau)$ is a \emph{(generalized affine) Satake diagram} and write $(X,\tau)\in\gsat{\aff{A}}$ if $X \subset\aIS$, $\tau$ is an involutive diagram automorphism which preserves $X$, and
\begin{enumerate}
	\item $\tau|_X = \oi_X$,
	\item for any $i \in \aIS \setminus X$ such that $\tau(i)=i$, the connected component of $X \cup \{ i \}$ containing $i$ is not of type ${\sf A}_2$.
\end{enumerate}
A list of Satake diagrams for $\aff{A}$ is given in \cite[App.~A, Tables 5, 6, 7]{regelskis-vlaar-21}.
Henceforth, we fix an affine Satake diagram $(X,\tau)\in \gsat{\aff{A}}$.

\subsection{Pseudo--involutions}\label{ss:pseudo-involutions}
The diagram automorphism $\tau\in\Aut(\wh{A})$ extends canonically to an automorphism
of $\wh{\g}'$, given on the generators by $\tau(\cort{i})\coloneqq\cort{\tau(i)}$, $\tau(e_i)=e_{\tau(i)}$, and $\tau(f_i)=f_{\tau(i)}$. 
We then consider the Lie algebra automorphism $\tsat:\agp\to\agp$ given by
\begin{align}\label{def:theta}
	\tsat \coloneqq \Ad(\wt{w}_X) \circ \omega \circ \tau\,,
\end{align}
where $\omega$ denotes the Chevalley involution and $\Ad(\wt{w}_X)$ denotes the
automorphism of $\g$ given by the triple exponentials.\footnote{More precisely, if $w_X=s_{i_1}\cdots s_{i_k}$ is a reduced expression of the 
longest element, then $\wt{w}_X=\wt{s}_{i_1}\cdots\wt{s}_{i_k}$, where $s_i=\exp(e_i)\exp(-f_i)\exp(e_i)$. It is well--known that the operators $s_i$ ($i\in I$)
induce an action of the braid group on any integrable $\g$--module. In particular, their adjoint action
gives rise to an automorphism of $\g$. For a concise summary of the main properties of triple exponentials, we refer the reader to \cite[\S 3.7]{appel-vlaar-20}.}
Note that, since $\cp{\rt{i}}{\tau(c)}=\cp{\rt{\tau(i)}}{c}=0$, 
one has $\tau(c)=c$ and $\tsat(c)=-c$. In \cite{regelskis-vlaar-21}, $\tsat$ is referred to as a \emph{pseudo--involution of $\wh{\g}$
	of the second kind} (cf.~\cite{kac-wang-92}). 


\subsection{Quantum pseudo--involution}\label{ss:qtheta} 
The pseudo--involution $\tsat$ has a distinguished lift $\tsatq$ at the quantum level
as an algebra automorphism of $\Uqagp$.
This is obtained by choosing a suitable lift for each of the three factors in $\tsat$. We follow the construction of
$\tsatq$ given in \cite{appel-vlaar-20}. First, one
considers the usual Chevalley involution on $\Uqag$ given by \eqref{eq:chevalley}.
Then, the diagram automorphism $\tau$ extends canonically to an autormorphism of 
$\Uqagp$ given on the generators by $\tau(\Eg{i})=\Eg{\tau(i)}$, $\tau(\Fg{i})=\Fg{\tau(i)}$ and $\tau(\Kg{h})=\Kg{\tau(h)}$.

By ~\cite[\S 5]{lusztig-94}, the quantum Weyl group  of $\Uqag$ provides a suitable quantisation of $w_X\in W_X$.
Namely, let $\qWS{X}$ be the quantum Weyl group operator corresponding to $w_X$ acting on any integrable $\Uqag$--module,\footnote{More precisely, given a reduced expression $s_{i_1}\cdots s_{i_{\ell}}$ of $w_X$ in
	terms of fundamental reflections, set $\qWS{X}\coloneqq \qWS{i_1}\cdots\qWS{i_{\ell}}$, where $\qWS{j}=T''_{j,1}$ in the notation from \cite[5.2.1]{lusztig-94}.} and set $	\bt{\tsat} \coloneqq \tcorr{\tsat}\cdot\qWS{X}$\,.
Here $\tcorr{\tsat}$ is the operator defined on any weight vector
of weight $\lambda$ as the multiplication by $q^{\iip{\tsat(\lambda)}{\lambda}/2+\iip{\lambda}{\rho_X}}$, where  $\rho_X$ is the half--sum of the positive roots in $\rootsys_X$. Therefore, $\bt{\tsat}$ can be thought of as an element of the completion of $\Uqag$ with respect to integrable
modules.
By \cite[Lemma~4.3 (iii)]{appel-vlaar-20}, $\adbt{\tsat}\coloneqq\Ad{\bt{\tsat}}$ yields an automorphism 
of $\Uqag$.

The \emph{quantum pseudo--involution} $\tsatq$ is given by
\begin{equation}
	\tsatq\coloneqq\adbt{\tsat}\circ\omega\circ\tau\,.
\end{equation}
Note that, as in the classical case, $\tsatq$ is independent of the order of the three factors.
Moreover, $\tsatq$ descends to an automorphism of $\UqLg$.


\subsection{QSP subalgebras}\label{ss:qsp}
It is well--known that there is a \emph{family} of coideal subalgebras of $\Uqag$ associated to the pseudo--involution $\tsat$, 
introduced in \cite{letzter-02,kolb-14,regelskis-vlaar-20}
and parametrized by two sets, 
$\Parsetc\subset(\bsF^\times)^{\aIS}$ and $\Parsets\subset\bsF^{\aIS}$ (cf.~Remark~\ref{rmk:parameters}).\\

Let $\Uqag_X\subset\Uqag$ be the subalgebra generated by $E_i, F_i$, and $K_i$ ($i\in X$).
The {\em QSP subalgebra} of $\Uqag$ corresponding to $\tsat$ with parameters  $(\Parc,\Pars)\in \Parsetc \times \Parsets$ is the subalgebra
$\Uqk\subseteq\Uqag$ generated by  $\Uqag_X$, the elements $\Kg{h}$ with $h\in(\aQvext)^\tsat$, and the elements $\Bg{i}$ ($i\in \aIS \setminus X$) given by
\begin{equation} \label{Bi:def}
	\Bg{i} \coloneqq \Fg{i} + \parc{i}\cdot\tsatq(\Fg{i}) + \pars{i}\cdot\Kg{i}^{-1}\,.
\end{equation}

In the following, a morphism of $\Uqk$--modules will be simply referred to
as a {\em QSP intertwiner}.

\begin{remarks}\label{rmk:parameters}
	\hfill
	\begin{enumerate}[leftmargin=2em]\itemsep0.25cm
		\item
		Roughly, $\Parsetc\times\Parsets$ is the set of all pairs $(\parc{},\pars{})$ satisfying 
		$(\parc{i},\pars{i})=(1,0)$ ($i\in X$) and $\Uqk\cap\Uqah=\Uqah^{\tsat}$.
		For the explicit description of the sets $\Parsetc$ and $\Parsets$  see, \eg \cite[\S 6.8]{appel-vlaar-20}.
		\Omit{
			is however less concise. First, we choose a subset 
			$\aIS^* \subseteq \aIS \setminus X$ by picking a representative for every $\tau$-orbit 
			in $\aIS\setminus X$. Then, 
			\begin{itemize}\itemsep0.25cm
				\item 
				$\Parsetc$ is the set of tuples $\Parc\in(\bsF^\times)^{\aIS}$ such that $\parc{i}=1$ if $i\in X$ and $\parc{i} = \parc{\tau(i)}$ if $\{i , \tau(i) \} \cap \aIdiff = \emptyset$, where 
				\begin{equation}
					\aIdiff = \{ i \in \aIS^* \, | \, \tau(i) \ne i \text{ and } \exists j \in X \cup \{ \tau(i)\} \text{ such that } a_{ij} \ne 0 \}.
				\end{equation}
				\item $\Parsets$ is the set of tuples $\Pars \in \bsF^{\aIS}$ such that $\pars{i}=0$ if $i\in\aIS\setminus\aIns$, where 
				\begin{equation}
					\aIns = \{ i \in \aIS^* \, | \, \tau(i) =i \text{ and } \forall j \in X \, a_{ij} = 0 \}
				\end{equation}
				and, for all $(i,j) \in (\aIns)^{2}$, $a_{ij} \in 2\bbZ$ or $\Pars_j = 0$.
			\end{itemize}
			Note that $\Parsetc$ and $\Parsets$ do not depend upon the choice of $\aIS^*$.
		}
		\item \label{rmk:parameter-operator}
		Following \cite[\S 7.4]{appel-vlaar-20}, we shall regard the tuple $\Parc$ as a diagonal 
		operator on integrable modules. Namely, we fix henceforth a group 
		homomorphism $\parc{}:\aP\to\bbF^\times$ such that
		$\parc{}(\rt{i})\coloneqq\parc{i}$ ($i\in\aIS$). Then, $\parc{}$ acts on any weight vector of
		weight $\lambda$ as the multiplication by $\parc{}(\lambda)$.
		\item 
		By \cite[\S 6.2]{appel-vlaar-20}, it follows that
		$\Uqk\subset\Uqagp$.
		\rmkend
	\end{enumerate}
\end{remarks}

\subsection{Coideal property and monoidal action}\label{ss:coideal-module-cat}

\nc{\monact}{\triangleleft}

It is well--known that $\Uqk$ is a right coideal subalgebra of $\Uqag$, \ie $\Delta(\Uqk)\subset\Uqk\ten\Uqag$, see \cite{kolb-14}. The categorical 
counterpart of this property is described in terms of module categories.
Roughly, a module category (resp. a morphism of module categories)
is the analogue for a monoidal category (resp. a tensor functor) of what
a module (resp. a morphism of modules) is to a ring (resp. a morphism of rings).

We briefly recall these notions following \cite{haring-old}.
Let $\mathcal{A}$ be a monoidal category with tensor product $\otimes$
and unit object $\mathbf{1}$.
A (right) \emph{monoidal action} of $\mathcal{A}$  
on a category $\mathcal{B}$ is a functor $\monact \colon \mathcal{B} \times \mathcal{A} \to \mathcal{B}$ together with an associativity constraint and a unit constraint
 \begin{equation}
 \Phi \colon \monact \circ(\id \times \otimes) \to \monact\circ (\monact \times \id)
 \quad\mbox{and}\quad
 u \colon \monact \circ (\id \times\mathbf{1}) \to \id, 
\end{equation}
satisfying the analogues of the pentagon axiom and the unit axiom for a monoidal category (see \cite[Eqns.\ (6)--(7)]{haring-old}). 

Then, we say that $\mathcal{B}$ is a {\em module category} over $\mathcal{A}$ 
(or, equivalently, that $(\mathcal{B}, \mathcal{A})$ is an {\em action pair}) if it
is equipped with a monoidal action of $\mathcal{A}$.

\begin{example}
Set $\mathcal{A}=\Mod{\Uqag}$ and $\mathcal{B}=\Mod{\Uqk}$.
Since $\Uqk$ is a coideal subalgebra in $\Uqag$, the standard tensor
product induces a functor
\[
\monact\colon\Mod{\Uqk}\times\Mod{\Uqag}\to\Mod{\Uqk}
\]
which is readily verified to be a monoidal action of $\Mod{\Uqag}$
on $\Mod{\Uqk}$.
\rmkend
\end{example}

Let now $(\mathcal{B}, \mathcal{A})$ and $(\mathcal{B}', \mathcal{A}')$ be
two action pairs. A functor $(\mathcal{B}, \mathcal{A})\to(\mathcal{B}', \mathcal{A}')$ is the datum of
\begin{itemize}\itemsep0.25cm
	\item a monoidal functor $F_{A} \colon \mathcal{A} \to \mathcal{A}'$, 
	\item a functor $F_B \colon \mathcal{B} \to \mathcal{B}'$, 
	\item a natural isomorphism $\omega\colon F_B\circ\monact\to\monact\circ(F_B\times F_A)$ 
	satisfying the 
	analogue of the consistency condition for the tensor structure of a
	monoidal functor (see \cite[Eqns.~(8)--(9)]{haring-old}).
%
\end{itemize}
More concisely, we shall say that the functor $F_B$ intertwines through $F_A$ the action
of $\mathcal{A}$ on $\mathcal{B}$ and of $\mathcal{A}'$ on $\mathcal{B}'$.


\section{Trigonometric K-matrices} \label{s:rational} 

\summary{In this section,
	\begin{itemize}
		\item \ref{ss:tau-grading}: $\tau$-invariant shifts
		\item \ref{ss:spectral-k}: formal K-matrices
		\item \ref{ss:rational-k}: trigonometric K-matrices
		\item \ref{ss:unitarity}: unitary K-matrices
	\end{itemize}
}

In this section, we briefly review the construction
of  trigonometric K-matrices on irreducible finite--dimensional
$\UqLg$--modules. The results of this section are due to the first author and
Vlaar \cite{appel-vlaar-22}.


\subsection{$\tau$-invariant shifts}\label{ss:tau-grading}
Henceforth, we replace the homogeneous grading shift defined in \S\ref{ss:spectral-R} 
with a fixed \emph{$\tau$-invariant} grading shift. Namely, we fix a morphism 
$\tshift{z}:\UqLg[z,z^{-1}]\to\UqLg[z,z^{-1}]$ given by
\begin{align}
	\tshift{z}(\Kg{h})=\Kg{h}\, ,\qquad \tshift{z}(\Eg{i})=z^{\shifta{i}}\Eg{i}\, ,\qquad \tshift{z}(\Fg{i})=z^{-\shifta{i}}\Fg{i}\, ,
\end{align}
where $\shifta{}:\aIS\to\bbZ_{\geqslant0}$ is a non-zero $\tau$-invariant function. 
In particular, $\tshift{z}\circ\tau=\tau\circ\tshift{z}$.
Examples of $\tau$-invariant grading shifts are those determined by the characteristic function of any union of $\tau$-orbits
(\eg the principal grading shift).
Clearly, the homogeneous grading shift $\shift{z}$ is $\tau$-invariant if and only if $\tau(0)=0$.
In this case, we say that the QSP subalgebra is $\tau$--{\em restrictable}.\\

Let $V\in\Modfd{\UqLg}$ with action $\pi_V\colon \UqLg\to\End(V)$. 
As before, we denote by $\pi_{V,z}\coloneqq\pi_V\circ\tshift{z}$ the $\tau$-shifted action of 
$\UqLg$ on the modules $\shrep{V}{z}\coloneqq V\ten\bsF(z)$ and 
$\Lshrep{V}{z}\coloneqq V\ten\Lfml{\bsF}{z}$. Note that the $\tau$--analogue of Theorem~\ref{thm:rational-R} holds 
with $\Delta_z$ replaced by the $\tau$-shifted coproduct
$\Delta_z^{\tau}(x)=\id\ten\tshift{z}(\Delta(x))$. Henceforth, we fix a $\tau$-invariant grading shift and drop the superscript $\tau$.

\subsection{Spectral K-matrices}\label{ss:spectral-k}
Fix a QSP subalgebra $\Uqk\subset\Uqag$.
By Remark~\ref{rmk:parameters} (3), every finite--dimensional $\UqLg$--module $V$ is naturally acted upon by $\Uqk$ through the projection
$\Uqagp\to\UqLg$. In the following, we shall consider the restriction of $V$ to $\Uqk$. With a slight abuse of notation, we will denote by the same symbol
an element in $\Uqk$ and its image in $\UqLg$.\\

Set \[\cG\coloneqq\{g\in(\UqLg)^{\scsop{fd},\times}\;\vert\;\Ad(g)(\UqLg)=\UqLg\,, \shift{z}(g)=g\}\,.\]
By \cite[Thm.~4.2.1]{appel-vlaar-22}, $\UqLg$ has a $\cG$-family of universal 
\emph{spectral} K-matrices relative to the QSP subalgebra $\Uqk$. Namely, 
for any $g\in\cG$, there is a canonical formal series $\sKM{}{z}\in\fml{\CUqLg{}}{z}$
such that $\shift{a}(\sKM{}{z})=\sKM{}{az}$ ($a\in\bsF^\times$) and the following properties hold.
\vspace{0.25cm}
\begin{enumerate}\itemsep0.25cm
	\item For any $b\in\Uqk$, 
	\begin{align}\label{eq:spectral-k-intertwiner}
		\sKM{}{z}\cdot\shift{z}(b)=\psi(\shift{1/z}(b))\cdot\sKM{}{z}\, ,
	\end{align}
	where $\psi=\Ad(g)\circ\tsat_{q}^{-1}$\,.
	\item Set $\sRM{}{z}^{\psi} \coloneqq\psi\ten \id(\sRM{}{z}) $. Then,
	\begin{align}\label{eq:spectral-k-coproduct}
		\Delta_{w/z}(\sKM{\psi}{z})=\RM{\psi}^{-1}\cdot 1\ten\sKM{}{w}
		\cdot\sRM{}{zw}^{\psi}\cdot\sKM{}{z}\ten 1\, .
	\end{align}
\end{enumerate}
\vspace{0.25cm}
Moreover, $\sKM{}{z}$ is a solution of Cherednik's reflection equation
\begin{align}\label{eq:spectral-tw-RE}
	\sRM{}{w/z}_{21}^{\psi\psi}\cdot 1\ten\sKM{}{w}
	\cdot&\sRM{}{zw}^{\psi}\cdot\sKM{}{z}\ten 1=\\
	&=\sKM{}{z}\ten 1\cdot\sRM{}{zw}_{21}^{\psi}\cdot 1\ten\sKM{}{w}
	\cdot\sRM{}{w/z}\, ,
\end{align}
where $\sRM{}{z}_{21}^{\psi}\coloneqq\psi\ten\id(\sRM{}{z})_{21}$.

\Omit{
\begin{theorem}\label{thm:spectral-k}
	The quantum loop algebra $\UqLg$ has a $\cG$-family of universal 
	\emph{spectral} K-matrices relative to the QSP subalgebra $\Uqk$. Namely, 
	for any $g\in\cG$, there is a canonical formal series $\sKM{}{z}\in\fml{\CUqLg{}}{z}$
	such that $\shift{a}(\sKM{}{z})=\sKM{}{az}$ ($a\in\bsF^\times$) and the following properties hold.
	\vspace{0.25cm}
	\begin{enumerate}\itemsep0.25cm
		\item For any $b\in\Uqk$, 
		\begin{align}\label{eq:spectral-k-intertwiner}
			\sKM{}{z}\cdot\shift{z}(b)=\psi(\shift{1/z}(b))\cdot\sKM{}{z}\, ,
		\end{align}
		where $\psi=\Ad(g)\circ\tsat_{q}^{-1}$\,.
		\item Set $\sRM{}{z}^{\psi} \coloneqq\psi\ten \id(\sRM{}{z}) $. Then,
		\begin{align}\label{eq:spectral-k-coproduct}
			\Delta_{w/z}(\sKM{\psi}{z})=\RM{\psi}^{-1}\cdot 1\ten\sKM{}{w}
			\cdot\sRM{}{zw}^{\psi}\cdot\sKM{}{z}\ten 1\, .
		\end{align}
	\end{enumerate}
	\vspace{0.25cm}
	Moreover, $\sKM{}{z}$ is a solution of Cherednik's generalised reflection equation
	\begin{align}\label{eq:spectral-tw-RE}
		\sRM{}{w/z}_{21}^{\psi\psi}\cdot 1\ten\sKM{}{w}
		\cdot&\sRM{}{zw}^{\psi}\cdot\sKM{}{z}\ten 1=\\
		&=\sKM{}{z}\ten 1\cdot\sRM{}{zw}_{21}^{\psi}\cdot 1\ten\sKM{}{w}
		\cdot\sRM{}{w/z}\, ,
	\end{align}
	where $\sRM{}{z}_{21}^{\psi}\coloneqq\psi\ten\id(\sRM{}{z})_{21}$.
\end{theorem}
}

\begin{remark}
	The result above relies on the construction of universal K--matrices for quantum Kac--Moody
	symmetric pairs given in \cite[Thm.~8.11-8.12]{appel-vlaar-20}, building on the work of 
	Bao-Wang \cite{bao-wang-18b} and Balagovi\'c-Kolb \cite{balagovic-kolb-19}. 
	
	The equation \eqref{eq:spectral-tw-RE} was first introduced by Cherednik in \cite[Eq.~(4.14)]{cherednik-92}.
	\rmkend
\end{remark}

\subsection{Trigonometric K-matrices}\label{ss:rational-k}
A finite--dimensional $\UqLg$--module $V$ is
\begin{enumerate}\itemsep0.25cm
	\item\label{cond:qsp-irr} \emph{QSP irreducible} if it is irreducible as a module over ${\Uqk}$\,;
	\item\label{cond:gen-qsp-irr} \emph{generically QSP irreducible} if $\Lfml{V}{z}$ is irreducible as a module over $\Lfml{\Uqk}{z}$.\footnote{
		Note that $\Lfml{V}{z}$ is equipped with the shifted action (see \S\ref{ss:tau-grading}).}
\end{enumerate}
Such condition is the natural counterpart of the 
generic irreducibility of the tensor product and yields the following QSP analogue of Theorem \ref{thm:rational-R}, see \cite[Thm.~5.2.1]{appel-vlaar-22}.

\begin{theorem}[\cite{appel-vlaar-22}]\label{thm:rational-k}\hfill 
	\begin{enumerate}\itemsep0.25cm
		\item Every finite-dimensional irreducible $\UqLg$-module is generically QSP irreducible.
		\item
		Let $V,W\in\Modfd{\UqLg}$ be irreducible modules. 
		There is a formal Laurent series $g_{V}(z)\in\Lfml{\bsF}{z}$ and a polynomial
		non--vanishing operator $\trKM{V}{z}\in\End(V)[z]$ (unique up to a scalar) such that
		\begin{align}
			\sKM{V}{z}= g_{V}(z)\cdot\trKM{V}{z}\, ,
		\end{align}
		where $\sKM{V}{z}\colon \Lfml{V}{z}\to\Lfml{\psi^*(V)}{z^{-1}}$ is the formal  QSP intertwiner given by the action of $\sKM{}{z}$ on $V$.
		\item 
		The operators $\trKM{V}{w}$ and $\trKM{W}{z}$ satisfy Cherednik's reflection equation in $\End(V\ten W)(z,w)$ 
		\begin{align}\label{eq:rational-tw-RE}
			\begin{split}
				\rRM{\psi^*(W)\, \psi^*(V)}{w/z}_{21}\cdot 1\ten\trKM{W}{w}
				\cdot\rRM{\psi^*(V)W}{zw}\cdot\trKM{V}{z}\ten 1=\\
				=\trKM{V}{z}\ten 1\cdot\rRM{\psi^*(W)\, V}{zw}_{21}\cdot 1\ten\trKM{ W}{w}
				\cdot\rRM{VW}{w/z}\, ,
			\end{split}
		\end{align}
		where $\rRM{VW}{z}$ is the trigonometric R-matrix (cf.\ \S\ref{ss:rational-R}),
		and
		$\rRM{WV}{z}_{21}\coloneqq(1\, 2)\circ\rRM{WV}{z}\circ (1\, 2)$.
	\end{enumerate}
\end{theorem}


\subsection{Unitary K-matrices}\label{ss:unitarity}
Let $V,W\in\Modfd{\UqLg}$ be irreducible modules. By Theorem~\ref{thm:rational-R}, 
the trigonometric R-matrix $\rRM{VW}{z}$ satisfies the unitarity condition 
$\rRM{VW}{z}^{-1}=(1\, 2)\circ\rRM{WV}{z^{-1}}\circ(1\, 2)$. 
The analogue result for
trigonometric K-matrices requires an additional assumption to hold, see  
\cite[Prop.~5.4.1]{appel-vlaar-22}.

\begin{theorem}[\cite{appel-vlaar-22}]\label{thm:unitary-k}
Let $V\in\Modfd{\UqLg}$ be a $\psi$-involutive irreducible module,  
\ie such that $(\psi^2)^*(V)\simeq V$. The trigonometric K-matrices 
$\trKM{V}{z}$ and $\trKM{\psi^*(V)}{z}$ from Theorem~\ref{thm:rational-k} give rise through normalization to 
two non-vanishing QSP intertwiners
\[
\rKM{V}{z}: V(z)\to\psi^*(V)(z^{-1})\quad\mbox{and}\quad
\rKM{\psi^*(V)}{z}: \psi^*(V)(z)\to V(z^{-1})
\]
satisfying the unitarity condition
\begin{align}\label{eq:unitarity-k}
	\rKM{V}{z}^{-1}=\rKM{\psi^*(V)}{z^{-1}}\,.
\end{align}
Moreover, if $V(\zeta)$ is QSP irreducible for some $\zeta\in\bsF^\times$, 
then $\rKM{V}{\zeta}$ is well--defined and invertible.
\end{theorem}

\begin{remark}
	In contrast with the case of the R-matrix, the normalization yielding the unitary K-matrix $\rKM{V}{z}$ is non-canonical in general, since it follows from the normalization of the operator $\trKM{\psi^*(V)}{z^{-1}}\circ\trKM{V}{z}$. 
	In certain cases, however, there is a canonical normalization. We refer the
	reader to \cite[\S 5.4--5.6]{appel-vlaar-22}.
	\rmkend
\end{remark}


\subsection{Trigonometric  K-matrices for Kirillov--Reshetikhin modules}\label{ss:K-mx-KR}
Kirillov--Reshetikhin modules are minimal affinizations of 
irreducible $U_q\g$--modules whose highest weight is a multiple of 
a fundamental weight.
More precisely, for any $i\in I$, $k\in\bbZ_{> 0}$, and $a\in\bbC^\times$, 
the Kirillov-Reshetikhin module $W_{k,a}^{(i)}$ is the unique irreducible $\UqLg$-module whose
Drinfeld polynomials are all trivial except for the node $i$, where the roots are given by a $q_i$-string of length $k$ starting at $a$, see, \eg~\cite{kirillov-reshetikhin-87, chari-pressley}.
The \emph{fundamental representations} of $\UqLg$ are the simplest non--trivial Kirillov--Reshetikhin modules. Specifically, we set $\sfV_{\omega_i}\coloneqq W_{1,1}^{(i)}$.\\

Let $\eta_0\in\Aut(\widehat{A})$ be the diagram automorphism which fixes the affine node and acts on any other node as the opposition
involution $\oi_I$, see \S\ref{ss:gsat}. From the classification of generalized Satake diagrams of affine type in \cite[App.~A, Tables 5, 6 and 7]{regelskis-vlaar-21} it follows that a QSP subalgebra is $\tau$-restrictable (\ie $\tau$ fixes the affine node) if and only if $\tau$ is either
the identity or $\eta_0$ (except in type $\sfD^{(1)}_n$ with $n$ even, where $\eta_0 = \id$ but there exist nontrivial involutive diagram automorphisms fixing 0).
By \cite[Thm.~7.8.1]{appel-vlaar-22}, the following holds.

\begin{theorem}[\cite{appel-vlaar-22}]\label{thm:kr-k}
Let $\Uqk\subset\UqLg$ be a $\tau$--restrictable QSP subalgebra.
Let $W$ be a Kirillov-Reshetikhin $\UqLg$-module. 
\begin{enumerate}[leftmargin=2em]\itemsep0.25cm
				\item  There is a unique QSP intertwiner (up to a scalar multiple)
				\begin{equation} 
						\rKM{W}{z}:\shrep{W}{z}\to(\eta_0\tau)^*\shrep{(W)}{z^{-1}}
				\end{equation} 
				up to a scalar multiple and a unique shift in $W$.
				Moreover, $\rKM{W}{z}$ is a solution of the {\em diagrammatic} reflection equation 
					\begin{align}\label{eq:bala-kolb-re}
					\rRM{WV}{\tfrac{w}{z}}_{21}&\cdot \id\ten\rKM{W}{w}
					\cdot{\bf R}_{(\eta_0\tau)^*(V)\,W}(zw)\cdot\rKM{V}{z}\ten \id=\\
					&=\rKM{V}{z}\ten \id\cdot{\bf R}_{(\eta_0\tau)^*(W)\,V}(zw)_{21}\cdot \id\ten\rKM{W}{w}
					\cdot\rRM{VW}{\tfrac{w}{z}}\, .
				\end{align}
				\item If $\tau=\eta_0$, 
				there is a unique QSP intertwiner 
				\begin{equation} 
					\rKM{W}{z}:\shrep{W}{z}\to\shrep{W}{z^{-1}}
				\end{equation} 
				up to a scalar multiple and a unique shift in $W$.
				Moreover, $\rKM{W}{z}$ is a solution of the standard reflection equation
				\begin{equation}\label{eq:rational-original-RE}
					\begin{aligned}
						\rRM{WV}{\tfrac{w}{z}}_{21} & \cdot \id_V \ten\rKM{W}{w} \cdot \rRM{VW}{zw} \cdot \rKM{V}{z} \ten \id_W =\\
						&=\rKM{V}{z} \ten \id_W \cdot \rRM{WV}{zw}_{21} \cdot \id_V \ten \rKM{W}{w} \cdot \rRM{VW}{\tfrac{w}{z}}
					\end{aligned}
				\end{equation}
				\item If $\tau=\id$,  there is a unique QSP intertwiner
				\begin{equation} 
					\rKM{W}{z}:\shrep{W}{z}\to\shrep{W^*}{z^{-1}}
				\end{equation} 
				up to a scalar multiple and a unique shift in $W$.
				Moreover, under the identification of $W$ and $W^*$ as vector spaces, $\rKM{W}{z}$
				gives rise to a solution of the transposed reflection equation 
				\begin{equation}\label{eq:rational-transposed-RE}
					\begin{aligned}
						\rRM{WV}{\tfrac{w}{z}}_{21}^{{\sf t}_V {\sf t}_W}& \cdot \id_V \ten \rKM{W}{w} \cdot ( \rRM{VW}{zw}^{-1})^{\sf t_V} \cdot \rKM{V}{z} \ten \id_W = \\*
						\qquad &=  \rKM{V}{z} \ten \id_W \cdot (\rRM{WV}{zw}^{-1}_{21})^{\sf t_W} \cdot \id_V \ten \rKM{W}{w}\cdot\rRM{VW}{\tfrac{w}{z}}\,.
					\end{aligned}
				\end{equation}
\end{enumerate}
\end{theorem}
	
\begin{remark}
The result appears as a direct consequence of Theorem~\ref{thm:rational-k} (3)
in the case $\psi=\omega\circ\tau$. More precisely, one first observes that,
by \cite[Eqs.~(2.20) and (2.21)]{chari-02},  if $W=W_{k,a}^{(i)}$, then 
there is an isomorphism of $\UqLg$-modules
\begin{equation}\label{eq:KR}
	\omega^*(W)\simeq \eta_0^*(W)(a^{-2}q_i^{-2(k-1)})\,.
\end{equation}
This yields (1) and (2). In order to prove (3), one then observes that, 
by \cite[Eqs.~(2.20) and (2.21)]{chari-02}, 
there exists an integer $c\in\bbZ$ depending only on $\g$ 
such that, for any irreducible finite-dimensional $\UqLg$-module $V$, 
one has $\eta_0^*(V)\simeq V^*(q^c)$.
\rmkend
\end{remark}	
	
	
	\section{$o$KLR algebras} \label{s:oKLR}
	
	\summary{In this section:
		\begin{itemize}
			\item \ref{ss:quiver-involution}: quiver with an involution
			\item \ref{ss:orientifold-KLR}: orientifold KLR (and standard KLR)
			\item \ref{ss:polynomial-rep}: polynomial representation
			\item \ref{ss:convolution-product}: convolution product
		\end{itemize}
	}
	
	
	In this section, we review the definition and the basic properties of KLR and orientifold KLR ($o$KLR) algebras, 
	in particular their polynomial representation and their convolution product.
	
	\subsection{Quiver with an involution}\label{ss:quiver-involution}

Let $\sym_n = \langle s_1, \cdots, s_{n-1} \rangle$ denote the symmetric group on $n$ letters, and $\weyl_n = \langle s_0, s_1, \cdots, s_{n-1} \rangle$ the Weyl group of type $\mathsf{B}_{n}$, \ie $(\Z/2\Z)^n\rtimes\sym_n$. 
	
	Let $\Gamma = (J,\Omega)$ be a quiver with vertices $J$ and arrows $\Omega$. We assume that $\Gamma$ does not have loops. 
	Given an arrow $a \in \Omega$, let $s(a)$ be its source, and $t(a)$ its target. 
	If $i, j \in J$, let $\Omega_{ij} \subset \Omega$ be the subset of arrows $a$ such that $s(a) = i$ and $t(a) = j$. Let $a_{ij} = |\Omega_{ij}|$ and abbreviate $\dvec a_{ij} = a_{ij}+a_{ji}$. We assume that $a_{ij} < \infty$ for all $i,j \in J$. 
	
	\begin{definition}  \label{def: contr-inv}
		A (contravariant) \emph{involution} of the quiver $\Gamma$ is a pair of involutions $\theta \colon J \to J$ and $\theta \colon \Omega \to \Omega$ such that: 
		\begin{enumerate}
			\item $s(\theta(a)) = \theta(t(a))$ and $t(\theta(a)) = \theta(s(a))$ for all $a \in \Omega$, 
			\item if $t(a) = \theta(s(a))$ then $a = \theta(a)$. 
		\end{enumerate} 
	\end{definition} 
	
	We denote by $J^\theta$ the subset of fixed points of $\theta$. 
	Let $\N[J]$ be the commutative semigroup freely generated by $J$. We call elements of $\N[J]$ \emph{dimension vectors}. Given a dimension vector $\beta = \sum_{i \in J} \beta(i) \cdot i$,  set $\Norm{\beta} = \sum_{i \in J} \beta(i)$. 
	
	We call a sequence $\nu = \nu_1 \cdots \nu_n \in J^{n}$ a \emph{composition} of $\beta$ of length $\ell(\nu) = n$ if $\norm{\nu} = \sum_{k=1}^n \nu_k = \beta$. We also set $\Norm{\nu}=n$. 
Let $J^\beta$ denote the set of all compositions of~$\beta$. There is a left action 
	of $\mathfrak{S}_n$ on $J^n$ by permutations
	\begin{equation} \label{eq: sym act seq}
		s_k \cdot \nu_1 \cdots \nu_n = \nu_1 \cdots \nu_{k+1} \nu_k \cdots \nu_n \quad (1 \leqslant k \leqslant n-1), 
	\end{equation} 
	whose orbits are the sets $J^\beta$ for all $\beta$ with $\Norm{\beta} = n$. 
	
	Let $\wor = \bigcup_{\beta \in \N[J]} J^\beta$ be the set of compositions of all dimension vectors. 
	We consider $\wor$ as a monoid with respect to concatenation: $\nu \mu = \nu_1\cdots \nu_{\ell{\nu}}\mu_1\cdots \mu_{\ell{\mu}}$, with the zero dimension vector composition as the identity. 
	
	\noeqref{eq: the theta doubling}
	
	The involution $\theta$ induces an involution $\theta \colon \N[J] \to \N[J]$. We call dimension vectors in $\N[J]^\theta$ \emph{self-dual}. 
	We will always assume, for any $\beta \in \N[J]^\theta$, that if $i \in J^\theta$ then $\beta(i)$ is even. 
	Set $\Norm{\beta}_{\theta} = \Norm{\beta}/2$ and  
	\begin{equation} \label{eq: the theta doubling}
  {}^\theta(-) \colon \N[J] \to \N[J]^\theta, \quad \beta \mapsto {}^\theta \beta = \beta + \theta(\beta).
	\end{equation}
	We call a sequence $\nu = \nu_1 \cdots \nu_n \in J^n$ an \emph{isotropic composition} of $\beta$ if $\ttt\norm{\nu}=\sum_{k=1}^n {}^\theta\nu_i = \beta$. 
Let ${}^\theta J^\beta$ denote the set of all isotropic compositions of $\beta$. 
	There is a left action 
	of $\weyl_n$ on $J^n$ extending \eqref{eq: sym act seq}, given by 
	\[ s_0 \cdot \nu_1 \cdots \nu_n = \theta(\nu_1) \nu_2 \cdots \nu_n,\]
	whose orbits are the sets ${}^\theta J^\beta$ for all self-dual $\beta$ with $\Norm{\beta}_\theta = n$. 
	Let $\twor = \bigcup_{\beta \in \N[J]^\theta} \comp$ be the set of all isotropic compositions of all self-dual dimension vectors. 
	
	\subsection{KLR and $o$KLR algebras}\label{ss:orientifold-KLR}

\noeqref{eq: idem def rels 2}
\noeqref{eq: quad def rels 2} 
\noeqref{eq: mixed def rel 2}
\noeqref{eq: def braid rel Aaa}
	
	We recall the definition of orientifold KLR algebras as given by the second author in \cite{Przez-oklr} (see also \cite{VV-HecB, Przez-coha, PAW-B, PAR}).

\begin{definition}\label{def: enhanced quiver}
Let $\Gamma$ be a quiver with a contravariant involution $\theta$ and a dimension vector $\bm\lambda \in \N[J]$ such that $\bm\lambda(i) = 0$ if $i \in J^\theta$. Note that $\bm\lambda$ need not be self-dual. 
We call $\bm\lambda$ the \emph{framing dimension vector}, and the datum $(\Gamma, \theta, \bm\lambda)$ an \emph{enhanced quiver}.
\end{definition}

	Let $(\Gamma, \theta, \bm\lambda)$ be a fixed enhanced quiver. 
Set 
	\[ P_{ij}(u,v) = \delta_{i \neq j} (v-u)^{a_{ij}}, \quad P_{i}(u) = \delta_{i \neq \theta(i)} (-u)^{\bm\lambda(i)}\] 
	for $i, j \in J$, and define $(Q,Q')$ as 
	\begin{equation} \label{eq: P->Q} Q_{ij}(u,v) = P_{ij}(u,v)P_{ji}(v,u), \quad Q'_i(u) = P_i(u)P_{\theta(i)}(-u), \qquad (i,j \in J). 
	\end{equation}  
	
	\begin{definition} \label{def: oklr}
		Let $\beta \in \N[J]^\theta$ with $\Norm{\beta}_{\theta} = n$, and $\alpha \in \N[J]$ with $\ttt \alpha = \beta$. 
\be 
\item The \emph{KLR algebra} $\mathcal{R}(\alpha)$ associated to $(\Gamma, \alpha)$ is the $\cor$-algebra generated by 
$e(\nu)$ $(\nu \in J^\alpha)$, $x_l$ $(1 \leqslant l \leqslant n)$ and $\tau_k$ $(1 \leqslant k \leqslant n-1)$, subject to relations \eqref{eq: idem def rels 1}, \eqref{eq: pol def rels}, \eqref{eq: quad def rels}, \eqref{eq: def braid rel Aa},  \eqref{eq: def braid rel A} and \eqref{eq: mixed def rel 1}. 
\item The \emph{orientifold KLR algebra} $\oklr$ associated to $(\Gamma, \theta,\bm\lambda;\beta)$ is the $\cor$-algebra generated by $e(\nu)$ $(\nu \in \comp)$, $x_l$ $(1 \leqslant l \leqslant n)$, and $\tau_k$ $(0 \leqslant k \leqslant n-1)$, subject to all the relations \eqref{eq: idem def rels 1}--\eqref{eq: mixed def rel 3}. 
\ee

		\vspace{0.25cm}
		\begin{itemize}\itemsep0.25cm
			\item {\it Idempotent relations:}
			\begin{align} 
			e(\nu)e(\nu') = \delta_{\nu,\nu'}e(\nu), \quad 
			x_l e(\nu) =& \ e(\nu) x_l, \quad \tau_k e(\nu) = e(s_k \cdot \nu) \tau_k, \label{eq: idem def rels 1} \\
 \tau_0 e(\nu) =& \ e(s_0 \cdot \nu) \tau_0, \label{eq: idem def rels 2} 
			\end{align}
			\item{\it Polynomial relations:}
			\begin{equation} \label{eq: pol def rels}
			x_lx_{l'} = x_{l'}x_l, 
			\end{equation}
			\item {\it Quadratic relations:} 
			\begin{align}
			\tau_k^2e(\nu) =& \ Q_{\nu_k,\nu_{k+1}}(x_{k+1},x_k)e(\nu), \label{eq: quad def rels} \\
	 \tau_0^2 e(\nu) =& \ Q'_{\nu_1}(-x_1) e(\nu),    \label{eq: quad def rels 2}
			\end{align}
			\item {\it Deformed braid relations:}
			\begin{align} \label{eq: def braid rel Aa}
				\tau_k\tau_{k'} =& \ \tau_{k'}\tau_k \ \ \mbox{if} \ \ k \neq k' \pm 1, \\
 \quad \tau_0\tau_k =& \ \tau_k \tau_0 \ \ \ \mbox{if} \ \ k \neq 1, \label{eq: def braid rel Aaa}
			\end{align}
			\begin{equation} \label{eq: def braid rel A}
				(\tau_{k+1}\tau_k\tau_{k+1} - \tau_k\tau_{k+1}\tau_k)e(\nu) = \textstyle \delta_{\nu_k,\nu_{k+2}} \frac{Q_{\nu_k,\nu_{k+1}}(x_{k+1},x_k) - Q_{\nu_k,\nu_{k+1}}(x_{k+1},x_{k+2})}{x_k - x_{k+2}}e(\nu), \end{equation} 
			\begin{align}
				\label{eq: B braid rel klr}			&((\tau_1\tau_0)^2 - (\tau_0\tau_1)^2)e(\nu) = \\ \\ 
				&=\left\lbrace
				\begin{array}{ll} 
					\frac{Q'_{\nu_2}(x_2) - Q'_{\nu_1}(x_1)}{x_1+x_2}\tau_1e(\nu) & \textrm{ if } \ \nu_1\neq \nu_2, \ \nu_2 = \theta(\nu_1) \\ \\
					\frac{Q_{\nu_1,\nu_2}(x_2,-x_1) - Q_{\nu_1,\nu_2}(-x_2,-x_1)}{x_2} \tau_0 e(\nu) & \textrm{ if } \ \nu_1\neq\theta(\nu_1), \ \nu_2 = \theta(\nu_2), \\ \\ 
					\frac{Q_{\nu_1,\nu_2}(x_2,-x_1)-Q_{\nu_1,\nu_2}(x_2,x_1)}{x_1x_2} (x_1\tau_0+1)e(\nu) & \textrm{ if } \ \theta(\nu_1) = \nu_1 \neq \nu_2=\theta(\nu_2),  \\ \\
					0 & \textrm{ else,} 
				\end{array}\right. 
			\end{align}
			\item {\it Mixed relations:}
			\begin{align} \label{eq: mixed def rel 1}
				(\tau_kx_l - x_{s_k(l)}\tau_k)e(\nu) =& \ 
				\left\lbrace
				\begin{array}{ll}
					- e(\nu) & \textrm{ if } \ l=k,\ \nu_k=\nu_{k+1}, \\[0.3em]
					e(\nu) & \textrm{ if } \ l=k+1,\ \nu_k=\nu_{k+1}, \\[0.3em]
					0 & \textrm{ else, } 
				\end{array} 
				\right. \\ \\ \label{eq: mixed def rel 2}
				(\tau_0x_1 + x_{1}\tau_0)e(\nu) =& \ 
				\left\lbrace
				\begin{array}{ll}
					0 & \textrm{ if } \ \nu_1 \neq \theta(\nu_1), \\[0.3em]
					-2 e(\nu) & \textrm{ if } \ \nu_1 = \theta(\nu_1), 
				\end{array} 
				\right. \\
				\tau_0x_l =& \ x_l \tau_0 \ \ \mbox{if} \ \ l \neq 1, \label{eq: mixed def rel 3}
			\end{align} 
		\end{itemize}
	where, in all the relations above, $1 \leqslant l,l' \leqslant n$ and $1 \leqslant k,k' \leqslant n-1$, except \eqref{eq: def braid rel A}, where $1 \leqslant k \leqslant n-2$. 
	\end{definition}  
	
	These algebras are endowed with the following grading: 
	\begin{align*}
		\deg e(\nu) =& \ 0, \\
		\deg x_k =& \ 2, \\ 
		\deg \tau_k e(\nu) =& \ 
		\left\lbrace
		\begin{array}{ll}
			-2 & \textrm{ if } \nu_k=\nu_{k+1}, \\ 
			\dvec a_{\nu_k,\nu_{k+1}} \quad  & \textrm{ otherwise, } \\  
		\end{array} 
		\right. \\
		\deg \tau_0 e(\nu) =& \ 
		\left\lbrace
		\begin{array}{ll}
			-2 & \textrm{ if } \theta(\nu_1) = \nu_{1}, \\ 
			{}^\theta\bm\lambda(\nu_1) \quad  \ & \textrm{ otherwise. } \\ 
		\end{array} 
		\right. 
	\end{align*} 
If $\bm\lambda = 0$, we abbreviate 
\[ \oaklr = \oklr. \]
	
\begin{remark}
It follows from the PBW theorems (see, e.g., \cite[Thm.~ 3.7]{Rouquier-2KM}, \cite[Prop.~ 2.9]{Przez-oklr}) for the KLR and oKLR algebras that $\klr$ is in fact isomorphic to the subalgebra of $\oklr$ generated by $e(\nu)$ $(\nu \in J^\alpha)$, $x_l$ $(1 \leqslant l \leqslant n)$ and $\tau_k$ $(1 \leqslant k \leqslant n-1)$. 
\rmkend
\end{remark}

	Let $\mathds{1}$ and ${}^\theta\mathds{1}$ denote the regular representations (in degree zero) of the trivial algebras $\klrv{0}$ and $\oklrv{0}$, respectively. 
	For a fixed $\bm\lambda \in \N[J]$, set 
	\[
	\Modgr{\mathcal{R}} = \soplus_{\alpha \in \N[J]} \Modgr{\klr}, \quad \Modgr{{}^\theta \mathcal{R}(\bm\lambda)} = \soplus_{\beta \in \N[J]^{\theta}} \Modgr{\oklr}, 
	\] 
	and abbreviate $\Modgr{{}^\theta \mathcal{R}} = \Modgr{{}^\theta \mathcal{R}(\bm\lambda = 0)}$. 
	We use analogous notation for direct sums of categories of finite dimensional modules. 
	
	
	\subsection{Polynomial representation}\label{ss:polynomial-rep}
	
	Set 
	\begin{alignat*}{6}
		\bP_{ \nu } \coloneqq& \ \cor [x_1, \ldots, x_n] e(\nu), &\quad \bPh_{ \nu } \coloneqq& \ \fml{\cor}{x_1, \ldots, x_n} e(\nu), &\quad
		\widehat{\mathbb{K}}_\nu \coloneqq& \ \Lfml{\cor}{x_1, \ldots, x_n} e(\nu), \\
		\tbP_{ \beta } \ \coloneqq& \soplus_{\nu \in \comp} \bP_{ \nu }, &\quad
		\tbPh_{ \beta } \ \coloneqq& \soplus_{\nu \in \comp} \bPh_{ \nu }, &\quad
		{}^\theta\widehat{\mathbb{K}}_{ \beta } \ \coloneqq& \soplus_{\nu \in \comp} \widehat{\mathbb{K}}_{ \nu }. 
	\end{alignat*}
	We abbreviate $x_{-l} = - x_l$ for $1 \leqslant l \leqslant n$. 
	There is a natural left action of the Weyl group $\weyl_n$ on $\Lfml{\cor}{x_1,\hdots,x_n}$ 
	given by $w\cdot x_l = x_{w(l)}$. This extends to an action on ${}^\theta\widehat{\mathbb{K}}_{ \beta }$ by
	\begin{equation} \label{eq: weyl action klr} w \cdot f e(\nu) = w(f) e(w\cdot \nu), \end{equation} 
	for $w \in \weyl_n$ and $f \in \Lfml{\cor}{x_1,\hdots,x_n}$. 
	
	\begin{prop} \label{pro: polrep oklr}
		The algebra $\oklr$ has a faithful polynomial representation on $\tbP_{ \beta }$, given by: 
		\begin{itemize} 
			\item $e(\nu)$ $(\nu \in \comp)$ acting as projection onto $\bP_\nu$, 
			\item $x_1, \hdots, x_n$ acting naturally by multiplication, 
			\item $\tau_1, \hdots, \tau_{n-1}$ acting via  
			\[ 
			\tau_k \cdot f e(\nu) = 
			\left\lbrace
			\begin{array}{ll}
				\displaystyle\frac{s_k(f)-f}{x_{k}-x_{k+1}}e(\nu) & \textrm{ if } \nu_k=\nu_{k+1}, \\ \\ 
				P_{\nu_k,\nu_{k+1}}(x_{k},x_{k+1})s_k( f)e\bigl(s_k\cdot \nu \bigr) & \textrm{ otherwise, } \\ \\
			\end{array} 
			\right. 
			\]
			\item $\tau_0$ acting via 
			\[
			\tau_0 \cdot f e(\nu) = 
			\left\lbrace
			\begin{array}{ll}
				\displaystyle\frac{s_0(f)-f}{x_1}e(\nu) & \textrm{ if } \theta(\nu_1) = \nu_1, \\ \\
				P_{\nu_1}(x_1)s_0(f)e\bigl(s_0\cdot \nu \bigr) & \textrm{ otherwise. } 
			\end{array}
			\right. 
			\]
		\end{itemize} 
	\end{prop} 
	
	\begin{proof}
		See \cite[Prop.~ 2.7]{Przez-oklr}. 
	\end{proof}
	
	Next, for each $i,j \in J $,
	we choose holomorphic functions $c_{ij}(u,v) \in \fml{\cor}{u,v}$ such that
	\begin{equation} \label{eq: c-ij} 
		c_{ij}(u,v)c_{ji}(v,u)=1, \quad c_{ii}(u,v)=1, \quad c_{ij}(u,v) = c_{\theta(j)\theta(i)}(-v,-u).  
	\end{equation}
	Moreover, for each $i \in J$, we also choose holomorphic functions $c_i \in \fml{\cor}{u}$ such that 
	\begin{equation} \label{eq: c-i} 
		c_i(u)c_{\theta(i)}(-u)=1, \qquad i = \theta(i) \ \Rightarrow \ c_i(u) = 1. 
	\end{equation} 
	Set
	\begin{equation*}
		\widetilde{P}_{ij}(u,v) = P_{ij}(u,v)c_{ij}(u,v), \quad \widetilde{P}_{i}(u) = P_i(u)c_i(u). 
	\end{equation*} 
Note that the corresponding $(Q,Q')$ remain unchanged. 
	
	\begin{corollary} \label{cor: skew ring emb}
		There is an injective $\tbP_{ \beta }$-algebra homomorphism 
		\begin{equation} \label{eq: oklr loc}
		\oklr \hookrightarrow \skewrr
		\end{equation}
		given by 
		\begin{align*} 
			\tau_0 e(\nu) =& \ 
			\left\lbrace
			\begin{array}{ll}
				x_1^{-1}(s_0-1	) e(\nu) \ \qquad \qquad & \textrm{ if } \nu_1 = \theta(\nu_1), \\[2pt] 
				\widetilde{P}_{\nu_1}(x_1)s_0e(\nu) \ \qquad \qquad & \textrm{ otherwise, } 
			\end{array}
			\right. \\
			\tau_k e(\nu) =& \ 
			\left\lbrace
			\begin{array}{ll}
				(x_k-x_{k+1})^{-1}(s_k-1)e(\nu) & \textrm{ if } \nu_k=\nu_{k+1}, \\ 
				\widetilde{P}_{\nu_k,\nu_{k+1}}(x_k,x_{k+1})s_k e(\nu) & \textrm{ otherwise, } \\[2pt] 
			\end{array} 
			\right. 
		\end{align*} 
		for $1 \leqslant k \leqslant n-1$. 
	\end{corollary}
	
	\begin{proof}
		See \cite[Corollary 2.8]{Przez-oklr}. 
	\end{proof}
	
	Given $\alpha \in \N[J]$ with $\ttt \alpha = \beta$, let 
	\[ 
	\bP_{ \alpha } \ \coloneqq \soplus_{\nu \in J^\alpha} \bP_{ \nu }, \qquad
	\bPh_{ \alpha } \ \coloneqq \soplus_{\nu \in J^\alpha} \bPh_{ \nu }, \qquad
	{}^\theta\widehat{\mathbb{K}}_{ \alpha } \ \coloneqq \soplus_{\nu \in J^\alpha} \widehat{\mathbb{K}}_{ \nu }. \] 
	The embedding \eqref{eq: oklr loc} restricts to a $\bP_{ \alpha }$-algebra homomorphism 
	\begin{equation} \label{eq: klr loc}
	\klr \hookrightarrow \widehat{\mathbb{K}}_\alpha \rtimes \cor[\sym_n].
	\end{equation}

	\subsection{One-dimensional modules}

	
	Given $\mu \in \comp$, let ${}^\theta L(\mu)$ be the free $\cor$-module $\cor u_\mu$ with generator $u_\mu$ of degree zero. 
	
\begin{lemma} \label{lem: 1dim mod oklr} 
Setting
\begin{equation} \label{eq: 1dim mod oklr}    x_l \cdot u_\mu= 0, \quad \tau_k \cdot u_\mu = 0, \quad e(\nu) \cdot u_\mu = \delta_{\nu,\mu}u_\mu, \end{equation}
for $1 \leqslant l \leqslant n$, $0 \leqslant k < n$ and $\nu \in \comp$, makes ${}^\theta L(\mu)$ into an $\oklr$-module if and only if 
\begin{enumerate}[label=(\alph*), itemsep = 2pt] 
\item 
$\mu_k \neq \mu_{k+1}$ and 	$\dvec a_{\mu_k,\mu_{k+1}} \geqslant 1$  (for $1 \leqslant k < n$), 
\item whenever $\mu_k = \mu_{k+2}$, then $\dvec a_{\mu_{k},\mu_{k+1}} \neq 1$ (for $1 \leqslant k < n-1$),
\item $\mu_1 \neq \theta(\mu_1)$ and ${}^\theta\bm\lambda(\mu_1) \geqslant 1$. 
		\end{enumerate}
\end{lemma} 
	
\begin{proof}
The proof is based on a direct inspection of the defining relations of $\oklr$. 
The definition of the action \eqref{eq: 1dim mod oklr} implies that the LHS of the relations \eqref{eq: quad def rels}--\eqref{eq: mixed def rel 3} must be zero. The RHS of \eqref{eq: mixed def rel 1} is zero if and only if $\mu_k \neq \mu_{k+1}$. Moreover, the RHS of \eqref{eq: quad def rels} is zero if and only if $\mu_k = \mu_{k+1}$ (which cannot be the case) or the polynomial $Q_{\mu_k,\mu_{k+1}}(x_{k+1}, x_k)$ has no constant term. The latter is the case if and only if $\dvec a_{\mu_k,\mu_{k+1}} \geqslant 1$. 
Similarly, the RHS of \eqref{eq: def braid rel A} is zero if and only if condition (b) holds. 

Next, the RHS of \eqref{eq: mixed def rel 2} is zero if and only if $\mu_1 \neq \theta(\mu_1)$. Moreover, the RHS of \eqref{eq: quad def rels 2} is zero if and only if $\mu_1 = \theta(\mu_1)$ (which cannot be the case) or $Q'_{\mu_1}(-x_1)$ has no constant term. The latter is the case if and only if ${}^\theta\bm\lambda(\mu_1) \geqslant 1$. 

Since all the other relations hold without any extra assumptions, we have proven the necessity and sufficiency of conditions (a)--(c). 
\end{proof} 
	
	Note that if $\mu \in J^\alpha$ and conditions (a)--(b) are satisfied then \eqref{eq: 1dim mod oklr} (with $1 \leqslant k <n$) defines the structure of a $\klr$-module on $\cor u_\mu$, which we then denote by $L(\mu)$. 
	
	\subsection{Convolution product and monoidal action} \label{ss:convolution-product}

	We recall the definition of the convolution product of modules over KLR algebras. 
	Let $\alpha, \alpha' \in \N[J]$ with $\Norm{\alpha} = n$ and $\Norm{\alpha'} = n'$. Set 
	\[ e_{\alpha,\alpha'} = \sum_{\stackrel{\nu \in J^{\alpha+\alpha'}}{ \nu_1 \cdots \nu_n \in J^\alpha}} e(\nu) \in \klrv{\alpha+\alpha'}.\]
	There is a non-unital algebra homomorphism 
	\begin{equation} \label{eq: induction inc klr}
	\iota_{\alpha,\alpha'} \colon \ \klrvv{\alpha}{\alpha'} := \klrv{\alpha} \otimes \klrv{\alpha'} \ \to \ \klrv{\alpha + \alpha'} 
	\end{equation}
	given by $e(\nu) \otimes e(\mu) \mapsto e(\nu\mu)$ for $\nu \in J^\alpha$, $\mu \in J^{\alpha'}$ and 
	\begin{alignat}{3} \label{eq: induc1}
		x_l \otimes 1 \mapsto& \ x_{l} e_{\alpha,\alpha'}, \quad& 1 \otimes x_{l'} \mapsto& \ x_{m+l'}e_{\alpha,\alpha'}\,, &\qquad& 
		\\ \label{eq: induc2}
		\tau_k \otimes 1 \mapsto& \ \tau_{k}e_{\alpha,\alpha'}, \quad& 1 \otimes \tau_{k'} \mapsto& \ \tau_{m+l}e_{\alpha,\alpha'\,,} &\qquad& 
	\end{alignat} 
	where 
	$1 \leqslant l \leqslant n$, $1 \leqslant l' \leqslant n'$, $1 \leqslant k< n$, and $1 \leqslant k'< n'$.
	Let $M$ be a graded $\klrv{\alpha}$-module and $N$ be a graded $\klrv{\alpha'}$-module. Their \emph{convolution product} is defined as 
	\[
	M \circ N = \klrv{\alpha+\alpha'}e_{\alpha,\alpha'}\otimes_{\klrvv{\alpha}{\alpha'}}(M \otimes N). 
	\] 
	
	Next, let $\beta \in \N[J]^\theta$ with $\Norm{\beta}_\theta = n$. Set 
	\[ {}^\theta e_{\beta,\alpha'} = \sum_{\substack{\nu \in {}^\theta J^{\beta+{}^\theta\alpha'}, \ \nu_1\hdots\nu_n \in {}^\theta J^\beta \\ \nu_{n+1}\hdots\nu_{n+n'} \in J^{\alpha'}}} e(\nu) \in \oklrv{\beta+{}^\theta\alpha'}.\] 
	There is an injective non-unital algebra homomorphism 
	\begin{equation} \label{eq: induction inc}
	\ttt\iota_{\beta,\alpha'} \colon \ \oklrvv{\beta}{\alpha'} := \oklr \otimes \klrv{\alpha'} \ \to \ \oklrv{\beta+{}^\theta\alpha'}
	\end{equation}
	given by formulae \eqref{eq: induc1}-\eqref{eq: induc2} (with $\nu \in \comp$ and $e_{\beta,\alpha'}$ replaced by ${}^\theta e_{\beta,\alpha'}$) and $\tau_0 \otimes 1 \mapsto \tau_0 {}^\theta e_{\beta,\alpha'}$. 
	The \emph{convolution action} of $N \in \Modgr{\klrv{\alpha'}}$ on $M \in \Modgr{\oklr}$ is defined as 
	\[
	M \acts N =  \oklrv{\beta+{}^\theta\alpha'}{}^\theta e(\beta,\alpha')\otimes_{\oklrvv{\beta}{\alpha'}}(M \otimes N). 
	\] 
	

	\begin{prop}
		The category $\Modgr{\mathcal{R}}$ is monoidal with product $\circ$ and unit $\mathds{1}$. Moreover, there is a right monoidal action of $\Modgr{\mathcal{R}}$ on \linebreak $\Modgr{{}^\theta \mathcal{R}(\bm\lambda)}$ via $\acts$. 
		The product and the action restrict to the full subcategories $\Modfdgr{\mathcal{R}}$ and $\Modfdgr{{}^\theta \mathcal{R}(\bm\lambda)}$, and descend to the corresponding Grothendieck groups. 
	\end{prop}
	
	\begin{proof}
In the case of KLR algebras, the proof is well known, see, e.g., \cite[\S3.1]{Khovanov-Lauda-1}. 
In the orientifold case, the argument is analogous, so we only highlight the main points. 
The crucial fact is that the $\oklrvv{\beta}{\alpha'}$-module $\oklrv{\beta+{}^\theta\alpha'}$ is free of finite rank (see, e.g., \cite[Lem.~8.7]{VV-HecB}). Firstly, it implies that $M \acts N$ is finite dimensional whenever $M$ and $N$ are. Secondly, it implies that the bifunctor $- \acts -$ is biexact, and hence descends to the Grothendieck groups of graded finite dimensional modules. Thirdly, 
the existence of a suitable associativity constraint (see \S\ref{ss:coideal-module-cat}) follows from the natural isomorphisms
\[
M \acts (N \circ P) \cong
\oklrv{\beta+{}^\theta\alpha^\circ} e \otimes_{\oklr \otimes \klrv{\alpha'} \otimes \klrv{\alpha''}}(M\otimes N \otimes P)
\cong 
(M \acts N) \acts P, 
\]
where $\alpha^\circ = \alpha' + \alpha''$ and $e$ is the sum of all idempotents $e(\nu)$ such that $\nu \in {}^\theta J^{\beta+{}^\theta\alpha^\circ}$, $\nu_1 \cdots \nu_n \in {}^\theta J^{\beta}$, $\nu_{n+1} \cdots \nu_{n+n'} \in J^{\alpha'}$ and $\nu_{n+n'+1} \cdots \nu_{n+n'+n''} \in J^{\alpha''}$. 
\end{proof}
	
	
	\section{The combinatorial model} \label{s:combinatorics}
	Throughout this section, we fix the following data:
	\begin{itemize}\itemsep0.25cm
		\item an affine Satake diagram $(X,\tau)$ with pseudo-involution $\tsat$\,;
		\item a QSP subalgebra $\Uqk\subseteq\UqLg$, associated to $(X,\tau)$, with parameters $(\Parc,\Pars)$ \,;
		\item a QSP admissible twisting operator $\psi\in\Aut(\UqLg)$ (cf.~\S\ref{ss:spectral-k}).
	\end{itemize}
	
\subsection{$J$--data}\label{ss:combinatorics-0}
Let $\qJ$ be a fixed index set. 
	A \emph{$J$--datum} is a choice, for each $i\in J$, of
	a finite-dimensional $\UqLg$-module $\qV{i}$ and 
	a non-zero scalar $\qX{i}\in\cor^{\times}$
such that the following properties are satisfied.
\vspace{0.25cm}
	\begin{enumerate}[label=(P\arabic*), start=1] \itemsep0.25cm
	\item \label{cond: P1}{\bf Admissibility}.  For any $i\in\qJ$, the module $\qV{i}$ is an 
	irreducible {\em real} module (\ie $\qV{i}\ten\qV{i}$ is irreducible).
	\item {\bf R-matrices}\label{cond: P2}. For any $i,j\in\qJ$, there is a non-vanishing \emph{unitary} trigonometric R-matrix 
	\[\rRM{\qV{i}\qV{j}}{w/z}:\zqV{i}{z}\ten\zqV{j}{w}\to\zqV{i}{z}\ten\zqV{j}{w}\]
	such that $\rRMv{\qV{i}\qV{i}}{1}=\id$, where $\zqV{i}{z}=(\qV{i})(z)$
	(cf.~ Theorem.~\ref{thm:rational-R} and Remarks~\ref{rmk:rational-R}).
	\item {\bf Poles}. \label{cond: P3} For any $i,j\in J$, the poles of the R-matrix $\rRM{\qV{i}\qV{j}}{w/z}$ 
	are in $q^{1/m}\fml{\bbC}{q^{1/m}}$ for some $m>0$.
\end{enumerate}

For any $i,j\in J$, we denote by $\dR{ij}(z)\in\bsF[z]$ the \emph{denominator}
of the trigonometric R-matrix $\rRM{\qV{i}\qV{j}}{z}$, \ie the polynomial of smallest degree 
such that $\dR{ij}(z)\rRM{\qV{i}\qV{j}}{z}$ is defined over $\bsF[z]$.
Moreover, we denote by $\dR{ij}\in\bbZ_{\geqslant0}$ 
the order of the pole of  $\rRM{\qV{i}\qV{j}}{z}$ at $z=\qX{j}/\qX{i}$.

\begin{remark}
	The conditions \ref{cond: P1}--\ref{cond: P3} are easily satisfied.
	\begin{itemize}\itemsep0.25cm
		\item Kirillov--Reshetikhin modules are known to satisfy \ref{cond: P1}.
		\item The existence and uniqueness of a unitary trigonometric R-matrix on a tensor product
		of finite--dimensional irreducible $\UqLg$--modules is guaranteed by Theorem~\ref{thm:rational-R}. Thus, \ref{cond: P2} is automatically satisfied
		(see also Remark~\ref{rmk:rational-R} (2)).
		\item The condition \ref{cond: P3} is related to the explicit computation of the denominators $\dR{ij}(z)$. Unfortunately, this is a difficult, largely open problem for arbitrary irreducible modules.
		However, it is well--known that \ref{cond: P3} is satisfied by \emph{good} modules, as defined by Kashiwara in \cite{kashiwara-02}, \ie irreducible finite--dimensional modules over $\UqLg$ with a bar involution, a crystal basis with a simple crystal graph, and a global basis.  \rmkend
	\end{itemize}
\end{remark}

\subsection{Enhanced $J$--data}\label{ss:combinatorics}
An \emph{enhanced} $J$--datum is a $J$--datum $(\qV{i}, \qX{i})_{i\in J}$
equipped with an involution $\qT:J\to J$ such that the following properties are satisfied. Set $\qJT=\{i\in J\,\vert\, \qT(i)=i\}$.
	\vspace{0.25cm}
\begin{enumerate}[label=(Q\arabic*), start=0] \itemsep0.25cm
	\item \label{cond: Q0} {\bf Symmetry}. For any $i\in\qJ$, we have $\qX{\qT(i)} = \qX{i}^{-1}$, and there is an isomorphism of $\Uqk$--modules $\qV{\qT(i)} \simeq \psi^*(\qV{i})$.
	\item \label{cond: Q1}{\bf Admissibility}.  For any $i\in\qJ$, there is an isomorphism of $\Uqk$--modules $(\psi^2)^*(\qV{i})\simeq\qV{i}$. 
	\item {\bf $\qT$-invariance}. \label{cond: Q2} For any $i\in\qJT$, we have $\qX{i}^2=1$,  and the module $\qV{i}$
	is \emph{QSP irreducible} (\ie it is irreducible under restriction to $\Uqk$, cf. \S\ref{ss:rational-k}), and there is an isomorphism of $\Uqk$--modules $\psi^*(\qV{i})\simeq\qV{i}$.
	\item {\bf K-matrices}\label{cond: Q3}. For any $i\in\qJ$, there is a non-vanishing \emph{unitary} K-matrix, \ie a QSP intertwiner 
	\[ \rKM{\qV{i}}{z}:\zqV{i}{z}\to\qV{\qT(i)}_{1/z}\simeq\psi^*(\qV{i})_{1/z}\,,\]
	satisfying Cherednik's reflection equation \eqref{eq:rational-tw-RE}
	and the unitarity condition \eqref{eq:unitarity-k}
	(cf.~Theorem~\ref{thm:unitary-k}).
	Moreover, if $i\in\qJT$, then $\rKM{\qV{i}}{1}=\id$.
	\item {\bf Poles}. \label{cond: Q4} For any $i\in J$, the poles of the K-matrix $\rKM{\qV{i}}{z}$ are 
	in $q^{1/m}\fml{\bbC}{q^{1/m}}$ for some $m>0$.
\end{enumerate} 

For any $i,j\in J$, we denote by $\dK{i}(z)\in\bsF[z]$ the denominator of the trigonometric K-matrix $\rKM{\qV{i}}{z}$ and by $\dK{i}\in\bbZ_{\geqslant0}$ 
the order of the pole of $\rKM{\qV{i}}{z}$ at $z=\qX{i}$.

\begin{remark}\label{rmk:comb-model}
	As before, the conditions \ref{cond: Q0}--\ref{cond: Q3} are easily satisfied.
	\begin{itemize}\itemsep0.25cm
		\item It is useful to observe that the condition \ref{cond: Q0} can be  satisfied \emph{by construction}. Namely, given an index set $J$ with
		an involution $\theta$, choose a representative for any $\qT$-orbit and
		consider the corresponding partition $\qJ=\qJp\sqcup\qJT\sqcup\qJm$
		with $\theta(\qJp)=\qJm$. Then, choose, for any $i\in(\qJp\sqcup\qJT)$, a finite-dimensional irreducible real
		$\UqLg$-module $\qV{i}$ and a non-zero scalar $\qX{i}\in\cor^{\times}$. 
		Finally, for any $\qT(i)\in\qJm$, set
		$\qV{\qT(i)}\coloneqq\psi^*(\qV{i})$ and $\qX{\qT(i)}\coloneqq\qX{i}^{-1}$.
		\item The requirement $(\psi^2)^*(V)\simeq V$ in \ref{cond: Q1} 
		can easily be overcome, since it is always possible to choose 
		involutive twisting operator $\psi$, \ie $\psi=\omega\circ\tau$, see~\cite[Ex. 3.6.3 (2)]{appel-vlaar-22}. With this choice, Kirillov--Reshetikhin modules satisfy also the condition $\psi^*(V)\simeq V$ in \ref{cond: Q2} (up to a shift, see \cite[Thm.~7.8.1]{appel-vlaar-22}). Note also that \ref{cond: Q1}  follows automatically
		whenever \ref{cond: Q0} is satisfied.
		\item By Theorem~\ref{thm:unitary-k}, the condition $(\psi^2)^*(V)\simeq V$ guarantees the existence of a unitary K-matrix. Thus, \ref{cond: Q3} is automatically satisfied.
	\end{itemize}
	The condition \ref{cond: Q4} is, however, harder to verify, since the properties of the poles of trigonometric K-matrices 
	are at the moment largely unknown. \rmkend
\end{remark}

\subsection{Enhanced $J$--quivers}

In \cite[\S 3]{kang-kashiwara-kim-18}, Kang, Kashiwara, and Kim 
defined a quiver attached to a $J$--datum, which we refer to as the associated $J$--quiver. We extend their
construction to a quiver with a framing and a contravariant involution, 
naturally associated to an enhanced $J$--datum.

\begin{definition}\label{def:quiver}
	\hfill
	\begin{enumerate}\itemsep0.25cm
	\item 
	Let $(\qV{i}, \qX{i})_{i\in J}$ be a $J$--datum. 
	The associated \emph{$J$--quiver} $\Gamma$ (c.f.\ \cite[\S 3]{kang-kashiwara-kim-18}) is the quiver defined as follows:
	\begin{itemize}\itemsep0.25cm
		\item the set of vertices is $J$;
		\item for any $i,j\in J$, there are $\dR{ij}$ arrows from $i$ to $j$.
	\end{itemize}
	\item 
	Let $(\qV{i}, \qX{i})_{i\in J}$ be an enhanced $J$--datum
	with involution $\qT$. The associated \emph{enhanced} $J$--quiver 
	is the $J$--quiver associated to the underlying $J$--datum, additionally 
	equipped with the framing dimension vector ${\bm\lambda}\in\N[J]$
	given by ${\bm\lambda}(i)=\dK{\qT(i)}$.
	\end{enumerate}
\end{definition}


The enhanced $J$--quiver has several convenient properties, which are easily proved.

\begin{prop}\label{prop:J-quiver}
	Let $\qG$ be an enhanced $J$--quiver. Then the following holds.
	\begin{enumerate}\itemsep0.25cm
		\item The quiver $\qG$ has neither loops nor cycles.
		\item If ${\bm\lambda}(i) \neq 0$ then ${\bm\lambda}(\theta(i)) = 0$. In particular, ${\bm\lambda}(i)=0$
		for $i\in\qJT$.
		\item The involution $\qT$ on $J$ lifts to a contravariant involution
		of the quiver $\qG$.
	\end{enumerate}
\end{prop}

\begin{proof}
	Part (1) is proved as in \cite{kang-kashiwara-kim-18}. 
	Namely, since $\qV{i}$ is real, $\rRM{\qV{i}\qV{i}}{z}$ has no pole
	at $z=1$, thus $\dR{ii}=0$. The condition \ref{cond: P3} then guarantees that no cycle can appear.
	Part (2) follows similarly from the unitarity
	condition on the K-matrix \ref{cond: Q3}. 
	For part (3), it is enough to observe that, by definition of the twisting operator
	$\psi$, see \cite[\S2]{appel-vlaar-20}, one has
	\begin{equation}
		(\rRM{\psi^*(\qV{i})\psi^*(\qV{j})}{w/z})_{21}=(F^{-1}_{\qV{i}\qV{j}})_{21}\circ\rRM{\qV{j}\qV{i}}{z/w}\circ F_{\qV{j}\qV{i}}
	\end{equation}
	where $F$ is a (constant) Drinfeld twist in $\UqLg$. Therefore,
	$\dR{\qT(i)\qT(j)}=\dR{ji}$ and the result follows.
\end{proof}

%
	
\subsection{Enhanced $J$--quivers of Dynkin type}\label{ss:ex-prelim}
It is well--known that every quiver of Dynkin type can be realized as the $J$--quiver
associated to a $J$--datum for a Lie algebra $\mathfrak{g}$ of the same type (see, \eg \cite{kang-kashiwara-kim-15, fujita-20, fujita-22, naoi}). 
The analogue result for enhanced $J$--quiver is much more restrictive. 
By Proposition~\ref{prop:J-quiver} (3), every enhanced $J$--quiver is naturally
equipped with a contravariant involution $\qT$. However, the latter exists only 
in the cases of Dynkin quivers of type $\sfA$, and affine Dynkin quivers of type 
$\sfA$ and $\sfD$. 

In \S \ref{ex:1}--\ref{ex:3}, we provide an explicit realisation as enhanced $J$--quivers of the following three examples: the linearly oriented quiver
of type $\sfA$; the bipartite quiver of type $\sfA$; the linearly oriented quiver of affine type $\sfD$. 

We briefly summarize our approach. In analogy with \cite{kang-kashiwara-kim-18, kang-kashiwara-kim-15}, 
the $J$--datum $(\qV{i}, \qX{i})_{i\in J}$ is given entirely in terms of the {fundamental representations}  $\sfV_{\omega_i}$ ($i\in I$) of the quantum affine algebra $\UqLg$ of the same type.\footnote{We shall consider in fact only \emph{small} fundamental representations, \ie
	irreducible under restriction to $\Uqg$. In particular, these are isomorphic to their finite type counterparts as $\Uqg$--modules.} 
The definition of the enhanced $J$--datum is quite subtle. First, we equip $J$ with the unique contravariant involution $\qT$. Then, we carefully choose the QSP subalgebra in order to satisfy the condition \ref{cond: Q0}. We proceed as follows.
\begin{itemize}\itemsep0.25cm 
	\item We fix a $\tau$--restrictable QSP subalgebra $\Uqk\subset\UqLg$ with twisting operator $\psi=\omega\circ\tau$ (cf.~\S\ref{ss:K-mx-KR}). By Theorem~\ref{thm:kr-k} (1), for any $i\in J$, we obtain a trigonometric K-matrix 
	\[
	\qV{i}(z)\to(\eta_0\tau)^*(\qV{i})(z^{-1})
	\]
	where $\eta_0$ is the extension of the opposition involution $\oi_I$.
	
	\item In type $\sfA$, the contravariant involution $\qT$ satisfies 
	$\qV{\qT(i)}=\qV{i}$. Therefore, if $\tau=\eta_0$, the symmetry condition
	\ref{cond: Q0} is automatically satisfied
	\footnote{Recall that \ref{cond: Q0} requires that, for any $i\in J$, $\qV{\qT(i)}\simeq\psi^*(\qV{i})$.}. On the other hand, if $\tau=\id$, 
	the same condition fails, since $\eta_0\neq\id$ and 
	$\eta_0^*(\qV{i})\not\simeq\qV{i}$ \footnote{This corresponds to the fact that
	the fundamental representations are not self--dual in type $\sfA$.}.
	
	\item In affine type $\widehat{\sfD}_N$, we consider a $J$--datum involving only the vector representation $\sfV_{\omega_1}$ and the two spin representations $\sfV_{\omega_{N-1}}$ and $\sfV_{\omega_{N}}$. As before, the contravariant 
	involution $\qT$ satisfies $\qV{\qT(i)}=\qV{i}$. When $N$ is even, $\eta_0=\id$ and we consider $\tau$--restrictable QSP subalgebras with $\tau=\id=\eta_0$. Thus, \ref{cond: Q0} is automatically satisfied in this case. On the other hand, when $N$ is odd, $\eta_0\neq\id$ and one has $\eta_0^*(\sfV_{\omega_1})\simeq\sfV_{\omega_1}$, but $\eta_0^*(\sfV_{\omega_{N-1}})\simeq\sfV_{\omega_N}$. Thus, in order to satisfy \ref{cond: Q0}, in this case we can only consider QSP subalgebras with $\tau=\eta_0$.
	
	\item In affine type $\widehat{\sfD}_N$ with $N$ odd, one can consider a simple
	modification of the $J$--datum such that the two spin representations correspond to each other through $\qT$. By the same argument, \ref{cond: Q0}
	holds for any QSP subalgebra with $\tau=\id$
	(cf.~Remark~\ref{rmk:DN-odd}).
\end{itemize}

Finally, the enhanced $J$--datum is obtained by $\qT$ together with the 
trigonometric K-matrices given by Theorem~\ref{thm:kr-k} with respect to
the QSP subalgebra chosen as above.

\begin{remark}
In the examples \S\ref{ex:2} and \S\ref{ex:3} below, we do not include a description
of the framing\footnote{For the example \S\ref{ex:1}, the framing is described in \S\ref{ss:AIII-functor} together with an explicit expression of the trigonometric K-matrix for the fundamental representation in type $\mathsf{AIII}$.}. This is due to the fact that, with the exception of the first 
fundamental representation, an explicit expression for the poles of trigonometric K-matrices is at the moment largely unknown. Even though some progress towards more general Kirillov--Reshetikhin modules has recently been made in \cite{kusano-okado-watanabe-22} from the point of view of crystal combinatorics, no explicit formula for the poles has yet been derived. Nevertheless, the poles are expected to
depend rationally on the parameters $(\Parc,\Pars)$ defining the QSP subalgebra
(see, \eg \cite{regelskis-vlaar-16} or also \S\ref{ss:AIII-functor}).
In this case, the condition \ref{cond: Q4} would always be satisfied for suitable choices of parameters, yielding a generically trivial framing. 
\rmkend
\end{remark}

\subsection{Example: the linearly oriented quiver of type $\mathsf{A}$}\label{ex:1}
This example is discussed in details in \S\ref{ss:AIII-functor} in the case of a (not necessarily $\tau$--restrictable) quasi--split QSP subalgebra of type $\mathsf{AIII}$. 
We provide here a brief summary in greater generality.

Let $\g=\mathfrak{sl}_{N+1}$ and $\sfV=\sfV_{\omega_1}$. 
Note that, as a $\Uqg$--module, $\sfV$ is isomorphic to the vector represententation. Set $J\coloneqq \bbZ_{\scsop{odd}}$ and consider the 
involution 
$\qT$ on $J$ given by $\qT(n)=-n$.\footnote{In particular, $\qJT=\emptyset$. The case with $J=\bbZ_{\scsop{even}}$ and $\qJT=\{0\}$ is analogous, so we omit it.}
For any $i\in J$, set 
\[\qV{n}\coloneqq\frep{}\qquad\mbox{and}\qquad \qX{n}=q^{n}\,.\]
First, we observed that the conditions \ref{cond: P1}--\ref{cond: P3} are clearly satisfied.
In particular, \ref{cond: P3} holds, since the R--matrix $\rRMv{\frep{}\frep{}}{z,w}$ is known to have only one simple pole at $w/z=q^2$, see \eqref{eq: R mat fund1}.
The resulting $J$--quiver with $\qT$ is in Figure~\ref{fig:Ainf}, where we use the notation
$(\omega_i, p)$ to indicate $(\sfV_{\omega_i}, p)$. This is the main example considered in \cite{kang-kashiwara-kim-18}.\\

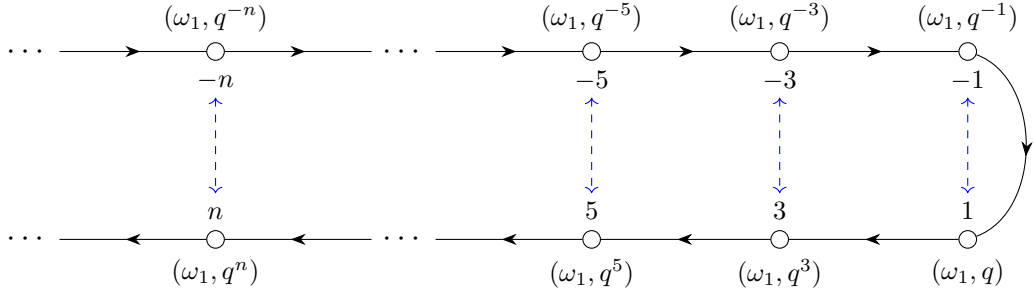
\begin{figure}
	\begin{center} 
		\begin{tikzpicture}[decoration={markings, mark= at position 0.55 with {\arrow{Stealth[length=2mm]}}}]
			
			\node (NN) at (-5,1.25) [draw, circle, fill=white, scale=0.6, label=below:{\footnotesize$-n$}, label=above:{\footnotesize$(\omega_1,q^{-n})$}] {}; 
			
			
			\node (N1) at (0,1.25) [draw, circle, fill=white, scale=0.6, label=below:{\footnotesize$-5$}, label=above:{\footnotesize$(\omega_1,q^{-5})$}] {}; 
			
			\node (N2) at (2.5,1.25) [draw, circle, fill=white, scale=0.6, label=below:{\footnotesize$-3$}, label=above:{\footnotesize$(\omega_1,q^{-3})$}] {}; 
			\node (N3) at (5,1.25) [draw, circle, fill=white, scale=0.6, label=below:{\footnotesize$-1$}, label=above:{\footnotesize$(\omega_1,q^{-1})$}] {}; 
			\node(N4) at (5,-1.25) [draw, circle, fill=white, scale=0.6, label=above:{\footnotesize$1$}, label=below:{\footnotesize$(\omega_{1},q)$}] {};
			\node (N5) at (2.5,-1.25) [draw, circle, fill=white, scale=0.6, label=above:{\footnotesize$3$}, label=below:{\footnotesize$(\omega_{1},q^3)$}] {};
			
			\node (N6) at (0,-1.25) [draw, circle, fill=white, scale=0.6, label=above:{\footnotesize$5$}, label=below:{\footnotesize$(\omega_{1},q^5)$}] {};
			
			
			\node (NX) at (-5,-1.25) [draw, circle, fill=white, scale=0.6, label=above:{\footnotesize$n$}, label=below:{\footnotesize$(\omega_{1},q^n)$}] {};
			
			\node (D0) at (-7.5,1.25) {$\cdots$};
			\draw[postaction={decorate}] (D0)--(NN);
			\node (DX) at (-7.5,-1.25) {$\cdots$};
			\draw[postaction={decorate}] (NX)--(DX);
			
			\node (D1) at (-2.5,1.25) {$\cdots$};
			\draw[postaction={decorate}] (NN)--(D1);
			\node (D2) at (-2.5,-1.25) {$\cdots$};
			\draw[postaction={decorate}] (D2)--(NX);
			
			\draw[postaction={decorate}] (D1)--(N1);
			\draw[postaction={decorate}] (N1)--(N2);
			\draw[postaction={decorate}] (N2)--(N3);
			\draw (N3) edge[postaction={decorate}, bend left=70] (N4); 
			\draw[postaction={decorate}] (N4)--(N5); 
			\draw[postaction={decorate}] (N5)--(N6); 
			\draw[postaction={decorate}] (N6)--(D2);

			\draw (N1) edge[blue, dashed, shorten >=14pt, shorten <=14pt, <->] (N6);
			\draw (N2) edge[blue, dashed, shorten >=14pt, shorten <=14pt, <->] (N5);
			\draw (N3) edge[blue, dashed, shorten >=14pt, shorten <=14pt, <->] (N4);
			\draw (NN) edge[blue, dashed, shorten >=14pt, shorten <=14pt, <->] (NX);
			
		\end{tikzpicture}
	\end{center}
	\caption{The case of the linearly oriented $\sfA_\infty$ quiver.}\label{fig:Ainf}
\end{figure}

We fix a QSP subalgebra $\Uqk\subset\UqLg$ with $\tau=\eta_0$ and we consider the twisting operator $\psi=\omega\circ\tau$. Then, by Theorem~\ref{thm:kr-k},
the conditions \ref{cond: Q0}, \ref{cond: Q1}, \ref{cond: Q2} and \ref{cond: Q3} 
are automatically satisfied. 
It remains to verify the condition \ref{cond: Q4}. 
In \cite{regelskis-vlaar-18}, Regelskis and Vlaar computed explicitly the trigonometric K--matrices for $\sfV$. Therefore, 
\ref{cond: Q4} holds, provided the QSP parameters are carefully chosen, as outlined in \S\ref{ss:AIII-functor}.
%
%

\subsection{Example: the bipartite quiver of type $\mathsf{A}$}\label{ex:2}

We shall now discuss the enhancement of an example from \cite{kang-kashiwara-kim-15} involving multiple fundamental representations.
Let $\mathfrak{g} = \mathfrak{sl}_{N+1}$. Let $J = \{1, \cdots, N\}$ be a set with 
the involution $\theta(i) = N+1-i$. We consider $J$ as the set of vertices of a quiver $Q$ of type $\mathsf{A}_N$, with an edge between $i$ and $j$ if and only if $|i-j| = 1$, and equipped with a bipartite (\ie alternating) orientation. 

Let $\Uqk\subset\UqLg$ be a QSP subalgebra with $\tau=\eta_0$ and twisting operator $\psi=\tau\circ\omega$. We consider the corresponding trigonometric K-matrices on the fundamental representations $\sfV_{\omega_i}$ ($i\in I$) given by Theorem~\ref{thm:kr-k}. 

\Prop
Set $p=-q$. 
\begin{enumerate}
\item The quiver $Q$ admits a contravariant involution if and only if $N$ is even. 
\item Let $N = 2m$ and set, for any $i \in J$, 
\begin{align*} 
V(i) &= \left\{ \begin{array}{ll}
\frep{\omega_i} & \mbox{if } i \mbox{ is a source in } Q \\
\frep{\omega_{\theta(i)}} & \mbox{if } i \mbox{ is a sink in } Q \\
\end{array} \right. \\ 
X(i) &=  \left\{ \begin{array}{ll}
p^{m} & \mbox{if } i \mbox{ is a source in } Q \\
p^{-m} & \mbox{if } i \mbox{ is a sink in } Q \\
\end{array} \right.  
\end{align*} 
If the condition \ref{cond: Q4} is satisfied, the datum $(V(i), X(i))_{i\in J}$
defines an enhanced $J$-datum with respect to the involution $\qT$. The resulting $J$-quiver $\Gamma$ is isomorphic to~$Q^{\textrm{rev}}$, \ie the quiver obtained from $Q$ by reversing all the arrows, see Figure~\ref{fig:Abip}. 
\end{enumerate}

\usetikzlibrary{arrows.meta}
\usetikzlibrary{ decorations.markings}

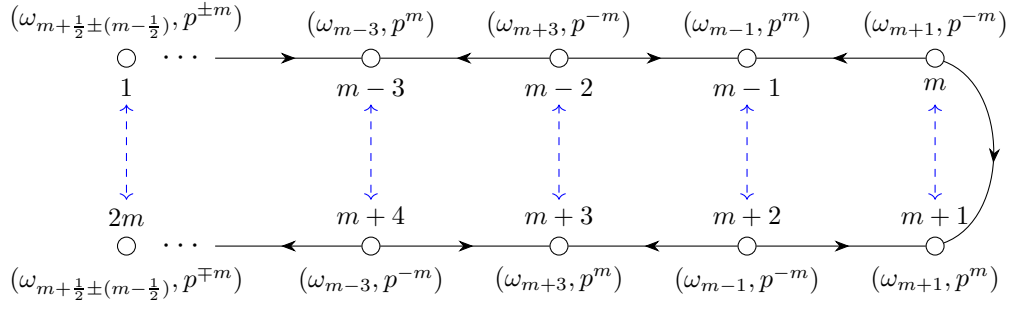
\begin{figure}
\begin{center} 
\begin{tikzpicture}[decoration={markings, mark= at position 0.55 with {\arrow{Stealth[length=2mm]}}}]

\node (NN) at (-5.75,1.25) [draw, circle, fill=white, scale=0.6, label=below:{\footnotesize$1$}, label=above:{\footnotesize$(\omega_{m+\frac{1}{2}\pm(m-\frac{1}{2})},p^{\pm m})$}] {}; 

\node (N0) at (-2.5,1.25) [draw, circle, fill=white, scale=0.6, label=below:{\footnotesize$m-3$}, label=above:{\footnotesize$(\omega_{m-3},p^{m})$}] {}; 

\node (N1) at (0,1.25) [draw, circle, fill=white, scale=0.6, label=below:{\footnotesize$m-2$}, label=above:{\footnotesize$(\omega_{m+3},p^{-m})$}] {}; 

\node (N2) at (2.5,1.25) [draw, circle, fill=white, scale=0.6, label=below:{\footnotesize$m-1$}, label=above:{\footnotesize$(\omega_{m-1},p^{m})$}] {};
\node (N3) at (5,1.25) [draw, circle, fill=white, scale=0.6, label=below:{\footnotesize$m$}, label=above:{\footnotesize$(\omega_{m+1},p^{-m})$}] {};
\node(N4) at (5,-1.25) [draw, circle, fill=white, scale=0.6, label=above:{\footnotesize$m+1$}, label=below:{\footnotesize$(\omega_{m+1},p^{m})$}] {};
\node (N5) at (2.5,-1.25) [draw, circle, fill=white, scale=0.6, label=above:{\footnotesize$m+2$}, label=below:{\footnotesize$(\omega_{m-1},p^{-m})$}] {};

\node (N6) at (0,-1.25) [draw, circle, fill=white, scale=0.6, label=above:{\footnotesize$m+3$}, label=below:{\footnotesize$(\omega_{m+3},p^{m})$}] {};

\node (N7) at (-2.5,-1.25) [draw, circle, fill=white, scale=0.6, label=above:{\footnotesize$m+4$}, label=below:{\footnotesize$(\omega_{m-3},p^{-m})$}] {};

\node (NX) at (-5.75,-1.25) [draw, circle, fill=white, scale=0.6, label=above:{\footnotesize$2m$}, label=below:{\footnotesize$(\omega_{m+\frac{1}{2}\pm(m-\frac{1}{2})},p^{\mp m})$}] {};

\draw[postaction={decorate}] (N1)--(N0);
\draw[postaction={decorate}] (N1)--(N2);
\draw[postaction={decorate}] (N3)--(N2);
\draw (N3) edge[postaction={decorate}, bend left=70] (N4); 
\draw[postaction={decorate}] (N5)--(N4); 
\draw[postaction={decorate}] (N5)--(N6); 
\draw[postaction={decorate}] (N7)--(N6);

\node (D0) at (-5,1.25) {$\cdots$};
\draw[postaction={decorate}] (D0)--(N0);
\node (DX) at (-5,-1.25) {$\cdots$};
\draw[postaction={decorate}] (N7)--(DX);

\draw (N0) edge[blue, dashed, shorten >=14pt, shorten <=14pt, <->] (N7);
\draw (N1) edge[blue, dashed, shorten >=14pt, shorten <=14pt, <->] (N6);
\draw (N2) edge[blue, dashed, shorten >=14pt, shorten <=14pt, <->] (N5);
\draw (N3) edge[blue, dashed, shorten >=14pt, shorten <=14pt, <->] (N4);
\draw (NN) edge[blue, dashed, shorten >=14pt, shorten <=14pt, <->] (NX);

\end{tikzpicture}
\end{center}
\caption{The case of the bipartite $\sfA_{2m}$ quiver.}\label{fig:Abip}
\end{figure}
\enprop

\begin{proof} 
%
(1) The existence of an involution when $N=2m$ is obvious from Figure~\ref{fig:Abip}.
Instead, when $N=2m+1$, it is enough to observe that the central vertex $m+1$ is bound 
to be fixed by the involution. Since $m+1$ is supposed to be either a source or a sink, the 
involution cannot be contravariant.

(2) The proof relies on some results from \cite{kang-kashiwara-kim-15}. 
Let $w$ be a Coxeter element adapted to $Q$. Without loss of generality, let us assume that $m$ is odd, and the central arrow sends $m+1$ to $m$. In that case even vertices are sources and odd vertices are sinks, and 
\[ w = s_2s_4 \cdots s_{2m} s_1 s_3 \cdots s_{2m-1}.\] 
Choose the height function 
\[
\xi_i = \left\{ \begin{array}{ll}
m+1 & \mbox{if } i \mbox{ is odd,} \\
m & \mbox{if } i \mbox{ is even.}
\end{array} \right.
\]

Let $\Pi_0$ and $\Delta_+$ denote the sets of simple and positive roots of $\mathfrak{g}$, respectively. 
 By \cite[Thm.~4.3.1]{kang-kashiwara-kim-15}, there exists a $J$-datum such that the corresponding $J$-quiver $\Gamma$ is isomorphic to $Q^{rev}$. To compute this datum, we need to consider the repetition quiver of $Q$, and a certain bijection $\phi \colon \widehat{I}_0 \to \widehat{\Delta}$, where $\widehat{I}_0 = \{1, 3, \cdots, m\} \times 2\Z \sqcup \{2,4,\cdots,2m\} \times \Z_{\text{odd}}$ and $\widehat{\Delta} = \Delta_+ \times \Z$. The bijection is defined recursively according to the rules in \cite[\S 3.1]{kang-kashiwara-kim-15}. To determine the $J$-datum, we need to find the preimage of $\Pi_0 \times \{0\}$ under $\phi$. 

For any source vertex $i$, we have, by definition, $\phi(i, m) = (\alpha_i, 0)$. We also claim that $w^m(\alpha_i) = \alpha_{\theta(i)}$, so that, by the recursive rule in the definition of $\phi$, we have $\phi(i, -m) = (\alpha_{\theta(i)}, 0)$. Then, 
we have
\[
V(i) = \operatorname{pr}_1 \circ \phi^{-1}(\alpha_i,0), \qquad 
X(i) = \operatorname{pr}_2 \circ \phi^{-1}(\alpha_i,0),
\]
which, according to \cite[Thm.~4.3.1]{kang-kashiwara-kim-15}, defines the desired $J$-datum.

It remains to prove the claim. Write $i=2j$ and assume, without loss of generality, that $i < m$. An easy calculation shows that $w^j(\alpha_i) = \alpha_1 + \cdots + \alpha_{2i}$, and that the subsequent $m-i$ applications of $w$ shift the string of simple roots by $2(m-i)$ so that $w^{m-j}(\alpha_i) = \theta(w^j(\alpha_i))$. The subsequent $j$ applications of $w$ shrink the string to $\alpha_{\theta(i)}$. 

From Theorem~\ref{thm:kr-k}, the conditions \ref{cond: Q0}--\ref{cond: Q3} are automatically satisfied (see also \S\ref{ss:ex-prelim}).
\end{proof}

\subsection{Example: the quiver of affine type $\mathsf{D}$}\label{ex:3}
We conclude with the example of a quiver of type other than $\sfA$.
Let $\g=\mathfrak{so}_{2N}$. As before, let $\Uqk\subset\UqLg$ be a QSP subalgebra with $\tau=\eta_0$, and set $\psi=\tau\circ\omega$.
We shall construct an enhanced $J$--datum relying on the fundamental representations $\sfV_{\omega_1}$, $\sfV_{\omega_{N-1}}$, $\sfV_{\omega_{N}}$,
where the vertices of the affine Dynkin diagram of $\g$ are numbered as in 
Figure~\ref{fig:Daff}. Note that all three modules are small, and, as
$\Uqg$--modules, are isomorphic to the vector representation and the spin representations, respectively.

Set $J=\{0,1,\dots,N\}$ and let  $\qT$ be the involution on $J$ given by
$\qT(i)=N-i$. We consider the assignments
\begin{align*}
	V(i)=
	\begin{cases}
	\sfV_{\omega_{N-i}} & \mbox{if } i=0, 1\\
	\sfV_{\omega_i} & \mbox{if } i=N-1, N\\
	\sfV_{\omega_1} & \mbox{if } i=2, \dots, N-2
	\end{cases}
\end{align*}
and
\begin{align*}
	X(i)=
	\begin{cases}
		(-1)^Nq^{-2(N-2)} & \mbox{if } i=0, 1\\
		(-1)^Nq^{2(N-2)} & \mbox{if } i=N-1, N\\
		q^{2i-N} & \mbox{if } i=2, \dots, N-2
	\end{cases}
\end{align*}

\begin{prop}
	If the condition \ref{cond: Q4} is satisfied, the datum $(V(i), X(i))_{i\in J}$
	defines an enhanced $J$-datum with respect to the involution $\qT$. The resulting $J$-quiver is represented in Figure~\ref{fig:Daff} for $N$ odd, and in Figure~\ref{fig:Daff2} for 
	$N$ even. 
\end{prop}

\begin{proof}
The proof relies entirely on the explicit computation of the denominators of the
R--matrix between all fundamental representations in affine type $\sfD$, provided in 
\cite[Thm.~A.1.1]{kang-kashiwara-kim-15}. In particular, one has that 
\begin{align*}
	d_{11}(z)&=(z-q^2)(z-q^{2N-2})\,, \qquad\, d_{(N-1)(N-1)}(z)=\,\,\,\prod_{s=1}^{\lfloor \frac{N}{2}\rfloor}(z-(-q)^{4s-2})=d_{NN}(z)\,,\\
	d_{1N}(z)&=z-(-q)^N=d_{1(N-1)}(z)\,, \qquad
	d_{(N-1)N}(z)=\prod_{s=1}^{\lfloor \frac{N-1}{2}\rfloor}(z-(-q)^{4s})\,.
\end{align*}
We immediately observe that, by restriction to the vertices $J'=\{2,\dots, N-2\}$, 
the datum $(\qV{i}, \qX{i})_{i\in J'}$ gives rise to a linearly oriented quiver of 
type $\sfA_{N-3}$, as in the finite analogue of Figure~\ref{fig:Ainf}. Then, by the denominators formulae above, the spin representations
placed at the four extremal vertices are shown to a single arrow each, towards 
either incoming to $\qV{2}$ or outcoming from $\qV{N-2}$.
The resulting $J$--quiver are represented in Figures~\ref{fig:Daff} and \ref{fig:Daff2} when
$N$ is odd or even, respectively. As before, one finally observes that,
provided the condition \ref{cond: Q4} is satisfied, these assignments
give rise to an enhanced $J$--datum with respect to $\qT$.
\end{proof}

\begin{remark}\label{rmk:DN-odd}
	As mentioned in \S\ref{ss:ex-prelim}, 
	the case with $N$ odd is of further interest, since in this
	case the spin representations correspond to each other through $\eta_0$. 
	Let $\Uqk\subset\UqLg$ be a $\tau$--restrictable QSP subalgebra with $\tau=\id$ and set 
	\[
	\qT'(i)=\begin{cases}N-i & \mbox{if } i=2,\dots, N-2\\
		N-1-i & \mbox{if }  i=0, 1\\
		N+1-i & \mbox{if } i=N-1, N\end{cases}
	\]
	Then, it follows immediately from Theorem~\ref{thm:kr-k} that, provided the condition \ref{cond: Q4} is satisfied, the previous assignments give rise to an enhanced $J$--datum
	with respect to $\qT'$.
	\rmkend
\end{remark}

\begin{figure}
	\begin{center} 
		\begin{tikzpicture}[decoration={markings, mark= at position 0.55 with {\arrow{Stealth[length=2mm]}}}]
			
			\node (NN) at (-6.75,2) [draw, circle, fill=white, scale=0.6, label=below:{\footnotesize$0$}, label=above:{\footnotesize$(\omega_{2m+1}, -q^{-4m+2})$}] {}; 
			
			\node (NN1) at (-5.25,0.5) [draw, circle, fill=white, scale=0.6, label=below:{\footnotesize$1$}, label=above:{\footnotesize$(\omega_{2m}, -q^{-4m+2})$}] {}; 
			
			\node (N0) at (-2.5,1.25) [draw, circle, fill=white, scale=0.6, label=below:{\footnotesize$2$}, label=above:{\footnotesize$(\omega_1,q^{-2m+3})$}] {}; 
			
			\node (N1) at (0,1.25)  {$\cdots$}; 
			
			\node (N2) at (2.5,1.25) [draw, circle, fill=white, scale=0.6, label=below:{\footnotesize$m-1$}, label=above:{\footnotesize$(\omega_1,q^{-3})$}] {}; 
			\node (N3) at (5,1.25) [draw, circle, fill=white, scale=0.6, label=below:{\footnotesize$m$}, label=above:{\footnotesize$(\omega_1,q^{-1})$}] {}; 
			\node(N4) at (5,-1.25) [draw, circle, fill=white, scale=0.6, label=above:{\footnotesize$m+1$}, label=below:{\footnotesize$(\omega_{1},q^{})$}] {};
			\node (N5) at (2.5,-1.25) [draw, circle, fill=white, scale=0.6, label=above:{\footnotesize$m+2$}, label=below:{\footnotesize$(\omega_{1},q^{3})$}] {};
			
			\node (N6) at (0,-1.25) {$\cdots$};
			
						\node (N7) at (-2.5,-1.25) [draw, circle, fill=white, scale=0.6, label=above:{\footnotesize$2m-1$}, label=below:{\footnotesize$(\omega_{1},q^{2m-3})$}] {};
			
			\node (NX) at (-6.75,-0.5) [draw, circle, fill=white, scale=0.6, label=above:{\footnotesize$2m+1$}, label=below:{\footnotesize$(\omega_{2m+1},-q^{4m-2})$}] {};
			
			\node (NX1) at (-5.25,-2) [draw, circle, fill=white, scale=0.6, label=above:{\footnotesize$2m$}, label=below:{\footnotesize$(\omega_{2m},-q^{4m-2})$}] {};

			\draw[postaction={decorate}] (NN)--(N0);
			\draw[postaction={decorate}] (NN1)--(N0);
			\draw[postaction={decorate}] (N0)--(N1); 
			\draw[postaction={decorate}] (N1)--(N2);
			\draw[postaction={decorate}] (N2)--(N3);
			\draw (N3) edge[postaction={decorate}, bend left=70] (N4); 
			\draw[postaction={decorate}] (N4)--(N5); 
			\draw[postaction={decorate}] (N5)--(N6); 
			\draw[postaction={decorate}] (N6)--(N7); 
			
				\draw[postaction={decorate}] (N7)--(NX);
			\draw[postaction={decorate}] (N7)--(NX1);

			\draw (N0) edge[blue, dashed, shorten >=14pt, shorten <=14pt, <->] (N7);
			\draw (N2) edge[blue, dashed, shorten >=14pt, shorten <=14pt, <->] (N5);
			\draw (N3) edge[blue, dashed, shorten >=14pt, shorten <=14pt, <->] (N4);
			\draw (NN) edge[blue, dashed, shorten >=14pt, shorten <=14pt, <->] (NX);
			\draw (NN1) edge[blue, dashed, shorten >=14pt, shorten <=14pt, <->] (NX1);
		\end{tikzpicture}
	\end{center}
	\caption{The case of the linearly oriented affine $\sfD_{2m+1}$ quiver.}\label{fig:Daff}
\end{figure}
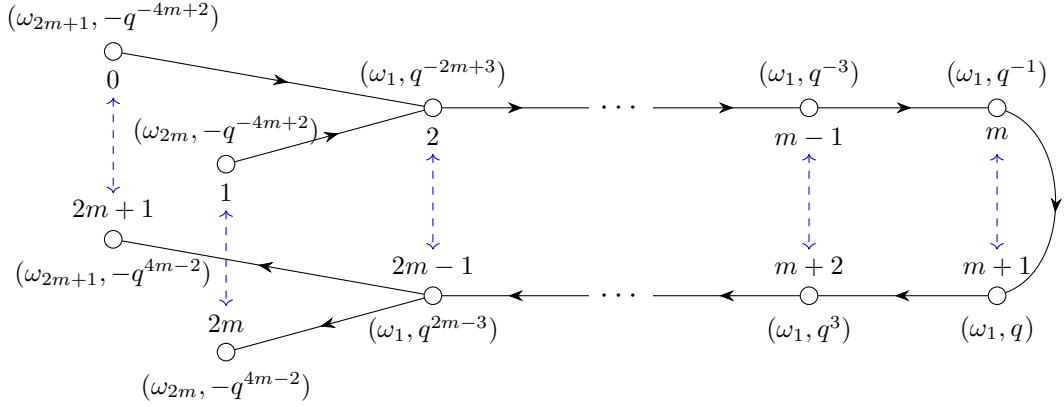

\begin{figure}
	\begin{center} 
		\begin{tikzpicture}[decoration={markings, mark= at position 0.55 with {\arrow{Stealth[length=2mm]}}}]
			
			\node (NN) at (-6.75,2) [draw, circle, fill=white, scale=0.6, label=below:{\footnotesize$0$}, label=above:{\footnotesize$(\omega_{2m}, q^{-4m+4})$}] {}; 
			
			\node (NN1) at (-5.25,0.5) [draw, circle, fill=white, scale=0.6, label=below:{\footnotesize$1$}, label=above:{\footnotesize$(\omega_{2m-1}, q^{-4m+4})$}] {}; 
			
			\node (N0) at (-2.5,1.25) [draw, circle, fill=white, scale=0.6, label=below:{\footnotesize$2$}, label=above:{\footnotesize$(\omega_1,q^{-2m+4})$}] {}; 
			
			\node (N1) at (0,1.25)  {$\cdots$}; 
			
			\node (N2) at (2.5,1.25) [draw, circle, fill=white, scale=0.6, label=below:{\footnotesize$m-2$}, label=above:{\footnotesize$(\omega_1,q^{-2})$}] {}; 
			\node (N3) at (5,1.25) [draw, circle, fill=white, scale=0.6, label=below:{\footnotesize$m-1$}, label=above:{\footnotesize$(\omega_1,q^{-2})$}] {}; 
			\node (ND) at (6.5,0) [draw, circle, fill=white, scale=0.6, label=left:{\footnotesize$m$}, label=right:{\footnotesize$(\omega_1,1)$}] {}; 
			\node(N4) at (5,-1.25) [draw, circle, fill=white, scale=0.6, label=above:{\footnotesize$m+1$}, label=below:{\footnotesize$(\omega_{1},q^{2})$}] {};
			\node (N5) at (2.5,-1.25) [draw, circle, fill=white, scale=0.6, label=above:{\footnotesize$m+2$}, label=below:{\footnotesize$(\omega_{1},q^{4})$}] {};
			
			\node (N6) at (0,-1.25) {$\cdots$};
			
			\node (N7) at (-2.5,-1.25) [draw, circle, fill=white, scale=0.6, label=above:{\footnotesize$2m-2$}, label=below:{\footnotesize$(\omega_{1},q^{2m-4})$}] {};
			
			\node (NX) at (-6.75,-0.5) [draw, circle, fill=white, scale=0.6, label=above:{\footnotesize$2m$}, label=below:{\footnotesize$(\omega_{2m},q^{4m-4})$}] {};
			
			\node (NX1) at (-5.25,-2) [draw, circle, fill=white, scale=0.6, label=above:{\footnotesize$2m-1$}, label=below:{\footnotesize$(\omega_{2m-1},q^{4m-4})$}] {};

			\draw[postaction={decorate}] (NN)--(N0);
			\draw[postaction={decorate}] (NN1)--(N0);
			\draw[postaction={decorate}] (N0)--(N1); 
			\draw[postaction={decorate}] (N1)--(N2);
			\draw[postaction={decorate}] (N2)--(N3);
			\draw (N3) edge[postaction={decorate}, bend left=40] (ND); 
			\draw (ND) edge[postaction={decorate}, bend left=40] (N4); 
			\draw[postaction={decorate}] (N4)--(N5); 
			\draw[postaction={decorate}] (N5)--(N6); 
			\draw[postaction={decorate}] (N6)--(N7); 
			
			\draw[postaction={decorate}] (N7)--(NX);
			\draw[postaction={decorate}] (N7)--(NX1);

			\draw (N0) edge[blue, dashed, shorten >=14pt, shorten <=14pt, <->] (N7);
			\draw (N2) edge[blue, dashed, shorten >=14pt, shorten <=14pt, <->] (N5);
			\draw (N3) edge[blue, dashed, shorten >=14pt, shorten <=14pt, <->] (N4);
			\draw (NN) edge[blue, dashed, shorten >=14pt, shorten <=14pt, <->] (NX);
			\draw (NN1) edge[blue, dashed, shorten >=14pt, shorten <=14pt, <->] (NX1);
		\end{tikzpicture}
	\end{center}
	\caption{The case of the linearly oriented affine $\sfD_{2m}$ quiver.}\label{fig:Daff2}
\end{figure}


\section{Boundary Schur--Weyl dualities} \label{s: boundary SW}   

\summary{In this section:
	\begin{itemize}
		\item \ref{ss:combinatorics}: the combinatorial model
		\item \ref{ss:boundary-SW}: generalized Schur-Weyl duality with boundary
		\item \ref{ss:KKK-compatibility}: compatibility with KKK
		\item \ref{ss:SW-properties}: basic properties
	\end{itemize}
}

In this section, we prove the main result of the paper, consisting in the construction of a {\em boundary} Schur-Weyl duality functor. We fix the following datum:
\begin{itemize}\itemsep0.25cm
	\item an affine Satake diagram $(X,\tau)$ with pseudo-involution $\tsat$\,;
	\item a QSP subalgebra $\Uqk\subseteq\UqLg$ with parameters $(\Parc,\Pars)$\,;
	\item a QSP admissible twisting operator $\psi\in\Aut(\UqLg)$\,;
	\item a set $J$ with an involution $\qT$\,;
	\item an enhanced $J$--datum $(V(i), X(i))_{i\in J}$\,; 
	\item the corresponding enhanced $J$--quiver $(\Gamma, \qT, \bm\lambda)$,
	see Definition~\ref{def:quiver}.
\end{itemize}
Henceforth we will consider $o$KLR algebras associated to  $(\Gamma, \qT, \bm\lambda)$ and varying dimension vectors $\beta \in \N[J]^\theta$. 
	Moreover, we fix the functions 
	\begin{equation} \label{eq: functor c-fun choice} c_{ij}(u,v) \in\fml{\cor}{u,v}, \quad c_i \in \fml{\cor}{u} \end{equation} 
	as in \eqref{eq: c-ij}--\eqref{eq: c-i}.

\subsection{Polynomial rings} \label{ss:field-completions}

As preparation for the main result, we need to introduce notation for various polynomial rings, their completions, and relate them to the polynomial rings that appeared earlier in \S\ref{ss:polynomial-rep} in the context of $o$KLR algebras. 
Let:   
	\[ \mathcal{O} := \cor[X_{\pm 1}, \cdots, X_{\pm n}], \qquad \mathcal{K} := \cor(X_1,\cdots, X_n),\]  
	where $X_{-l} = X_l^{-1}$. 
	The group $\weyl_n$ acts on $\mathcal{O}$ and $\mathcal{K}$ from the left by $w\cdot X_l = X_{w(l)}$. 

Given a self-dual dimension vector $\beta \in \N[J]^\theta$ with $\Norm{\beta}_\theta$, we also set
	\begin{alignat*}{8}
		    \Oh_{ \nu } :=&\ \fml{\cor}{X_1-X(\nu_1), \cdots, X_n-X(\nu_n)}, \qquad&   \tOhb \ \seteq& \soplus_{\nu \in \comp} \Oh_{ \nu }e(\nu), \\
		    \widehat{\mathcal{K}}_{ \nu } :=&\ \cor ((X_1-X(\nu_1), \cdots, X_n-X(\nu_n))), \qquad&  {}^\theta\widehat{\mathcal{K}}_\beta \ \seteq& \soplus_{\nu \in \comp} \widehat{\mathcal{K}}_{ \nu }e(\nu). 
	\end{alignat*}
	
	\begin{lemma}
		There is a $\weyl_n$-equivariant algebra isomorphism
		\begin{equation} \label{x-X map}
			\tbPh_{\beta} \ \isoto \ \tOhb, \qquad e(\nu) \mapsto e(\nu), \quad x_ke(\nu) \mapsto \left( \frac{X(\nu_k)}{X_k} - \frac{X_k}{X(\nu_k)} \right) e(\nu). 
		\end{equation}
	\end{lemma}
	
	\begin{proof}
		Observe that 
		\[  \bl X(\nu_k)X_k^{-1} - X(\nu_k)^{-1}X_k \br  = \mathsf{f}(1-X(\nu_k)^{-1}X_k), \]
		where
		\begin{equation} \label{eq: f series} \mathsf{f}(z) = z + \frac{z}{1-z} = 2z + \sum_{k \geqslant 2} z^k \in\fml{\cor}{z}. \end{equation} 
		Hence \eqref{x-X map} is well-defined. 
		Since the constant coefficient of $\mathsf{f}$ vanishes and the degree one coefficient is invertible, $\mathsf{f}$ has a composition inverse. Therefore, \eqref{x-X map} is an isomorphism. The equivariance is clear. 
	\end{proof} 

We identify $$\skewrr  \isoto \skewr$$ using \eqref{x-X map} and, following \eqref{eq: oklr loc}, we regard the oKLR algebra $\oklr$ as a subalgebra 
	of $\skewr$.

	\subsection{Boundary Schur-Weyl duality functor}\label{ss:boundary-SW}

We will now construct a $(\Uqk,\oklr)$-bimodule and the associated Schur-Weyl functor. 


	For each $\nu \in \comp$, set 
	$$
	V_\nu \coloneqq V(\nu_1)_{z_{\nu,1}} \otimes \cdots
	\otimes V(\nu_n)_{z_{\nu,n}} e(\nu).$$
	It is a right $\mathcal{O}\otimes \UqLg$-module,
	with $X_k$ acting as $z_{\nu, k}$.
	Set 
\begin{alignat}{8} \label{eq:Vhat}
\hV_\nu \coloneqq& \ \Oh_{\nu} \otimes _{\mathcal{O}} V_\nu, \qquad& \tbV[ \beta ] \coloneqq& \ \soplus_{\nu \in \comp} \hV_\nu, \\
\hVK_\nu \coloneqq& \ \widehat{\mathcal{K}}_{ \nu } \otimes _{\mathcal{O}} V_\nu, \qquad& \tbVK[ \beta ] \coloneqq& \ \soplus_{\nu \in \comp} \hVK_\nu. 
\end{alignat}
	
	\begin{theorem} \label{thm: bimodule}
		\hfill
		\begin{enumerate}\itemsep0.25cm
			\item 
			The space $\tbVK[ \beta ]$ has a natural structure of a $(\Uqk,\oklr)$-bimodule
			induced by the action of the trigonometric R- and K-matrices.
			\item
			The subspace $\tbV[{ \beta }] \subset \tbVK[ \beta ]$ is stable under
			the right action of the subalgebra $\oklr$ of $\skewr$.
			In particular, $\tbV[{ \beta }]$ has the structure of
			$(\Uqk,\oklr)$-bimodule.
			\item 	There is a right exact functor 
			\[ 
			\tfun_{ \beta } \colon \Modgr{\oklr} \longrightarrow \Mod{\Uqk},  \quad M \mapsto \tbV[{ \beta }]\otimes_{\oklr}M . 
			\] 
			Summing over all self-dual dimension vectors, we get 
			\begin{equation*} \label{eq: orbKKK}
				\tfun \coloneqq \soplus_{\beta \in \N[J]^\theta} \tfun_\beta : \ \Modgr{{}^\theta \mathcal{R}({\bm\lambda})}
				\longrightarrow \Mod{\Uqk}.
			\end{equation*} 
		\end{enumerate}
	\end{theorem}
	
	\begin{proof}
		(1) We first endow $\tbVK[ \beta ]$ with the structure of a right $\skewr$-module. 
		For each $\nu \in \comp$ and $k=1,\cdots, n-1$, let 
		$\Rm^{\nu}_{k} \colon \hVK_\nu \rightarrow \ 
		\hVK_{s_k(\nu)}$ be the intertwiner defined as follows. 
		We first consider the $\bsF$--linear map $V_\nu \rightarrow \ 
		V_{s_k(\nu)}$ given by the composition of the operator
				\begin{align*}
			v_1 \otimes \cdots \otimes v_n
			\mapsto& \ v_1 \otimes \cdots \otimes
			\rRMv{V(\nu_k),V(\nu_{k+1})}{z_{\nu, k+1}/z_{\nu,k}}
			\cdot(v_k \otimes v_{k+1}) \otimes \cdots \otimes v_n
		\end{align*} 
		with the map identifying $z_{\nu, k}$ (resp. $z_{\nu, k+1}$) 
		with $z_{s_k(\nu), k+1}$ (resp. $z_{s_k(\nu), k}$) and $z_{\nu, \ell}$ with 
		$z_{s_k(\nu), \ell}$ for $\ell\neq k,k+1$. We then
		extend it by $\mathcal{O}$--linearity.
		\Omit{
		 given by
		\begin{align*}
			v_1 \otimes \cdots \otimes v_n
			\mapsto& \ v_1 \otimes \cdots \otimes
			\rRMv{V(\nu_k),V(\nu_{k+1})}{X_{k+1}/X_k}
			 \cdot(v_k \otimes v_{k+1}) \otimes \cdots \otimes v_n
		\end{align*} 
		for $v_k \in V(\nu_k)_{z_{\nu,k}}$. Here we identify $z_{\nu, k}, z_{\nu, k+1}$ in the domain with $z_{s_k(\nu), k+1}, z_{s_k(\nu), k}$, respectively, in the image.  
		}
		Similarly, let $\Km^{\nu} \colon \hVK_\nu \rightarrow \
		\hVK_{s_0(\nu)}$ be the QSP intertwiner defined as follows. 
		We first consider the $\bsF$--linear map $V_\nu \rightarrow \ 
		V_{s_0(\nu)}$ given by the composition of the operator
		\begin{align*}
	v_1 \otimes \cdots \otimes v_n
	\mapsto& \ \Km_{V(\nu_1)}(z_{1})\cdot(v_1) \otimes v_2 \otimes \cdots \otimes v_n. 
		\end{align*} 
		with the map identifying $z_{\nu, 1}$ with $z_{s_0(\nu), 1}^{-1}$
		and $z_{\nu, \ell}$ with 
		$z_{s_0(\nu), \ell}$ for $\ell\neq 1$. We then
		extend it by $\mathcal{O}$--linearity.
		\Omit{
		given by
		\begin{align*}
			v_1 \otimes \cdots \otimes v_n
			\mapsto& \ \Km_{V(\nu_1)}(X_{1})\cdot(v_1) \otimes v_2 \otimes \cdots \otimes v_n. 
		\end{align*} 
Again, we identify $z_{\nu, 1}$ in the domain with $z_{s_0(\nu), 1}^{-1}$ in the image. }
		From the Yang-Baxter equation \eqref{eq:z-YBE} and the unitarity condition \eqref{eq:unitarity-R}, it follows that 
		\begin{align*}
			\Rm^{\nu}_{k} \circ X_l =& \ X_{s_k(l)} \circ \Rm^{\nu}_{k}, \\
			\Rm^{s_k(\nu)}_{k} \circ \Rm^{\nu}_{k} =& \ \id_{V_{\nu}},   \\ 
			\Rm^{s_{k+1}s_k(\nu)}_{k} \circ \Rm^{s_k (\nu)}_{k+1} \circ \Rm^{\nu}_{k} =& \ 
			\Rm^{s_k s_{k+1}(\nu)}_{k+1} \circ \Rm^{s_{k+1}(\nu)}_{k} \circ \Rm^{\nu}_{k+1}. 
		\end{align*} 
		Instead, from Cherednik's generalized reflection equation \eqref{eq:rational-tw-RE} and the unitarity condition \eqref{eq:unitarity-k}, it follows that 
		\begin{align*} 
			\Km^{\nu} \circ X_k =& \ X_{s_0(k)} \circ \Km^{\nu},  \\ 
			\Km^{s_0(\nu)}\circ\Km^\nu =& \ \id_{V_\nu}, \\
			\Km^{s_1s_0s_1(\nu)} \circ \Rm_{1}^{s_0s_1(\nu)} \circ \Km^{s_1(\nu)} \circ \Rm_{1}^{\nu} =& \ \Rm_{1}^{s_0s_1s_0(\nu)} \circ \Km^{s_1s_0(\nu)} \circ \Rm_{1}^{s_0(\nu)} \circ \Km^{\nu}.    
		\end{align*} 
		It follows that letting $\Kcompl$ act by multiplication and assigning
		\begin{equation} \label{eq: oklr action on bimodule}
			e(\nu)s_0 \mapsto \Km^\nu, \quad e(\nu)s_k \mapsto \Rm_k^\nu
		\end{equation} 
		makes 
		$\tbVK[ \beta ]$ into a $\bl \Uqk, \skewr \br$-bimodule.\\
		
		(2)
		It is enough to show that $\tbV[{ \beta }]$ is stable under the action of the generators $\tau_k \in \oklr$. 
		First assume that $1 \leqslant k \leqslant n-1$. The proof is similar to that of \cite[Thm.~ 3.1.3]{kang-kashiwara-kim-18}. There are two cases to be considered. First suppose that $\nu_k \neq \nu_{k+1}$. 
		Then 
		\[ 
		e(\nu)\widetilde{P}_{\nu_{k+1},\nu_{k}}(x_k,x_{k+1})  = e(\nu)d_{\nu_{k},\nu_{k+1}}(X_{k+1}/X_{k}) \frac{ \bl x_k - x_{k+1} \br^{d_{\nu_{k},\nu_{k+1}}}c_{\nu_{k},\nu_{k+1}}(x_k,x_{k+1})}{d_{\nu_{k},\nu_{k+1}}(X_{k+1}/X_{k})}. 
		\] 
		Let $f$ denote the fraction on the RHS. 
		One easily calculates that 
		\[
		x_k - x_{k+1} = \bl X_k/ X(\nu_{k+1}) + X(\nu_k)/X_{k+1} \br \bl X_{k+1}/X_k - X(\nu_{k+1})/X(\nu_k)\br, 
		\]
		where the first factor on the RHS is an invertible element of $\Oh_\nu$. 
		Since $d_{\nu_{k},\nu_{k+1}}$ is the multiplicity of the polynomial $d_{\nu_{k},\nu_{k+1}}(X_{k+1}/X_{k})$ at $X_{k+1}/X_k = X(\nu_{k+1})/X(\nu_k)$, it follows that $f \in \Oh_\nu$. Hence 
		\[ 
		\hV_\nu \tau_k \subset \hV_\nu \widetilde{P}_{\nu_{k+1},\nu_{k}}(x_k,x_{k+1})s_k \subset  \hV_\nu d_{\nu_{k},\nu_{k+1}}(X_{k+1}/X_{k})s_k, 
		\]
		which, by definition, is contained in $\hV_{s_k\cdot\nu}$. 
		
		Secondly, suppose that $\nu_k = \nu_{k+1}$. By \ref{cond: P2}, we have 
		$$\rRMv{V(\nu_k),V(\nu_{k})}{X_{k+1}/X_k}|_{X_{k+1}/X_k=1} =\id.$$
		Hence $\rRMv{V(\nu_k),V(\nu_{k})}{X_{k+1}/X_k} - \id$  
		has a zero at $X_{k+1}/X_k = X(\nu_{k+1})/X(\nu_k) = 1$ and so 
		\begin{align*}
			\hV_\nu \tau_k =& \  \hV_\nu(x_k-x_{k+1})^{-1}(s_k-1)\\
 \subset& \ \hV_\nu( X_{k+1}/X_k - 1 )^{-1} (\rRMv{V(\nu_k),V(\nu_{k})}{X_{k+1}/X_k}
			- 1) \subset \hV_\nu. 
		\end{align*}
		
		Next, let $k = 0$. Again, there are two cases to be considered. First suppose that $\theta(\nu_1) \neq \nu_1$. Then 
		\[
		e(\nu)\widetilde{P}_{\theta(\nu_1)}(x_1) = e(\nu) d_{\nu_1}(X_1)\frac{x_1^{d_{\nu_1}}c_{\nu_1}(x_1)}{d_{\nu_1}(X_1)}.  
		\] 
		Let $f$ denote the fraction on the RHS. 
		Note that 
		\[
		x_1 = -(X(\nu_1)X_1)^{-1}(X_1 - X(\nu_1))(X_1 + X(\nu_1)).  
		\]
		Since $d_{\nu_1}$ is the multiplicity of the polynomial $d_{\nu_1}(X_{1})$ at $X_1 = X(\nu_{1})$, it follows that $f \in \Oh_\nu$. Hence 
		\[ 
		\hV_\nu \tau_0 \subset \hV_\nu \widetilde{P}_{\theta(\nu_1)}(x_1)s_0 \subset \hV_\nu d_{\nu_1}(X_1)s_0, 
		\]
		which, by definition, is contained in $\hV_{s_k\cdot\nu}$. 
		
		Secondly, suppose that $\nu_1 = \theta(\nu_1)$. Then \ref{cond: Q0} implies that $X(\nu_1) = \pm 1$. By \ref{cond: Q3}, we have $\Km_{V(\nu_1)}(X_1)|_{X_1=\pm 1} = \id$. 
		Hence $\Km_{V(\nu_1)}(X_1) - \id$ has a zero at $X_{1} = X(\nu_1) = \pm 1$ and so
		\begin{align*}
			\hV_\nu \tau_0 =  \hV_\nu x_1^{-1}(s_0-1) \subset  \hV_\nu ( X_1 \mp 1)^{-1} (\Km_{V(\nu_1)}(X_1) - 1) \subset \hV_\nu. 
		\end{align*}
		Part (3) follows immediately from the bimodule structure.
	\end{proof}
	

	\subsection{Compatibility with the Kang--Kashiwara--Kim functor}\label{ss:KKK-compatibility}
	
	We now consider the relationship between $\tfun$ and the Kang--Kashiwara--Kim functor from \cite{kang-kashiwara-kim-18}. 
	Take $\alpha \in \N[J]$ with $\ttt \alpha = \beta$. Set  
	\[
	\widehat{\mathcal{O}}_\alpha \seteq \soplus_{\nu \in J^\alpha} \Oh_{ \nu }e(\nu), \quad \widehat{\mathcal{K}}_\alpha \seteq \soplus_{\nu \in J^\alpha} \widehat{\mathcal{K}}_{ \nu }e(\nu), \quad \bV[ \alpha ] := \soplus_{\nu \in J^\alpha} \hV_\nu, \quad \bVK[ \alpha ] := \soplus_{\nu \in J^\alpha} \hVK_\nu
	\] 
	Note that there is a $\sym_n$-equivariant algebra isomorphism
	\begin{equation} \label{x-X map2}
		\bPh_{\alpha} \ \isoto \ \widehat{\mathcal{O}}_\alpha, \qquad e(\nu) \mapsto e(\nu), \quad x_ke(\nu) \mapsto \left(\frac{X_k}{X(\nu_k)} - 1 \right) e(\nu). 
	\end{equation} 
	We identify 
\[ \widehat{\mathbb{K}}_\alpha \rtimes \cor[\sym_n]  \isoto \widehat{\mathcal{K}}_\alpha \rtimes \cor[\sym_n]
\]
 using \eqref{x-X map2}, and consider $\klr$ as a subalgebra of $\widehat{\mathcal{K}}_\alpha \rtimes \cor[\sym_n]$ via \eqref{eq: klr loc}. 
	The right  action of $\skewr$ on $\tbVK[ \beta ]$ restricts to a $\widehat{\mathcal{K}}_\alpha \rtimes \cor[\sym_n]$-action on $\bVK[ \alpha ]$. This action commutes with the left action of $\UqLg$ and yields the following result due to Kang, Kashiwara, and Kim \cite[Thm.~ 3.1.3 and 3.2.1]{kang-kashiwara-kim-18}.
	
	\begin{theorem}[\cite{kang-kashiwara-kim-18}]
		\hfill
		\begin{enumerate}\itemsep0.25cm
			\item
			The subspace $\bV[{ \alpha }] \subset \bVK[ \alpha ]$ is stable under
			the right action of the subalgebra $\klr$ of $\widehat{\mathcal{K}}_\alpha \rtimes \cor[\sym_n]$. In particular, $\bV[{ \alpha }]$ has the structure of
			$(\UqLg,\klr)$-bimodule and yields a functor \[
			\fun_\alpha \colon \Modgr{\mathcal{R}} \longrightarrow \Mod{\UqLg}, \quad M \mapsto \bV[{ \alpha }]\otimes_{\klr}M . 
			\] 
			Summing over all the dimension vectors, one gets 
			\begin{equation} \label{eq: KKK functor}
				\fun := \soplus_{\alpha \in \N[J]} \fun_\alpha : \ \Modgr{\mathcal{R}}
				\longrightarrow \Mod{\UqLg}.
			\end{equation} 
			\item The functor \eqref{eq: KKK functor} is a tensor functor, which preserves finite-dimensional modules:
			\[
			\fun \colon \Modfdgr{\mathcal{R}} \longrightarrow \Modfd{\UqLg}. 
			\]
		\end{enumerate}
	\end{theorem} 
	
	
	We now prove that the functor $\tfun$ is \emph{compatible} with $\fun$, \ie it is a functor
	of module categories over $ \Modfdgr{\mathcal{R}}$ and $\Modfd{\UqLg}$, respectively.
	
	\begin{theorem} \label{thm: KKK compatibility} 
		The functors $(\tfun,\fun)$ intertwine the two monoidal actions:    
\[ \Modgr{{}^\theta \mathcal{R}({\bm\lambda})} \curvearrowleft \Modgr{\mathcal{R}}, \qquad \modv{\Uqk} \curvearrowleft \modv{\UqLg}. \] 
	\end{theorem} 
	
	\begin{proof}
We need to show that there are natural isomorphisms
		\begin{equation} \label{eq: functor preserves action} 
		\tfun(M \acts N) \cong \tfun(M) \otimes \fun(N), 
		\end{equation} 
		for all $M \in \Modgr{{}^\theta \mathcal{R}({\bm\lambda})}$ and $N \in \Modgr{\mathcal{R}}$. 

		Let $\beta = \beta_1 + {}^\theta\beta_2$ for some $\beta_1 \in \N[J]$ and $\beta_2 \in \N[J]^\theta$ with $\Norm{\beta_1}_\theta=m$ and $\Norm{\beta_2}=n$. 
		For each $\nu \in {}^\theta J^{\beta}$ with $\nu'= \nu_1,\cdots,\nu_m \in {}^\theta J^{\beta_1}$ and $\nu''= \nu_{m+1},\cdots,\nu_{n+m} \in J^{\beta_2}$, there is an algebra homomorphism 
		$\Oh_{\nu'} \otimes \Oh_{\nu''} \to \Oh_{\nu}$, restricting to the identity map on $\Oh_{\nu'}$ and sending $1-X_{m+k}X(\nu_{m+k})^{-1}$ to $\mathsf{f}(1-X_{m+k}X(\nu_{m+k})^{-1})$ for $1 \leqslant k \leqslant n$. 
		For any finite-dimensional $\Oh_{\nu'}$-module $L_1$ and any finite-dimensional $\Oh_{\nu''}$-module $L_2$, the induced morphism
		\[
		L_1 \otimes L_2 \to \Oh_{\nu} \otimes_{\Oh_{\nu'} \otimes \Oh_{\nu''}}(L_1 \otimes L_2) 
		\]
		is an isomorphism. 
		It follows that for any finite-dimensional $\tOh_{\beta_1}$-module $L_1$ and any finite-dimensional $\Oh_{\beta_2}$-module $L_2$, the induced morphism 
		\begin{equation} \label{eq: orb and KKK}
			(\tbV[\beta_1]\otimes \bV[\beta_2])\otimes_{\tOh_{\beta_1} \otimes\Oh_{\beta_2} } (L_1 \otimes L_2) \to \tbV[\beta] \otimes_{\tOh_{\beta_1} \otimes \Oh_{\beta_2}} (L_1 \otimes L_2)
		\end{equation} 
		is also an isomorphism. 
		
		The module $\tbV[\beta] \otimes_{\oklr}(M \acts N) \cong \tbV[\beta] \otimes_{\oklrvv{\beta_1}{\beta_2}}(M \otimes N)$ is the quotient of $\tbV[\beta] \otimes_{\tOh_{\beta_1} \otimes\Oh_{\beta_2}}(M \otimes N)$ by the submodule generated by $ur\otimes v - u \otimes rv$, where $r \in \oklrvv{\beta_1}{\beta_2}$, $u \in M \otimes N$ and $v \in \tbV[\beta]$. An analogous statement holds for $(\tbV[\beta_1]\otimes \bV[\beta_2])\otimes_{\oklrvv{\beta_1}{\beta_2}}(M \otimes N)$. This, together with \eqref{eq: orb and KKK}, implies the existence of an isomorphism \eqref{eq: functor preserves action}. It is routine to check that the conditions from \cite[(8)--(9)]{haring-old} are also satisfied. 
	\end{proof}
	
	\subsection{Basic properties of the Schur--Weyl functor}\label{ss:SW-properties}
	We now prove several basic properties of $\tfun$. 

	\begin{prop} \label{pro: basic props of SW}
		The functor $\tfun$ has the following properties.
		\begin{enumerate}[itemsep = 2pt]
			\item $\tfun$ preserves finite-dimensional modules. 
			\item For any $\klr$-module $M$, we have 
			\[
			{}^\theta\mathsf{F}(\mathds{1} \acts M) = \mathsf{F}(M)|_{\Uqk}. 
			\] 
			\item 
			For any $i \in J$ such that $i \notin J^\theta$ and ${}^\theta{\bm\lambda}(i) \geqslant 1$, we have 
			\[ {}^\theta\mathsf{F}({}^\theta L(i)) \cong \coker K_{i}(0),\]
where $K_i(x_1) = x_1^{d_i}c_i(x_1)\mathbf{K}_{V(i)}(x_1)$ and  $x_1 =  (X(i)X_1^{-1} - X(i)^{-1}X_1)$.
		\end{enumerate}
	\end{prop}
	
	\begin{proof} 
		Part (1) is obvious and (2) follows immediately from Theorem \ref{thm: KKK compatibility}. Let us prove (3). 
		By Lemma \ref{lem: 1dim mod oklr}, the $\oklr$-module ${}^\theta L(i)$, with $\beta = i + \theta(i)$, is well-defined. Let us abbreviate $j = \theta(i)$. 
		By definition, 
		\[
		{}^\theta\mathsf{F}({}^\theta L(i)) = (\hV_i \oplus \hV_{j})\otimes_{\oklr}{}^\theta L(i),
		\]
		which is the quotient of $(\hV_i \oplus \hV_{j})\otimes {}^\theta L(i)$ by the subspace $N$ spanned by elements of the form $v \otimes r \cdot u_i - v \cdot r \otimes u_i$, for $r \in \oklr$. 
		Since $e(j)$ annihilates ${}^\theta L(i)$, it follows that $\hV_j\otimes {}^\theta L(i) \subset N$. 
		Next, since $x_1 e(i) = (X(i)X_1^{-1} - X(i)^{-1}X_1)e(i)$ annihilates ${}^\theta L(i)$, we get that $(X_1 - X(i))\hV_i\otimes{}^\theta L(i) \subset N$. Finally, since $\tau_0$ annihilates ${}^\theta L(i)$, it follows that $\tau_0 (\hV_j)\otimes {}^\theta L(i) \subset N$. But $\tau_0|_{\hV_j} = K_i(0)$. 
		Since $\oklr$ is generated by $e(i), e(j), x_1$ and $\tau_0$, we deduce that $N$ is spanned by the aforementioned subspaces and that $((\hV_i \oplus \hV_{j})\otimes{}^\theta L(i))/N \cong \coker K_i(0)$. 
	\end{proof} 
	
	
	\section{The BKR isomorphism in type $\mathsf{C}$} \label{s: BRK}
	
	In this section only, we allow $q$ to have a finite order, \ie $\ord(q) \in \Z_{\geqslant 3} \cup \{\infty\}$,
	and we set 
	\[
	\cor\coloneqq
	\begin{cases}
		\C[q]/(q^{\ord(q)}-1) & \mbox{if } \ord(q) < \infty\\
		\overline{\C(q)} & \mbox{if } \ord(q) = \infty\\
	\end{cases}
	\]
	The main result is the construction of an isomorphism 
	between certain completions of the $o$KLR algebras and affine Hecke algebras of type $\sfC$
	(Thm.~\ref{thm: bkr iso C}) in analogy with a similar result due to 
	Brundan-Kleshchev and Rouquier \cite{Brundan-Kleshchev, Rouquier-2KM}.  
	
	\subsection{Affine Hecke algebras of type $\mathsf{C}$}\label{ss:aff-Hecke-C} 
	
	We recall the definition of the (3-parameter) affine Hecke algebra of type $\mathsf{C}$. 
	
	\begin{definition} 
		Fix $p_0,p_1 \in \cor^\times$. The affine Hecke algebra $\mathcal{H}_{\mathsf{C}_n}(p_0,p_1)$ of type $\mathsf{C}_n$ is the $\cor$-algebra generated by $T_k$ $(0 \leqslant k \leqslant n-1)$ and $X_l^{\pm1}$ $(1 \leqslant l \leqslant n)$ subject to the relations: 
		\begin{itemize}
			\item quadratic relations: 
			\begin{alignat*}{3} &(T_k - q)(T_k + q^{-1}) = 0 &\quad \quad &(1 \leqslant k \leqslant n-1),\\
				&(T_0 - p_0)(T_0 + p_1^{-1}) = 0, \\
				\intertext{\item braid relations: }
				&T_kT_{k+1}T_k = T_{k+1}T_kT_{k+1} &\quad \quad &(1 \leqslant k \leqslant n-1),\\
				&(T_0T_1)^2 = (T_1T_0)^2,\\
				&T_{k}T_{k'} = T_{k'}T_k &\quad \quad &(k \neq k' \pm 1), \\
				\intertext{\item Laurent polynomial relations:}
				&X_l X_{l'} = X_{l'}X_l, &\quad \quad & (1 \leqslant l,l' \leqslant n), \\
				& X_lX_l^{-1} = 1 = X_l^{-1}X_l &\quad \quad & (1 \leqslant l \leqslant n), \\ 
				\intertext{\item mixed relations: } 
				& T_kX_kT_k = X_{k+1} &\quad \quad&(1 \leqslant k \leqslant n-1),\\ 
				&T_0X_1^{-1}T_0 = p_0p_1^{-1}X_1 + (p_0p_1^{-1} - 1)T_0, \\ 
				& T_kX_l = X_lT_k &\quad \quad& (l \neq k, k+1).
			\end{alignat*}
		\end{itemize} 
		The affine Hecke algebra of type $\mathsf{B}_n$ is the specialization $\mathcal{H}_{\mathsf{B}_n}(p) = \mathcal{H}_{\mathsf{C}_n}(p,p)$. The finite Hecke algebra $\HecBf$ of type $\mathsf{B}_n$ is the subalgebra of $\mathcal{H}_{\mathsf{B}_n}(p)$ generated by the $T_k$. 
	\end{definition} 
	
	\begin{remark} \label{rem: conv Kato VV}
		We use the conventions of \cite[\S A.1]{VV-HecB} (the assignment $q \mapsto p$, $p_0 \mapsto q_0$, $p_1 \mapsto q_1$ gives a matching between our parameters and those in \emph{loc.\ cit.}). This convention is the same as in \cite[\S2.3]{FLLLWW}, if one matches the parameters as follows: $q \mapsto q^{-1}$, $p_0 \mapsto q_0^{-1}$, $p_1 \mapsto q_1^{-1}$. 
		To match our conventions with those of \cite[Definition 2.1]{Kato-09}, one uses the assignment 
		\[ T_k \mapsto \mathbf{q}_2^{-1/2}T_k, \quad T_0 \mapsto \mathbf{q}_1^{-1}T_n, \quad p_1 \mapsto \mathbf{q}_1, \quad p_0 \mapsto -\mathbf{q}_0.
		\]
		\rmkend 
	\end{remark}

	The intertwiners $\Phi_k \in \mathcal{H}_{\mathsf{C}_n}(p_0,p_1) \otimes_{\mathcal{O}} \mathcal{K}$ are defined as (see, e.g., \cite[\S A.3]{VV-HecB}: 
	\begin{align} 
		\label{eq: int C1} \Phi_k =& \ 1 + \frac{X_k-X_{k+1}}{qX_k - q^{-1}X_{k+1}}(T_k - q) \quad \quad \qquad \qquad  (1 \leqslant k \leqslant n-1), \\ 
		\label{eq: int C2}
		\Phi_0 =& \ 1 + p_1\frac{X_1^2-1}{(X_1+p_0)(X_1-p_1)}(T_0 - p_0). 
	\end{align}

	\begin{prop} \label{pro: Hecke C pol rep}
		The following hold. 
		\begin{enumerate} 
			\item There is a an isomorphism of $\mathcal{K}$-algebras   
			\[ 
			\cor[\weyl_n] \ltimes \mathcal{K} \ \isoto \ \mathcal{H}_{\mathsf{C}_n}(p_0,p_1) \otimes_{\mathcal{O}} \mathcal{K}, \quad s_k \mapsto \Phi_k \quad (0 \leqslant k \leqslant n-1). 
			\]
			\item 
			The affine Hecke algebra $\mathcal{H}_{\mathsf{C}_n}(p_0,p_1)$ has a faithful representation on $\mathcal{O}$ given by 
			\begin{itemize}
				\item $X_{\pm 1}, \cdots, X_{\pm n}$ acting naturally by multiplication, 
				\item $T_1, \cdots, T_{n-1}$ acting via 
				\[
				(T_k-q) \cdot f = \frac{qX_k - q^{-1}X_{k+1}}{X_k - X_{k+1}}(s_k(f) - f), 
				\]
				\item $T_0$ acting via 
				\[
				(T_0-p_0) \cdot f = p_1^{-1}\frac{(X_1+p_0)(X_1-p_1)}{X_1^2-1}(s_0(f) - f). 
				\]
			\end{itemize}
		\end{enumerate}
	\end{prop} 
	
	\begin{proof} 
		The first statement can be found in, e.g., \cite[\S A.3]{VV-HecB}. The second statement is \cite[Thm.~ 2.7]{Kato-09}, taking into account the difference in conventions explained in Remark \ref{rem: conv Kato VV}. 
	\end{proof}
	
	\subsection{BKR-type isomorphism} \label{ss:BKR-iso}
	We establish a Brundan--Kleshchev--Rouquier-style isomorphism between completions of orientifold KLR algebras and affine Hecke algebras of type $\mathsf{C}$, generalizing \cite[Thm.~ A.4]{VV-HecB} and \cite[Thm.~ 1.1]{PAW-B}. 
	
	Assume $p_0,p_1 \neq \pm 1$. We now determine the datum defining $\oklr$. Let $\xi \in \cor^\times$ and set 
	\[ J = \{ \xi^{\pm 1} q^{2k} \mid k \in \Z \}, \quad \theta \colon i \mapsto i^{-1}, \quad a_{ij} = \delta_{j = q^2i}, \quad \bm\lambda(i) = \delta_{i=p_1}+\delta_{i=-p_0}. \] 
	The resulting quiver $\Gamma = (J,\Omega)$ can, depending on $\xi$ and the order of $q$, be of the following types:  
	\renewcommand{\arraystretch}{1.5}
	\setlength{\tabcolsep}{6pt}

	\[
	\begin{tabular}{ |l|l|l|l|l| }
		\hline
		& $\ord(q)$ & $\xi$ & $\Gamma$ & $J^\theta$ \\
		\hline
		(1) & $\infty$ & $1$ & $A_\infty$ & $\{1\}$ \\
		\hline
		(2) & $\infty$ & $q$ & $A_\infty$ & $\varnothing$ \\
		\hline
		(3) & $\infty$ & $\notin \{\pm q^{\Z}\}$ & $A_\infty \times A_\infty$ & $\varnothing$ \\
		\hline
		(4) & $2m$ & $1$ & $A_m^{(1)}$ & $\{\pm1\}$ \\
		\hline
		(5) & $2m$ & $q$ & $A_m^{(1)}$ & $\varnothing$ \\
		\hline
		(6) & $2m$ & $\notin \{q^{\Z}\}$ & $A_m^{(1)} \times A_m^{(1)}$ & $\varnothing$ \\
		\hline
		(7) & $m$ odd & $1$ & $A_m^{(1)}$ & $\{1\}$ \\
		\hline
		(8) & $m$ odd & $\notin \{\pm q^{\Z}\}$ & $A_m^{(1)} \times A_m^{(1)}$ & $\varnothing$ \\
		\hline 
	\end{tabular} 
	\]
\vspace{0.3cm}
	
	Given a self-dual dimension vector $\beta$ with $\Norm{\beta}_\theta = n$, let $\oklr$ be the orientifold KLR algebra associated to the datum $(\Gamma, \theta, \beta, \bm\lambda)$. Choose $X(-)$ in  the $J$--datum to be the identity function. 
	\renewcommand{\arraystretch}{1}

	\begin{theorem} \label{thm: bkr iso C} 
		The isomorphism \eqref{x-X map} extends to an algebra isomorphism 
		\begin{equation} \label{eq: oKLR-Hecke comp iso}
			\oklr \otimes_{\tbP_{\beta}} \tbPh_{\beta} \ \isoto \ \mathcal{H}_{\mathsf{C}_n}(p_0,p_1) \otimes_{\mathcal{O}} \tOhb
		\end{equation}
		given by: 
		\begin{align*} 
			\tau_k e(\nu) \mapsto& \ \left\{ \begin{array}{ll}
				\left( \frac{X_k}{X(\nu_{k+1})} + \frac{X(\nu_k)}{X_{k+1}} \right)^{-1} \left( \frac{X_{k+1}}{X_k} - 1 \right)^{-1} \left({\Phi}_k - 1\right) e(\nu)  & \mbox{ if } \nu_k = \nu_{k+1}, \\ \\
				\left( \frac{X_k}{X(\nu_{k+1})} + \frac{X(\nu_k)}{X_{k+1}} \right) \left( \frac{X(\nu_{k+1})}{X(\nu_k)} - \frac{X_{k+1}}{X_k}\right) {\Phi}_ke(\nu) & \mbox{ if } \nu_{k+1} = q^2\nu_k, \\ \\ 
				{\Phi}_ke(\nu) & \mbox{ otherwise, }
			\end{array}\right. \\ \\
			\tau_0 e(\nu) \mapsto& \ \left\{ \begin{array}{ll} 
				X(\nu_1)\left( X_1^{-1} - X_1 \right)^{-1} \left({\Phi}_0 - 1 \right) e(\nu) & \qquad \qquad \quad \ \ \mbox{ if } \nu_1 = \theta(\nu_1), \\ \\
				\left(\frac{X_1}{X(\nu_1)} - \frac{X(\nu_1)}{X_1} \right)^{\bm\lambda(\nu_1)}{\Phi}_0 e(\nu) &  \qquad \qquad \quad \ \  \mbox{ otherwise. } 
			\end{array}\right. 
		\end{align*}
	\end{theorem} 
	
	
	\begin{proof} 
		The completed algebra on the LHS of \eqref{eq: oKLR-Hecke comp iso} has a faithful representation on $\tbPh_{\beta}$ while the completed algebra on the RHS has a faithful representation on $\tOhb$. Therefore, it suffices to check that the actions of the generators agree under the isomorphism \eqref{x-X map}. This follows by a direct calculation using Proposition \ref{pro: polrep oklr}, Proposition \ref{pro: Hecke C pol rep} and \eqref{eq: int C1}--\eqref{eq: int C2}. 
	\end{proof}
	
	\begin{remark}
		A weaker statement about a Morita equivalence between orientifold KLR algebras and blocks of affine Hecke algebras of type $\mathsf{C}$, in the case when $\theta$ has no fixed points, can be found in \cite[Thm.~ A.4]{VV-HecB}. A proof of the isomorphism of cyclotomic quotients, in type $\mathsf{B}$ case, for any $\theta$, appeared in \cite[Thm.~ 1.1]{PAW-B}. 
		The idea of comparing the polynomial representations goes back to \cite{Rouquier-2KM}
		(see also \cite{Miemietz-Stroppel, Maksimau-Stroppel}). 
		\rmkend
	\end{remark}
	
	
	\section{Boundary Schur-Weyl duality in type $\sfA$}\label{s:BSW-A} 
	
	In this section we study a particular instance of the boundary Schur-Weyl duality from Section~\ref{s: boundary SW} in the case $\g=\mathfrak{sl}_N$. Specifically, we consider
	only quasi-split affine QSP subalgebras of type $\mathsf{AIII}$ with the unitary K-matrices  satisfying the
	standard reflection equation on the first fundamental representation.
	
	\subsection{The fundamental representation}
	Recall that the first fundamental representation $\sfV_{\omega_1}$ of $U_q\mathfrak{sl}_N$ extends, through 
	the evaluation morphism, to a representation of the quantum loop algebra $\UqLsl{N}$.
	More precisely, we denote by $\frep{}$ the $N$-dimensional $\UqLsl{N}$-module with
	basis $u_1,\cdots, u_N$ and action given by
	\begin{align*}
		E_i\cdot u_r =& \ \delta(r \equiv i+1 \textrm{ mod } N) \ u_i, \\
		F_i\cdot u_r =& \ \delta(r \equiv i \textrm{ mod } N) \ u_{i+1}, \\
		K_i\cdot u_r =& \ q^{\delta(r \equiv i \textrm{ mod } N) - \delta(r \equiv i+1 \textrm{ mod } N)} \ u_{r}, 
	\end{align*}
	where for simplicity we are adopting the cyclic notation on the indices of the 
	Chevalley generators and the basis vectors $u_k$ 
	(\eg $u_N=u_0$, $E_N=E_0$, and so on). More precisely, $\frep{}=\operatorname{ev}_1^*(\sfV_{\omega_1})$, where
	$\operatorname{ev}_1 \colon \UqLsl{N}\to U_q\mathfrak{sl}_N$ is the evalution morphism at $z=1$.
	We set $\frep{z}\coloneqq\frep{}\ten\bsF[z,z^{-1}]$ and $\shrep{\frep{}}{z}\coloneqq
	\frep{}\ten\bsF(z)$ endowed with the shifted $\UqLsl{N}$-action 
	(cf.~\S\ref{ss:spectral-R}).\\
	
	The explicit formulae for the unitary R-matrix on $\frep{}$ is well-known and due to Jimbo
	\cite{jimbo-85}. 
	More precisely, 
	\[
	\rRMv{}{z,w} \colon \shrep{\frep{}}{z}\ten\shrep{\frep{}}{w}
	\to \shrep{\frep{}}{w}\ten\shrep{\frep{}}{z}
	\]
	is given by $u_r\ten u_r\mapsto u_r\ten u_r$ and 
	\begin{equation}\label{eq: R mat fund1} 
		u_r \otimes u_s \quad\mapsto\quad
		\displaystyle \frac{(1-q^2)w^{\delta(r > s)}z^{\delta(r<s)}}{w-q^2z}\cdot  u_r \otimes u_s + \frac{q(w-z)}{w-q^2z}\cdot u_s \otimes u_r 
		\qquad (r\neq s)
	\end{equation} 
	In particular, $\rRMv{}{z,w}$ is a rational function in $w/z$ with only one simple pole 
	at $w/z=q^2$.
	
	\subsection{Normalised parameters and K-matrices} \label{ss:fund-rep-K-mx}
	
	Let $(X,\tau)$ be an affine Satake diagram of type $\sfA$ and $\Uqk\subset\UqLsl{N}$ 
	the corresponding QSP subalgebra with parameters
	$(\parc{},\pars{})\in\Parsetc\times\Parsets$ (cf.~\S\ref{ss:qsp}). 
	We are interested in the K-matrices on $\shrep{\frep{}}{z}$ supported on $\Uqk$. 
	From the point of view of the combinatorial model described in \S\ref{ss:combinatorics}, it is convenient to consider $\UqLsl{N}$-modules
	$V$ that are \emph{fixed} by the chosen twisting operator, \ie $\psi^*(V)=V$, or conversely
	to consider only twisting operators that fix a given module. In both cases, the 
	corresponding K-matrix gives rise to a solution of the {\em standard} reflection equation (see, \eg \S\ref{ss:K-mx-KR}).
	Since $\frep{}$ is small, there exists a distinguished twisting operator $\psi$ 
	which fixes it.
	Note, however, that this requires to impose a normalisation condition
	$\parc{}(\drv{})=\beta$ for a  uniquely determined $\beta\in\bsF^\times$, cf.~\cite[\S 7]{appel-vlaar-22}.\\
	
	Therefore, for any QSP subalgebra in $\UqLsl{N}$ with normalized parameter $\parc{}$, we obtain a rational QSP intertwiner
	\[
	\rKM{}{z} \colon \shrep{\frep{}}{z}\to\shrep{\frep{}}{z^{-1}}
	\]
	which satisfies the standard reflection equation. By generic QSP irreducibility of the
	first fundamental representation, we recover in this way the QSP intertwiners explicitly 
	described by Regelskis-Vlaar in \cite{regelskis-vlaar-16,regelskis-vlaar-18}.
	By direct inspection, it then follows that the operators $\rKM{}{z}$ can be normalized to be 
	non-vanishing, unitary and such that $\rKM{}{1}=\id$. In particular, the condition 
	\ref{cond: Q3} holds.
	
	\subsection{Quasi-split affine QSP subalgebras of type $\mathsf{AIII}$}\label{ss:AIII-qsp}
	
	We provide the explicit formulae of the unitary K-matrix for the first fundamental 
	representation in the case of quasi-split affine QSP subalgebra of type $\mathsf{AIII}$.
	Our main reference is \cite{regelskis-vlaar-16, regelskis-vlaar-18} as explained above. 
	We note however that the K-matrices in this case first appeared in \cite{abad-rios-95}.
	
	\subsubsection{The non-restrictable case} \label{sec: non-restr case}
	
	Recall that a quasi-split QSP subalgebra is simply determined by a non-trivial involution on 
	the Dynkin diagram. We first consider the case where the affine node, which is numbered $N$
	according to our cyclic notation, is not fixed by the involution, \ie we consider the case of 
	$N$ even and Satake diagram
	
	\eq 
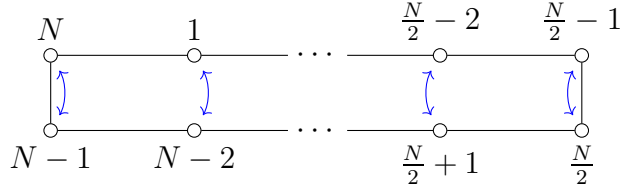
\label{eq: Satake aff diag 3}
	\begin{tikzpicture}
		\matrix [column sep={0.6cm}, row sep={0.5 cm,between origins}, nodes={draw = none,  inner sep = 3pt}]
		{
			&\node(U1) [draw, circle, fill=white, scale=0.6, label = $N$] {};
			&\node(U2) [draw, circle, fill=white, scale=0.6, label = $1$] {};
			&\node(U3) {$\cdots$};
			&\node(U4)[draw, circle, fill=white, scale=0.6, label =$\frac{N}{2}-2$] {};
			&\node(U5)[draw, circle, fill=white, scale=0.6, label =$\frac{N}{2}-1$] {};
			\\
			&&&&&
			\\
			&\node(L1) [draw, circle, fill=white, scale=0.6, label =below:$N-1$] {};
			&\node(L2) [draw, circle, fill=white, scale=0.6, label =below:$N-2$] {};
			&\node(L3) {$\cdots$};
			&\node(L4)[draw, circle, fill=white, scale=0.6, label =below:$\frac{N}{2}+1$] {};
			&\node(L5)[draw, circle, fill=white, scale=0.6, label =below:$\frac{N}{2}$] {};
			\\
		};
		\begin{scope}
			\draw (U1) -- node {} (L1); 
			\draw (U1) -- node  {} (U2);
			\draw (U2) -- node  {} (U3);
			\draw (U3) -- node  {} (U4);
			\draw (U4) -- node  {} (U5);
			\draw (U5) -- node  {} (L5);
			\draw (L1) -- node  {} (L2);
			\draw (L2) -- node  {} (L3);
			\draw (L3) -- node  {} (L4);
			\draw (L4) -- node  {} (L5);
			\draw (L1) edge [color = blue,<->, bend right, shorten >=4pt, shorten <=4pt] node  {} (U1); 
			\draw (L2) edge [color = blue,<->, bend right, shorten >=4pt, shorten <=4pt] node  {} (U2);
			\draw (L4) edge [color = blue,<->, bend left, shorten >=4pt, shorten <=4pt] node  {} (U4);
			\draw (L5) edge [color = blue,<->, bend left, shorten >=4pt, shorten <=4pt] node  {} (U5);
		\end{scope}
	\end{tikzpicture}
	\eneq 
	For any $1\leqslant r\leqslant N$, we set $\bar{r}\coloneqq N-r$. 
	The corresponding QSP subalgebra $\Uqk$ is
	generated by $K_rK_{\bar{r}-1}^{-1}$ $(0 \leqslant r \leqslant N/2-1)$ and 
	\[
	B_r \coloneqq F_r + \parc{r}E_{\bar{r}-1}K_r^{-1} + \pars{r}K_r^{-1} \quad (1 \leqslant r \leqslant N)
	\]
	where the parameters $\parc{},\pars{}$ are determined according to Remark~\ref{rmk:parameters}(1) by the following assignments ($\lambda,\mu\in\bsF^{\times}$):\footnote{Note that, both here and in 
		\S\ref{sec: restr case}, with respect to the formulae given in
		\cite[\S 9.3]{regelskis-vlaar-16}, we are setting the shifting parameter 
		$\eta$ equal to 1 and we are choosing suitable dressing parameters $\omega_i$.}
	\begin{align}
		\parc{N} = q,\quad \parc{N-1} = \mu^{-2},\quad
		\parc{N/2+1} = q\mu\lambda^{-1},\quad \parc{N/2-1} = \lambda\mu,
	\end{align}
	\begin{align}
		\pars{N}=\frac{\mu-\mu^{-1}}{q-q^{-1}},\quad 
		\pars{N/2}=\frac{\lambda-\lambda^{-1}}{q-q^{-1}},
		\quad
		\pars{r}=0\quad \left( 0 \leqslant r < \frac{N}{2} \right)
	\end{align}
	Following \cite[\S 9]{regelskis-vlaar-16}, the compact form of the K-matrix $\rKM{}{z}$ 
	on $\frep{}$ supported on $\Uqk$ is given by: 
	\begin{align} \label{eq: RV V1 K-matrix}
		\rKM{}{z} =& \ \id + \frac{z-z^{-1}}{(\lambda \mu - z)}\left( E_{NN} + \frac{M_2}{(\lambda^{-1} + (\mu z)^{-1})} \right)
	\end{align}
	where
	\begin{align}
		M_2 =& \ 
		\sum_{1 \leqslant i \leqslant N/2-1} (\lambda E_{ii} + \lambda^{-1} E_{N-i,N-i} + \mu E_{i,N-i} + \mu^{-1} E_{N-i,i}). 
	\end{align} 
	It is convenient for us to consider the parameters $p_0 = \lambda \mu^{-1}$ and $p_1 = \lambda \mu$. Then, formula \eqref{eq: RV V1 K-matrix} reads $u_{N/2}\mapsto u_{N/2}$, $u_N\mapsto 
	\left(1+\frac{z-z^{-1}}{p_1 - z}\right) \cdot u_N$, and 
	\begin{equation}
		\label{eq: explicit K matrix 2} 
		u_r \mapsto
		\frac{(p_0^{-1}p_1-1) + (p_1-p_0^{-1})z^{\delta(r < \bar{r})}z^{-\delta(r > \bar{r})}}{(p_1-z)(p_0^{-1} + z^{-1})}\cdot u_r + 
		\frac{(p_0^{-1}p_1)^{\delta(r < \bar{r})}(z-z^{-1})}{(p_1-z)(p_0^{-1}+z^{-1})}
		\cdot u_{\bar{r}} 
	\end{equation}
	if $1 \leqslant r < N/2$.
	
	\begin{remark}
	Since the QSP subalgebra is non--restrictable, the K--matrix $\rKM{}{z}$
	does not fit into the framework described in Section~\ref{s:rational}. It is
	expected, however, that $\rKM{}{z}$ is also induced by the action of a 
	universal K--matrix.\rmkend
	\end{remark}
	
	\subsubsection{The restrictable case} \label{sec: restr case} 
	For $N$ even, we consider the Satake diagram 
	\eq 
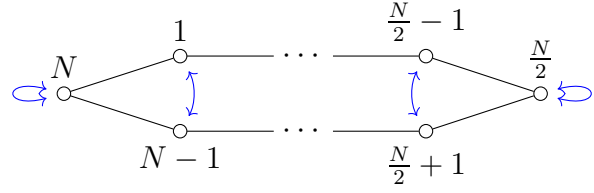
\label{eq: Satake aff diag 1}
	\begin{tikzpicture} 
		\matrix [column sep={0.6cm}, row sep={0.5 cm,between origins}, nodes={draw = none,  inner sep = 3pt}]
		{ 
			&\node(U1) [draw, circle, fill=white, scale=0.6, label = $1$] {};
			&\node(U2) {$\cdots$};
			&\node(U3)[draw, circle, fill=white, scale=0.6, label =$\frac{N}{2}-1$] {};
			\\
			\node(L)[draw, circle, fill=white, scale=0.6, label =$N$] {};
			&&&&
			\node(R)[draw, circle, fill=white, scale=0.6, label =$\frac{N}{2}$] {};
			\\
			&\node(L1) [draw, circle, fill=white, scale=0.6, label =below:$N-1$] {};
			&\node(L2) {$\cdots$};
			&\node(L3)[draw, circle, fill=white, scale=0.6, label =below:$\frac{N}{2}+1$] {};
			\\
		};
		\begin{scope}
			\draw (L) -- node  {} (U1);
			\draw (U1) -- node  {} (U2);
			\draw (U2) -- node  {} (U3);
			\draw (U3) -- node  {} (R);
			\draw (L) -- node  {} (L1);
			\draw (L1) -- node  {} (L2);
			\draw (L2) -- node  {} (L3);
			\draw (L3) -- node  {} (R);
			\draw (L) edge [color = blue, loop left, looseness=40, <->, shorten >=4pt, shorten <=4pt] node {} (L);
			\draw (R) edge [color = blue,loop right, looseness=40, <->, shorten >=4pt, shorten <=4pt] node {} (R);
			\draw (L1) edge [color = blue,<->, bend right, shorten >=4pt, shorten <=4pt] node  {} (U1);
			\draw (L3) edge [color = blue,<->, bend left, shorten >=4pt, shorten <=4pt] node  {} (U3);
		\end{scope}
	\end{tikzpicture}
	\eneq  
	
	The corresponding QSP subalgebra $\Uqk$ is generated by $K_rK_{\bar{r}-1}^{-1}$ $(1 \leqslant r < N/2)$ and 
	\begin{equation} \label{eq: Br generators coid} 
		B_r = F_r + \parc{r}E_{\bar{r}-1}K_r^{-1} + \pars{r}K_r^{-1} \quad (1 \leqslant r \leqslant N),
	\end{equation} 
	where $\bar{r} = N+1-r$ and the parameters $\parc{}$ and $\pars{}$ are determined by
	the assignments:
	\begin{align}
		\parc{N} = q^{-1}\mu^{-2},\quad 
		\parc{N/2} = q^{-1}\mu^2, \quad
		\parc{r} = 1\quad\left(1 \leqslant r < \frac{N}{2}\right)
	\end{align}
	\begin{align}
		\pars{N}= \frac{1 - \mu^2}{q - q^{-1}},\quad
		\pars{\frac{N}{2}} =\frac{(\lambda\mu)^{-1}-\lambda\mu^{-1}}{q-q^{-1}},\quad
		\pars{r}=0\quad \left( 0 \leqslant r < \frac{N}{2} \right)
	\end{align}
	
	For $N$ odd, we consider the Satake diagram 
	\eq 
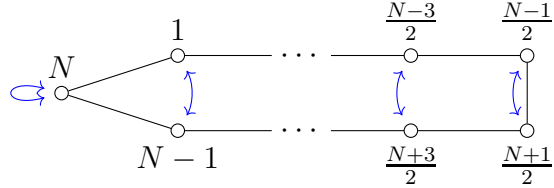
\label{eq: Satake aff diag 2}
	\begin{tikzpicture}
		\matrix [column sep={0.6cm}, row sep={0.5 cm,between origins}, nodes={draw = none,  inner sep = 3pt}]
		{
			&\node(U1) [draw, circle, fill=white, scale=0.6, label = $1$] {};
			&\node(U3) {$\cdots$};
			&\node(U4)[draw, circle, fill=white, scale=0.6, label =$\frac{N-3}{2}$] {};
			&\node(U5)[draw, circle, fill=white, scale=0.6, label =$\frac{N-1}{2}$] {};
			\\
			\node(L)[draw, circle, fill=white, scale=0.6, label =$N$] {};
			&&&&&
			\\
			&\node(L1) [draw, circle, fill=white, scale=0.6, label =below:$N-1$] {};
			&\node(L3) {$\cdots$};
			&\node(L4)[draw, circle, fill=white, scale=0.6, label =below:$\frac{N+3}{2}$] {};
			&\node(L5)[draw, circle, fill=white, scale=0.6, label =below:$\frac{N+1}{2}$] {};
			\\
		};
		\begin{scope}
			\draw (L) -- node  {} (U1);
			\draw (U1) -- node  {} (U3);
			\draw (U3) -- node  {} (U4);
			\draw (U4) -- node  {} (U5);
			\draw (U5) -- node  {} (L5);
			\draw (L) -- node  {} (L1);
			\draw (L1) -- node  {} (L3);
			\draw (L3) -- node  {} (L4);
			\draw (L4) -- node  {} (L5);
			\draw (L) edge [color = blue, loop left, looseness=40, <->, shorten >=4pt, shorten <=4pt] node {} (L);
			\draw (L1) edge [color = blue,<->, bend right, shorten >=4pt, shorten <=4pt] node  {} (U1);
			\draw (L4) edge [color = blue,<->, bend left, shorten >=4pt, shorten <=4pt] node  {} (U4);
			\draw (L5) edge [color = blue,<->, bend left, shorten >=4pt, shorten <=4pt] node  {} (U5);
		\end{scope}
	\end{tikzpicture}
	\eneq 
	In this case, the parameters of $\Uqk$ are determined by
	\begin{align}
		\parc{N} = q^{-1}\mu^2,\quad 
		\parc{N/2} = q^{-1}\lambda\mu^{-1}, \quad
		\parc{r} = 1\quad \left( 1 \leqslant r < \frac{N}{2}\right)
	\end{align}
	\begin{align}
		\pars{N}= \frac{1 - \mu^2}{q - q^{-1}},\quad
		\pars{\frac{N}{2}} =\frac{(\lambda\mu)^{-1}-\lambda\mu^{-1}}{q-q^{-1}},\quad
		\pars{r}=0\quad \left( 0 \leqslant r < \frac{N}{2} \right)
	\end{align}
	In both cases,  the K-matrix is given by  
	\begin{align} \label{eq: RV V1 K-matrix 2}
		\mathbf{K} =& \ \id + \frac{(z-z^{-1})}{(\lambda \mu - z)}\frac{M_2}{(\lambda^{-1} + (\mu z)^{-1})},
	\end{align}
	where
	\begin{align}
		M_2 =& \ 
		\sum_{1 \leqslant i \leqslant N/2} (\lambda E_{ii} + \lambda^{-1} E_{N+1-i,N+1-i} + \mu E_{i,N+1-i} + \mu^{-1} E_{N+1-i,i}). 
	\end{align} 
	As before, we set $p_0 = \lambda \mu^{-1}$ and $p_1 = \lambda \mu$. Then, the
	formula \eqref{eq: RV V1 K-matrix 2} reads $u_r\mapsto u_{\bar{r}}$ if $r=\bar{r}$
	and 
	\begin{equation} \label{eq: explicit K}
		u_r \mapsto
		\frac{(p_0^{-1}p_1-1) + (p_1-p_0^{-1})z^{\delta(r < \bar{r})}z^{-\delta(r > \bar{r})}}{(p_1-z)(p_0^{-1} + z^{-1})} u_r + \frac{(p_0^{-1}p_1)^{\delta(r < \bar{r})}(z-z^{-1})}{(p_1-z)(p_0^{-1}+z^{-1})} u_{\bar{r}} 
	\end{equation} 
	if $r\neq\bar{r}$.\\
	
	Note that, in the case $\mu=1$, we get $p_0=p_1\eqqcolon p$ and \eqref{eq: explicit K} simplifies to 
	$u_r\mapsto u_{\bar{r}}$ if $r=\bar{r}$
	and 
	\begin{equation} \label{eq: K mat fund1}
		u_r \mapsto
		\frac{(1-p^2)z^{\delta(r < \bar{r})}z^{-\delta(r > \bar{r})}}{z - p^2z^{-1}} u_r + \frac{p(z-z^{-1})}{z-p^2z^{-1}} u_{\bar{r}} 
	\end{equation}
	if $r\neq\bar{r}$.
	
	\subsection{The boundary Schur-Weyl functor}\label{ss:AIII-functor}
	As anticipated in \ref{ex:1}, we briefly describe here one example of the 
	combinatorial model from \S\ref{ss:combinatorics} in this setting.\\ 
	
	We consider a quasi-split affine QSP subalgebra of type $\mathsf{AIII}$ with 
	a choice of parameters as described in \S\ref{ss:AIII-qsp} such that 
	\begin{enumerate}\itemsep0.25cm
		\item the twisting $\psi$ satisfying $\psi^*(\frep{})=\frep{}$ is QSP-admissible
		(cf.~\S\ref{ss:spectral-k}),
		\item $\lambda,\mu\in\bsF^{\times}$ are such that $\lambda\mu, \lambda\mu^{-1}\in
		q^{1/m}\fml{\bbC}{q^{1/m}}$ for some $m>0$.
	\end{enumerate}
	
	Set $J\coloneqq \bbZ_{\scsop{odd}}$ and consider the involution 
	$\qT$ on $J$ given by $\qT(n)=-n$. Note that $\qJT=\emptyset$.
	For any $i\in J$, we set $\qV{n}\coloneqq\frep{}$ and $\qX{n}=q^{n}$.
	
	As observed in \ref{ex:1}, the conditions \ref{cond: Q0}--\ref{cond: Q3} 
	are easily verified to hold. Moreover, by 
	\eqref{eq: RV V1 K-matrix}-\eqref{eq: explicit K}, the K-matrix has only two simple poles
	at $z=\lambda\mu^{-1}$ and $z=\lambda\mu$. Thus, the condition \ref{cond: Q4} follows from the condition 
	(2) above.\\
	
	By Definition~\ref{def:quiver}, we obtain a quiver $\Gamma$ whose nodes are 
	indexed by odd numbers, there is an edge between any 
	two consecutive numbers, and it is equipped with a natural contravariant involution 
	induced from $\qT$, \ie $\Gamma$ is an $\mathsf{A}_{\infty}$ quiver with a non-trivial 
	involution and no fixed points:
	
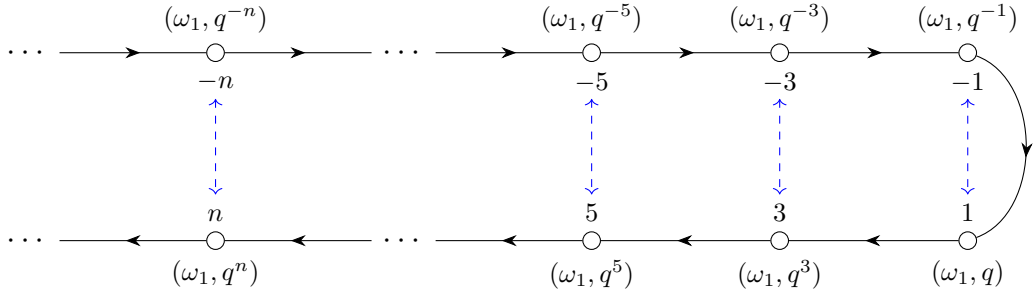
\vspace{0.5cm}
	\begin{center} 
		\begin{tikzpicture}[decoration={markings, mark= at position 0.55 with {\arrow{Stealth[length=2mm]}}}]
			
			\node (NN) at (-5,1.25) [draw, circle, fill=white, scale=0.6, label=below:{\footnotesize$-n$}, label=above:{\footnotesize$(\omega_1,q^{-n})$}] {}; 
			
			
			\node (N1) at (0,1.25) [draw, circle, fill=white, scale=0.6, label=below:{\footnotesize$-5$}, label=above:{\footnotesize$(\omega_1,q^{-5})$}] {}; 
			
			\node (N2) at (2.5,1.25) [draw, circle, fill=white, scale=0.6, label=below:{\footnotesize$-3$}, label=above:{\footnotesize$(\omega_1,q^{-3})$}] {}; 
			\node (N3) at (5,1.25) [draw, circle, fill=white, scale=0.6, label=below:{\footnotesize$-1$}, label=above:{\footnotesize$(\omega_1,q^{-1})$}] {}; 
			\node(N4) at (5,-1.25) [draw, circle, fill=white, scale=0.6, label=above:{\footnotesize$1$}, label=below:{\footnotesize$(\omega_{1},q)$}] {};
			\node (N5) at (2.5,-1.25) [draw, circle, fill=white, scale=0.6, label=above:{\footnotesize$3$}, label=below:{\footnotesize$(\omega_{1},q^3)$}] {};
			
			\node (N6) at (0,-1.25) [draw, circle, fill=white, scale=0.6, label=above:{\footnotesize$5$}, label=below:{\footnotesize$(\omega_{1},q^5)$}] {};
			
			
			\node (NX) at (-5,-1.25) [draw, circle, fill=white, scale=0.6, label=above:{\footnotesize$n$}, label=below:{\footnotesize$(\omega_{1},q^n)$}] {};
			
			\node (D0) at (-7.5,1.25) {$\cdots$};
			\draw[postaction={decorate}] (D0)--(NN);
			\node (DX) at (-7.5,-1.25) {$\cdots$};
			\draw[postaction={decorate}] (NX)--(DX);
			
			\node (D1) at (-2.5,1.25) {$\cdots$};
			\draw[postaction={decorate}] (NN)--(D1);
			\node (D2) at (-2.5,-1.25) {$\cdots$};
			\draw[postaction={decorate}] (D2)--(NX);
			
			\draw[postaction={decorate}] (D1)--(N1);
			\draw[postaction={decorate}] (N1)--(N2);
			\draw[postaction={decorate}] (N2)--(N3);
			\draw (N3) edge[postaction={decorate}, bend left=70] (N4); 
			\draw[postaction={decorate}] (N4)--(N5); 
			\draw[postaction={decorate}] (N5)--(N6); 
			\draw[postaction={decorate}] (N6)--(D2);

			\draw (N1) edge[blue, dashed, shorten >=14pt, shorten <=14pt, <->] (N6);
			\draw (N2) edge[blue, dashed, shorten >=14pt, shorten <=14pt, <->] (N5);
			\draw (N3) edge[blue, dashed, shorten >=14pt, shorten <=14pt, <->] (N4);
			\draw (NN) edge[blue, dashed, shorten >=14pt, shorten <=14pt, <->] (NX);
			
		\end{tikzpicture}
	\end{center}
	\Omit{
	\begin{center} 
		\begin{tikzpicture} \label{pic: invol quiver}
			\node (N1) at (0,0) [draw, circle, fill=white, scale=0.6, label=$-3$, label=below:{$(\frep{},q^{-3})$}] {};
			\node (N2) at (2.5,0) [draw, circle, fill=white, scale=0.6, label=$-1$, label=below:{$(\frep{},q^{-1})$}] {};
			\node(N3) at (5,0) [draw, circle, fill=white, scale=0.6, label=$1$, label=below:{$(\frep{},q)$}] {};
			\node (N4) at (7.5,0) [draw, circle, fill=white, scale=0.6, label=$3$, label=below:{$(\frep{},q^{3})$}] {};
			\draw[-stealth] (N1)--(N2);
			\draw[-stealth] (N2)--(N3);
			\draw[-stealth] (N3)--(N4);
			\node (D0) at (-2.5,0) {$\cdots$};
			\draw[-stealth] (D0)--(N1);
			\node (D5) at (10,0) {$\cdots$};
			\draw[-stealth] (N4)--(D5);
			\draw[stealth-stealth] (5,1) arc
			[
			start angle=0,
			end angle=180,
			x radius=1.25cm,
			y radius =1cm
			] ;
			\draw[stealth-stealth] (7.5,1) arc
			[
			start angle=0,
			end angle=180,
			x radius=3.75cm,
			y radius =1.5cm
			] ;
		\end{tikzpicture}
	\end{center}
}
	Finally, the framing dimension vector $\frv{}$ on $\Gamma$ is non-trivial if 
	and only if $\lambda=q^{n}\mu^{\pm 1}$ for some $n\in \bbZ_{\scsop{odd}}$. In this case, the 
	framing dimension vector is given by $\frv{(m)}=\drv{m,-n}$ for any $m\in \bbZ_{\scsop{odd}}$.
	
	\subsection{Hecke algebras and $\imath/\jmath$Schur--Weyl dualities}
	It is proved in \cite{kang-kashiwara-kim-18, kwon-lee} that the functor $ \mathsf{F}$ 
	recovers Chari-Pressley's quantum affine Schur-Weyl duality \cite{Chari-Pressley-96}
	through the BKR isomorphism between KLR algebras and affine Hecke algebras. 
	We shall prove the analogue result for the functor $\tfun$.\\
	
	More precisely, let $\Uqk\subseteq\UqLsl{N}$ be a quasi-split QSP subalgebra of 
	type $\mathsf{AIII}$ with the choice of parameters described above. 
	A Schur-Weyl duality functor between $\Uqk$ and 
	the 2-parameters affine Hecke algebra of type $\mathsf{C}$
	(cf.~\S\ref{ss:aff-Hecke-C}) was constructed in \cite{FLLLWW}, 
	relying on an explict action of the latter on 
	$\frep{z}^n\coloneqq\shrep{\frep{}}{z_1}\ten\cdots\ten\shrep{\frep{}}{z_n}$.

	\begin{lemma}\label{lem:tfun-FL3W2}
		The action of $\mathcal{H}_{\mathsf{C}_n}(p_0,p_1)$ on $\frep{z}^n$, 
		determined by the assignment 
		\[ \Phi_k \mapsto \rRMv{k,k+1}{z_k,z_{k+1}}, \quad \Phi_0 \mapsto \mathbf{K}_1(z_1), \quad X_l \mapsto z_l, \]
		coincides with the action from \cite[Prop.~2.5]{FLLLWW}. 
	\end{lemma} 
	
	\begin{proof} 
		A formula for the action of the generators $T_i$ is obtained by substituting the expressions \eqref{eq: R mat fund1}, \eqref{eq: explicit K matrix 2} and \eqref{eq: explicit K} for R- and K-matrices into the formulas \eqref{eq: int C1}--\eqref{eq: int C2}  for the intertwiners. The lemma then follows from a straightforward comparison with the formulae in \cite[\S2.4]{FLLLWW}\footnote{More precisely, it is also necessary to take into account that the formulae in \cite{FLLLWW} are given in terms of the left-coideal subalgebra $\omega(\Uqk)$.}.
	\end{proof} 
	
	Let
	\[
	\fun_{\scsop{AHC}} \colon \modv{\mathcal{H}_{\mathsf{C}_n}(p_0,p_1)} \to \modv{\Uqk}, \quad M \mapsto \frep{z}^n \otimes_{\mathcal{H}_{\mathsf{C}_n}(p_0,p_1)} M 
	\]
	be the functor defined by the bimodule $\frep{z}^n$. 
	For any $\beta \in N[J]^\theta$, let $\modv{\oklr}_0$ be the full subcategory of
	$\modv{\oklr}$ whose objects are modules with a locally nilpotent action of the 
	elements $x_i$. Also let $\modv{\mathcal{H}_{\mathsf{C}_n}(p_0,p_1)}_\beta$ 
	be the full subcategory of $\mathcal{H}_{\mathsf{C}_n}(p_0,p_1)$-modules such 
	that the action of the $X_i^{\pm 1}$ is locally nilpotent and their multiset of eigenvalues
	equals $\beta$. 
	
	\begin{theorem}\label{thm:ij-schur}
		Let $\Norm{\beta}_\theta = n$. The following diagram commute
		\begin{equation}
			\begin{tikzcd}
				\modv{\oklr\otimes_{\tbP_{\beta}} \tbPh_{\beta}} \arrow[rr,"\sim"] \arrow[ddr, dashed]& &
				\modv{\mathcal{H}_{\mathsf{C}_n}(p_0,p_1)\otimes_{\mathcal{O}} \tOhb}
				\arrow[d, "\wr"] \arrow[ddl, dashed]\\
				\modv{\oklr}_0 \arrow[dr,"\tfun_\beta"'] \arrow[u, "\wr"]& & \modv{\mathcal{H}_{\mathsf{C}_n}(p_0,p_1)}_\beta\arrow[dl, "\fun_{\scsop{AHC}}"]\\
				& \modv{\Uqk} &
			\end{tikzcd}
		\end{equation}
	\end{theorem}
	
	\begin{proof}
		The diagram is composed of three triangles.  
		The commutativity of the side ones follows immediately from the definition 
		of the completions $\tbPh_{\beta}$ and $\tOhb$ (cf.~\S\ref{ss:field-completions}). 
		The commutativity of the triangle in the middle square follows directly from the BKR-type
		isomorphism constructed in Theorem \ref{thm: bkr iso C} and Lemma~\ref{lem:tfun-FL3W2}. 
	\end{proof}
	
	In the case of a restrictable QSP subalgebra, we further recover a finite-type Schur-Weyl
	duality between the Hecke algebra of type $\sfB$ and a finite-type QSP subalgebra in $U_q\mathfrak{sl}_N$. More precisely, let $\Uqk\in\UqLsl{N}$ be a restrictable affine QSP subalgebra from \S\ref{sec: restr case}. Then, $\Uqk^{\scsop{fin}}\coloneqq
	\Uqk\cap U_q\mathfrak{sl}_N$ is the QSP subalgebra corresponding to the finite-type
	Satake diagram obtained from \eqref{eq: Satake aff diag 1} (for $N$ even)
	or \eqref{eq: Satake aff diag 2} (for $N$ odd) by removing the affine node.\\ 
	
	Consider the operators $\mathbf{R}^{\scsop{fin}} \colon \frep{}^{\otimes 2} 
	\to\frep{}^{\ten 2}$ and $\mathbf{K}^{\scsop{fin}} \colon \frep{} \to\frep{}$ on the fundamental representation given respectively by
	\begin{alignat}{3}
		\label{eq: finite R mat}
		u_r \otimes u_s \mapsto& \ \left\{ \begin{array}{ll}
			u_s \otimes u_r & \textrm{ if } r > s, \\ 
			u_s \otimes u_r + (q-q^{-1}) u_r \otimes u_s & \textrm{ if } r < s, \\ 
			q u_r \otimes u_r & \textrm{ if } r = s,
		\end{array}\right. \\
		\label{eq: finite K mat}
		u_r \mapsto& \ \left\{ \begin{array}{ll}
			u_{\bar{r}}  & \textrm{ if } r < \bar{r}, \\ 
			u_{\bar{r}} + (p-p^{-1})u_r \qquad \qquad \ & \textrm{ if } r > \bar{r}, \\
			p u_r  & \textrm{ if } r = \bar{r}. 
		\end{array}\right. 
	\end{alignat} 
	By \cite[Thms. 2.6 and 4.4]{Bao-Wang-Watanabe}, the assignment
	\[ T_0 \mapsto \mathbf{K}^{\scsop{fin}}_{1}, \qquad 
	T_k \mapsto \mathbf{R}^{\scsop{fin}}_{k,k+1} \quad (1 \leqslant k \leqslant n-1)\] 
	defines a $\bl \Uqk^{\scsop{fin}}, \mathcal{H}^{\scsop{fin}}_{\mathsf{B}_n}(p) \br$-bimodule structure on $\frep{}^{\otimes n}$, where $ \mathcal{H}^{\scsop{fin}}_{\mathsf{B}_n}(p)$
	is the finite Hecke algebra of type $\sfB_n$. Let
	\[
	\fun_{\scsop{HB}} \colon \modv{\mathcal{H}^{\scsop{fin}}_{\mathsf{B}_n}(p)} 
	\to \modv{\Uqk^{\scsop{fin}}}, \quad M \mapsto \frep{}^{\otimes n} \otimes_{\HecBf} M 
	\]
	be the functor induced by the bimodule structure on $\frep{}^{\ten n}$
	(cf.~\cite{Watanabe-20}).
	
	\begin{corollary}
		The diagram
		\begin{equation}
			\begin{tikzcd}
				\modv{\oklr}_0 \arrow[r,"\tfun_\beta"] \arrow[d, "\wr"]&  \modv{\Uqk} \arrow[d, equal]\\ \modv{\mathcal{H}_{\mathsf{C}_n}(p_0,p_1)}_\beta\arrow[r, "\fun_{\scsop{AHC}}"]
				\arrow[d, "p_0=p_1", "\operatorname{res}"'] &  \modv{\Uqk} \arrow[d, "\mu=1", "\operatorname{res}"']\\
				\modv{\mathcal{H}^{\scsop{fin}}_{\mathsf{B}_n}(p)} \arrow[r, "\fun_{\scsop{HB}}"'] &  \modv{\Uqk^{\scsop{fin}}}
			\end{tikzcd}
		\end{equation}
		where the upper vertical arrow on the left is as in Theorem~\ref{thm:ij-schur} and
		the bottom vertical arrows are given by restriction under the additional 
		assumption $p_0=p_1$ and $\mu=1$ (cf.~\S\ref{sec: restr case}), is commutative.
	\end{corollary}
	
	\begin{proof}
		Let $\tbV[{ \beta }]$ be the module defined in \S\ref{ss:boundary-SW} with respect to the
		combinatorial model described in \S\ref{ss:AIII-functor}.
		Relying on \eqref{eq: R mat fund1}, \eqref{eq: K mat fund1},  Theorem \ref{thm: bkr iso C} and \eqref{eq: finite R mat}-\eqref{eq: finite K mat}, one  shows by direct inspection that 
		\[ \tbV[{ \beta }] \cong \frep{}^{\otimes n} \otimes_{\mathcal{H}^{\scsop{fin}}_{\mathsf{B}_n}(p)} \HecB\otimes_{\mathcal{O}} \tOhb \]
		as $(\Uqk^{\scsop{fin}},  \HecB\otimes_{\mathcal{O}} \tOhb)$-bimodules. Hence, for any $M$ in $\modv{\oklr}_0$, we have 
		\[ \tfun_\beta(M) = \tbV[{ \beta }] \otimes_{\HecB\otimes_{\mathcal{O}} \tOhb} M \cong \frep{}^{\otimes n} \otimes_{\mathcal{H}^{\scsop{fin}}_{\mathsf{B}_n}(p)} M = \fun_{\scsop{HB}}(M). \] 
	\end{proof}
	
	\begin{remark}
		Relying on the explicit formulae for the K-matrix provided in \cite[\S 9.3]{regelskis-vlaar-16} and \cite{Shen-Wang}, a similar computation shows that
		the same result holds for arbitrary affine QSP subalgebras of type  $\mathsf{AIII}$.
		It is also expected that one can recover the Schur-Weyl functor with the 2-parameter affine 
		Hecke algebra of type $\sfB$ defined in \cite{chen-guay-ma-14}. This is, however, harder to verify since \cite{chen-guay-ma-14} relies on the FRT presentation of QSP subalgebras. 
		\rmkend
	\end{remark}

	\bibliographystyle{myamsalpha}

\begin{thebibliography}{KKKO18}
		
		\bibitem[AR95]{abad-rios-95}
		J. Abad, M. Rios, \emph{Non-diagonal solutions to reflection equations in
			{${\rm su}(n)$} spin chains}, Phys. Lett. B \textbf{352} (1995), no.~1-2,
		92--95. \MR{1344428}
		
		\bibitem[ATL19]{appel-toledano-19b}
		A. Appel, V. Toledano~Laredo, \emph{Coxeter categories and quantum groups},
		Selecta Math. (N.S.) \textbf{25} (2019), no.~3, Paper No. 44, 97. \MR{3984102}
		
		\bibitem[AV22a]{appel-vlaar-20}
		A. Appel, B. Vlaar, \emph{Universal K-matrices for quantum {K}ac--{M}oody algebras}, Represent. Theory {\bf 26} (2022), 764-824.
		
		\bibitem[AV22b]{appel-vlaar-22}
		A. Appel, B. Vlaar, \emph{Trigonometric K-matrices for finite-dimensional representations of
			quantum affine algebras}, \href{https://arxiv.org/abs/2203.16503}{\sf
			arXiv:2203.16503}, 2022.

		
		\bibitem[BK19]{balagovic-kolb-19}
		M. Balagovi\'{c}, S. Kolb, \emph{Universal {K}-matrix for quantum symmetric
			pairs}, J. Reine Angew. Math. \textbf{747} (2019), 299--353. \MR{3905136}
		
		\bibitem[BW18a]{bao-wang-18}
		H. Bao, W. Wang, \emph{Canonical bases arising from quantum symmetric pairs},
		Invent. Math. \textbf{213} (2018), no.~3, 1099--1177. \MR{3842062}
		
		\bibitem[BW18b]{bao-wang-18b}
		H. Bao, W. Wang, \emph{A new approach to {K}azhdan-{L}usztig theory of type {$B$} via
			quantum symmetric pairs}, Ast\'{e}risque (2018), no.~402, vii+134.
		\MR{3864017}
		
		\bibitem[BWW18]{Bao-Wang-Watanabe}
		H. Bao, W. Wang, H. Watanabe, \emph{Multiparameter quantum {S}chur duality of
			type {B}}, Proc. Amer. Math. Soc. \textbf{146} (2018), no.~8, 3203--3216.
		\MR{3803649}
		
				\bibitem[BK09]{Brundan-Kleshchev}
		J. Brundan, A. Kleshchev, \emph{Blocks of cyclotomic {H}ecke algebras and
			{K}hovanov-{L}auda algebras}, Invent. Math. \textbf{178} (2009), no.~3,
		451--484. \MR{2551762}
		

		
				\bibitem[Cha02]{chari-02}
		V. Chari, \emph{Braid group actions and tensor products}, Int. Math. Res. Not.
		(2002), no.~7, 357--382. \MR{1883181}
		
		
				\bibitem[CP95]{chari-pressley}
		V. Chari, A. Pressley, \emph{A guide to quantum groups}, Cambridge University
		Press, Cambridge, 1995, Corrected reprint of the 1994 original.
		
		\bibitem[CP96]{Chari-Pressley-96}
		V. Chari, A. Pressley, \emph{Quantum affine algebras and affine {H}ecke
			algebras}, Pacific J. Math. \textbf{174} (1996), no.~2, 295--326.
		\MR{1405590}
		
		\bibitem[CGM14]{chen-guay-ma-14}
H. Chen, N. Guay, X. Ma, \emph{Twisted {Y}angians, twisted quantum loop
	algebras and affine {H}ecke algebras of type {$BC$}}, Trans. Amer. Math. Soc.
\textbf{366} (2014), no.~5, 2517--2574. \MR{3165646}
		
		\bibitem[Che84]{cherednik-84}
		I. Cherednik, \emph{{F}actorizing particles on a half line, and root
			systems}, Teoret. Mat. Fiz. \textbf{61} (1984), no.~1, 35--44. \MR{774205}
		
		\bibitem[Che92]{cherednik-92}
		I. Cherednik, \emph{Quantum {K}nizhnik-{Z}amolodchikov equations and affine
			root systems}, Comm. Math. Phys. \textbf{150} (1992), no.~1, 109--136.
		\MR{1188499}
		

		
		
		
		\bibitem[Dri87]{drinfeld-quantum-groups-87}
		V. Drinfeld, \emph{Quantum groups}, Proceedings of the {I}nternational
		{C}ongress of {M}athematicians, {V}ol. 1, 2 ({B}erkeley, {C}alif., 1986),
		Amer. Math. Soc., Providence, RI, 1987, pp.~798--820. \MR{934283}
		
		
				\bibitem[ES18]{ehrig-stroppel}
		M. Ehrig, C. Stroppel, \emph{Nazarov-{W}enzl algebras, coideal subalgebras and
			categorified skew {H}owe duality}, Adv. Math. \textbf{331} (2018), 58--142.
		\MR{3804673}
		
		\bibitem[EK06]{Enomoto-Kashiwara-06}
		N. Enomoto, M. Kashiwara, \emph{Symmetric crystals and affine {H}ecke algebras
			of type {B}}, Proc. Japan Acad. Ser. A Math. Sci. \textbf{82} (2006), no.~8,
		131--136.
		
		\bibitem[EK08]{Enomoto-Kashiwara-08}
		N. Enomoto, M. Kashiwara, \emph{Symmetric crystals for {$\mathfrak{gl}_\infty$}}, Publ. Res.
		Inst. Math. Sci. \textbf{44} (2008), no.~3, 837--891.
		
		
				\bibitem[FLL{\etalchar{+}}20]{FLLLWW}
		Z. Fan, C.-J. Lai, Y. Li, L. Luo, W. Wang, H. Watanabe, \emph{Quantum {S}chur
			duality of affine type {C} with three parameters}, Math. Res. Lett.
		\textbf{27} (2020), no.~1, 79--114. \MR{4088809}
		
		\bibitem[FHR22]{frenkel-hernandez-reshetikhin-21}
		E. Frenkel, D. Hernandez, N.~Y. Reshetikhin, \emph{Folded quantum integrable
			models and deformed $\mathcal{W}$-algebras}, Lett. Math. Phys. {\bf 112}, 80 (2022).
		

		
		
		\bibitem[Fuj20]{fujita-20}
		R. Fujita, \emph{Geometric realization of {D}ynkin quiver type quantum affine {S}chur-{W}eyl duality}, 
		Int. Math. Res. Not. IMRN \textbf{22} (2020), 8353--8386. 

		\bibitem[Fuj22]{fujita-22}
		R. Fujita, \emph{Affine highest weight categories and quantum affine
			{S}chur-{W}eyl duality of {D}ynkin quiver types}, Represent. Theory
		\textbf{26} (2022), 211--263. \MR{4396615}
		
		\bibitem[Har01]{haring-old}
		R. H\"{a}ring-Oldenburg, \emph{Actions of tensor categories, cylinder braids and their
              {K}auffman polynomial}, Topology Appl. \textbf{112} (2001), no.~3, 297--314.

%
		
		\bibitem[HL10]{HL-cluster}
		D. Hernandez, B. Leclerc, \emph{Cluster algebras and quantum affine algebras},
		Duke Math. J. \textbf{154} (2010), no.~2, 265--341. \MR{2682185}
		
		\bibitem[HL15]{HL-Hall}
		D. Hernandez, B. Leclerc, \emph{Quantum {G}rothendieck rings and derived {H}all algebras}, J.
		Reine Angew. Math. \textbf{701} (2015), 77--126. \MR{3331727}
		
		\bibitem[IT10]{Ito-Terwilliger-10}
		T. Ito, P. Terwilliger, \emph{The augmented tridiagonal algebra}, Kyushu J.
		Math. \textbf{64} (2010), no.~1, 81--144. \MR{2662661}
		
		\bibitem[Jim85]{jimbo-85}
		M. Jimbo, \emph{A {$q$}-difference analogue of {$U(\mathfrak{g})$} and the
			{Y}ang-{B}axter equation}, Lett. Math. Phys. \textbf{10} (1985), no.~1,
		63--69. \MR{797001}
		
		\bibitem[Kac90]{kac-90}
		V. Kac, \emph{Infinite-dimensional {L}ie algebras}, third ed., Cambridge
		University Press, Cambridge, 1990.
		
				\bibitem[KW92]{kac-wang-92}
		V. Kac, S. Wang, \emph{On automorphisms of {K}ac-{M}oody algebras and
			groups}, Adv. Math. \textbf{92} (1992), no.~2, 129--195. \MR{1155464}
	
		
		\bibitem[KKK15]{kang-kashiwara-kim-15}
		S.-J. Kang, M. Kashiwara, M. Kim, \emph{Symmetric quiver {H}ecke algebras and
			{$R$}-matrices of quantum affine algebras, {II}}, Duke Math. J. \textbf{164}
		(2015), no.~8, 1549--1602. \MR{3352041}
		
		\bibitem[KKK18]{kang-kashiwara-kim-18}
		S.-J. Kang, M. Kashiwara, M. Kim, \emph{Symmetric quiver {H}ecke algebras and {R}-matrices of quantum
			affine algebras}, Invent. Math. \textbf{211} (2018), no.~2, 591--685.
		\MR{3748315}
		
		\bibitem[KKKO15]{kang-kashiwara-kim-oh-15}
		S.-J. Kang, M. Kashiwara, M. Kim, S.-J. Oh, \emph{Symmetric quiver {H}ecke
			algebras and {$R$}-matrices of quantum affine algebras {III}}, Proc. Lond.
		Math. Soc. (3) \textbf{111} (2015), no.~2, 420--444. \MR{3384517}
		
		\bibitem[KKKO16]{kang-kashiwara-kim-oh-16}
		S.-J. Kang, M. Kashiwara, M. Kim, S.-J. Oh, \emph{Symmetric quiver {H}ecke
			algebras and {$R$}-matrices of quantum affine algebras {IV}}, Selecta Math.
		(N.S.) \textbf{22} (2016), no.~4, 1987--2015. \MR{3573951}
		
		\bibitem[KKKO18]{kang-kashiwara-kim-oh-18}
		S.-J. Kang, M. Kashiwara, M. Kim, S.-j. Oh, \emph{Monoidal categorification of
			cluster algebras}, J. Amer. Math. Soc. \textbf{31} (2018), no.~2, 349--426.
		\MR{3758148}
	
		
				\bibitem[Kas02]{kashiwara-02}
		M. Kashiwara, \emph{On level-zero representations of quantized affine
			algebras}, Duke Math. J. \textbf{112} (2002), no.~1, 117--175. \MR{1890649}
		


		
		\bibitem[Kat09]{Kato-09}
		S. Kato, \emph{An exotic {D}eligne-{L}anglands correspondence for symplectic
			groups}, Duke Math. J. \textbf{148} (2009), no.~2, 305--371. \MR{2524498}
		
				\bibitem[KS95]{kazhdan-soibelman-95}
		D. Kazhdan, Y. Soibelman, \emph{Representations of quantum affine algebras},
		Selecta Math. (N.S.) \textbf{1} (1995), no.~3, 537--595. \MR{1366624}
		

		
		\bibitem[KL09]{Khovanov-Lauda-1}
		M. Khovanov, A.~D. Lauda, \emph{A diagrammatic approach to categorification of
			quantum groups. {I}}, Represent. Theory \textbf{13} (2009), 309--347.
		\MR{2525917}
		
				\bibitem[KR87]{kirillov-reshetikhin-87}
		A. Kirillov, N. Reshetikhin, 
		\emph{Representations of Yangians and multiplicities of the inclusion of the irreducible components of the tensor product of representations of simple Lie algebras.} (Russian); 
		translated from Zap. Nauchn. Sem. Leningrad. Otdel. Mat. Inst. Steklov. (LOMI) {\bf 160} (1987), Anal. Teor. Chisel i Teor. Funktsii. 8, 211--221, 301. J. Soviet Math. {\bf 52} (1990), no.~3, 3156--3164.
		
		\bibitem[KR11]{Kleshchev-Ram-11}
		A. Kleshchev, A. Ram, \emph{Representations of {K}hovanov-{L}auda-{R}ouquier
			algebras and combinatorics of {L}yndon words}, Math. Ann. \textbf{349}
		(2011), no.~4, 943--975. \MR{2777040}
		
				\bibitem[Kol14]{kolb-14}
		S. Kolb, \emph{Quantum symmetric {K}ac-{M}oody pairs}, Adv. Math. \textbf{267}
		(2014), 395--469. \MR{3269184}
		
		\bibitem[KOW22]{kusano-okado-watanabe-22}
		H. Kusano, M. Okado, H. Watanabe, \emph{Kirillov-Reshetikhin modules and quantum K-matrices}, \arxiv{2209.10325} (2022).
		
		\bibitem[KL21]{kwon-lee}
		J.-H. Kwon, S.-M. Lee, \emph{{Super Duality for Quantum Affine Algebras of Type
				A}}, International Mathematics Research Notices (2021), rnab230.
		

		
		

		

		
		\bibitem[Let02]{letzter-02}
		G. Letzter, \emph{Coideal subalgebras and quantum symmetric pairs}, New
		directions in {H}opf algebras, Math. Sci. Res. Inst. Publ., vol.~43,
		Cambridge Univ. Press, Cambridge, 2002, pp.~117--165. \MR{1913438}
		
		\bibitem[Li19]{Li}
		Y. Li, \emph{Quiver varieties and symmetric pairs}, Represent. Theory
		\textbf{23} (2019), 1--56. \MR{3900699}
		
				\bibitem[LW21]{lu-wang-drin}
		M. Lu, W. Wang, \emph{A {D}rinfeld type presentation of affine
			{$\imath$}quantum groups {I}: {S}plit {ADE} type}, Adv. Math. \textbf{393}
		(2021), Paper No. 108111, 46. \MR{4340233}
		
		\bibitem[Lus10]{lusztig-94}
		G. Lusztig, \emph{Introduction to quantum groups}, Modern Birkh\"{a}user
		Classics, Birkh\"{a}user/Springer, New York, 2010, Reprint of the 1994
		edition. \MR{2759715}
		

		
%

				\bibitem[MS21]{Maksimau-Stroppel}
		R. Maksimau, C. Stroppel, \emph{Higher level affine {S}chur and {H}ecke
			algebras}, J. Pure Appl. Algebra \textbf{225} (2021), no.~8, Paper No.
		106442, 44. \MR{4182993}
		
		\bibitem[MS19]{Miemietz-Stroppel}
		V. Miemietz, C. Stroppel, \emph{Affine quiver {S}chur algebras and {$p$}-adic
			{$GL_n$}}, Selecta Math. (N.S.) \textbf{25} (2019), no.~2, Paper No. 32, 66.
		\MR{3948934}
		
		\bibitem[Nak04]{Nak3}
		H. Nakajima, \emph{Quiver varieties and $t$-analogs of $q$-characters of quantum affine algebras}, Annals of Math. \textbf{160} (2004),
1057--1097.

		
		\bibitem[Nao21]{naoi}
		K. Naoi, \emph{Equivalence between module categories over quiver {H}ecke
			algebras and {H}ernandez-{L}eclerc's categories in general types}, Adv. Math.
		\textbf{389} (2021), Paper No. 107916, 47. \MR{4290135}
		
		\bibitem[PdR21]{PAR}
		L. Poulain~d'Andecy, S. Rostam, \emph{Morita equivalences for cyclotomic
			{H}ecke algebras of types {B} and {D}}, Bull. Soc. Math. France \textbf{149}
		(2021), no.~1, 179--233. \MR{4250039}
		
		\bibitem[PdW20]{PAW-B}
		L. Poulain~d'Andecy, R. Walker, \emph{Affine {H}ecke algebras and
			generalizations of quiver {H}ecke algebras of type {$B$}}, Proc. Edinb. Math.
		Soc. (2) \textbf{63} (2020), no.~2, 531--578. \MR{4085039}
		
		\bibitem[Prz19]{Przez-coha}
		T. Prze\'{z}dziecki, \emph{Quiver {S}chur algebras and cohomological {H}all
			algebras}, \arxiv{1907.03679} (2019).
		
		\bibitem[Prz21]{Przez-oklr}
		T. Prze\'{z}dziecki, \emph{Representations of orientifold {K}hovanov--{L}auda--{R}ouquier
			algebras and the {E}nomoto--{K}ashiwara algebra}, Pac. J. Math. \textbf{322} (2023), no.~2, 				407--441.

		
		\bibitem[RV16]{regelskis-vlaar-16}
		V. Regelskis, B. Vlaar, \emph{Reflection matrices, coideal subalgebras and
			generalized {S}atake diagrams of affine type}, 2016.
		
		\bibitem[RV18]{regelskis-vlaar-18}
		V. Regelskis, B. Vlaar, \emph{Solutions of the {$U_q(\widehat{\mathfrak{sl}}_N)$} reflection
			equations}, J. Phys. A \textbf{51} (2018), no.~34, 345204, 41. \MR{3835331}
		
		\bibitem[RV20]{regelskis-vlaar-20}
		V. Regelskis, B. Vlaar, \emph{Quasitriangular coideal subalgebras of {$U_q(\mathfrak g)$} in
			terms of generalized {S}atake diagrams}, Bull. Lond. Math. Soc. \textbf{52}
		(2020), no.~4, 693--715. \MR{4171396}
		
		\bibitem[RV22]{regelskis-vlaar-21}
		V. Regelskis, B. Vlaar, \emph{Pseudo-symmetric pairs for {K}ac-{M}oody algebras},
	Hypergeometry, integrability and Lie theory, 155--203, Contemp. Math., {\bf 780}, Amer. Math. Soc., [Providence], RI, 2022.
	
			\bibitem[Rou08]{Rouquier-2KM}
	R. Rouquier, \emph{2--{K}ac--{M}oody algebras}, \arxiv{0812.5023} (2008).
		

%
		
		
		\bibitem[SW21]{Shen-Wang}
		Y. Shen, W. Wang, \emph{$\imath${S}chur duality and {K}azhdan--{L}usztig basis
			expanded}, \arxiv{2108.00630}, 2021.
			
				\bibitem[Skl88]{sklyanin-88}
		E.~K. Sklyanin, \emph{Boundary conditions for integrable quantum systems},
		Journal of Physics A: Mathematical and General \textbf{21} (1988), no.~10,
		2375.

		\bibitem[VV02]{VVper}
		M. Varagnolo, E. Vasserot, \emph{Perverse sheaves and quantum Grothendieck rings}, in Studies in memory of Issai
Schur, Prog. Math. \textbf{210}, Birkh\"{a}user, 2002, 345--365.

		
		\bibitem[VV11]{VV-HecB}
		M. Varagnolo, E. Vasserot, \emph{Canonical bases and affine {H}ecke algebras of
			type {B}}, Invent. Math. \textbf{185} (2011), no.~3, 593--693. \MR{2827096}
		
		\bibitem[Wat20]{Watanabe-20}
		H. Watanabe, \emph{Crystal basis theory for a quantum symmetric pair
			{$(\mathrm{U}, \mathrm{U}^J)$}}, Int. Math. Res. Not. IMRN (2020), no.~22,
		8292--8352. \MR{4216690}
		
	\end{thebibliography}
	\newcommand{\etalchar}[1]{$^{#1}$}

\end{document}